\documentclass[10pt]{article}
\usepackage[utf8]{inputenc}   
\usepackage[left=2.2cm,right=2.2cm,top=2.0cm,bottom=2.2cm]{geometry}
\usepackage[T1]{fontenc}                
\usepackage[english]{babel}              
\usepackage{graphicx}
\usepackage{amsmath,amsfonts,amssymb,amsthm}  
\usepackage{xspace}
\usepackage{geometry} 
\usepackage[overload]{empheq}
\usepackage{dsfont} 
\usepackage{stmaryrd}
\usepackage{tikz} 
\usepackage{mathrsfs}
\usepackage{enumerate}
\usepackage{subfig}
\usepackage{float}
\usepackage{titletoc}

\usepackage{hyperref}
\hypersetup{
colorlinks = true,
urlcolor = blue,
linkcolor = blue,
citecolor = blue,
}

\newtheorem{theorem}{Theorem}
\newtheorem{definition}{Definition} 
\newtheorem{proposition}{Proposition}
\newtheorem{lemma}{Lemma}

\newtheorem{assump:model}{assump:model}

\usepackage{amsmath,amssymb,euscript ,yfonts,psfrag,latexsym,dsfont,graphicx,bbm,color,amstext,wasysym,epsfig,subfig,parskip,textcomp}

\usepackage{natbib}
\bibpunct{(}{)}{;}{a}{,}{,}
\usepackage{hyperref}
\hypersetup{
colorlinks = true,
urlcolor = blue,
linkcolor = blue,
citecolor = blue,
}

\newtheorem{assumption}{Assumption}
\DeclareMathOperator*{\argmax}{arg\,max}

\DeclareMathOperator*{\argmin}{arg\,min}

\title{Generative model for optimal  density estimation\\ on unknown manifold}

\usepackage{authblk}
\author{Arthur St\'ephanovitch}
\affil{\small D\'epartement de Math\'ematiques et Applications\\
       Ecole Normale Sup\'erieure, Université PSL, CNRS\\
       F-75005 Paris, France,\\ \texttt{stephanovitch@dma.ens.fr}}
\date{}
\begin{document}

\maketitle
\begin{abstract}
We propose a generative model that achieves minimax-optimal convergence rates for estimating probability distributions supported on unknown low-dimensional manifolds. Building on Fefferman's solution to the geometric Whitney problem, our estimator is itself supported on a submanifold that matches the regularity of the data's support. This geometric adaptation enables the estimator to be simultaneously minimax-optimal for all \( \gamma \)-Hölder Integral Probability Metrics (IPMs) with \( \gamma \geq 1 \). We validate our approach through experiments on synthetic and real datasets, demonstrating competitive or superior performance compared to Wasserstein GAN and score-based generative models. 
\end{abstract}

\tableofcontents
\section{Introduction}

\subsection{From the geometric Whitney problem to statistical estimation}

The problem of reconstructing a manifold from discrete data has a rich history, connecting geometry, analysis, and learning theory \citep{lassas2001determining,aamari2017nonasymptotic,boissonnat2018delaunay}. In particular, the recent work of~\cite{fefferman2015reconstruction} formulated and solved the \emph{geometric Whitney problem}, providing conditions under which a point cloud can be approximated by a smooth submanifold with controlled geometry. Their work gave a constructive solution to manifold interpolation problems, leading to advances in both theory and algorithms.

In this paper, we push this geometric program further by addressing the statistical estimation of \emph{distributions} supported on unknown manifolds. Our goal is not only to reconstruct the underlying geometry but also to estimate a smooth probability measure on the manifold, achieving optimal statistical guarantees. Specifically, given an iid sample of a probability measure having a density with respect to an unknown manifold, we propose the first tractable estimator that:
\begin{itemize}
    \item (i) constructs a submanifold approximating the support of the data with the same regularity,
    \item (ii) defines a probability measure admitting a smooth density with respect to the volume measure of the estimated submanifold,
    \item (iii) achieves minimax-optimal convergence rates under Hölder Integral Probability Metrics \citep{IPMsMuller} for all Hölder indices $\gamma \geq 1$ simultaneously.
    \item (iv) enables efficient sampling from the estimated distribution.
\end{itemize}

Our method combines manifold reconstruction techniques inspired by~\cite{fefferman2015reconstruction} with tools from adversarial estimation and optimal transport, particularly Wasserstein GANs~\citep{arjovsky2017wasserstein}. By leveraging wavelet decompositions adapted to Besov regularity, we capture both the (unknown) local geometry of the manifold and the (known) smoothness of the density in a unified framework. 

Finally, we illustrate the effectiveness of our estimator through numerical experiments on synthetic and real datasets, demonstrating competitive or superior performance compared to WGANs and SGMs, especially in settings where the data exhibits nontrivial manifold structure.

\subsection{Related works}
\subsubsection{The geometric Whitney problem}

The \emph{geometric Whitney problem} asks: given a metric space \( (X,d_X) \), can we construct a smooth \( d \)-dimensional submanifold \( M \) that approximates \( X \) with precise control over the manifold’s geometry (curvature, injectivity radius) and the distance between \( M \) and \( X \)?

This problem is a geometric generalization of the classical Whitney extension problem, where the goal was to smoothly extend functions defined on arbitrary subsets of \( \mathbb{R}^n \).
In~\cite{fefferman2015reconstruction}, the authors provided a major breakthrough by:
\begin{itemize}
    \item Giving  conditions under which a metric space \( (X, d_X) \) can be approximated by a Riemannian manifold \( (M, g) \) of bounded geometry, both in Gromov--Hausdorff and quasi-isometric senses.
    \item Providing explicit reconstruction algorithms when $X$ is finite for building manifolds with controlled sectional curvature and injectivity radius.
\end{itemize}

Their work builds on classical results by  \citep{whitney1992functions,glaeser1958etude,brudnyi1997whitney,bierstone2003diff}, while pushing the theory into the global geometric setting.

In this paper, we take a different direction: instead of focusing on geometric reconstruction, we aim to push \cite{fefferman2015reconstruction}'s construction toward statistical estimation and generative modeling. Our goal is to build an estimator that not only reconstructs a smooth manifold approximating the data but also defines a probability measure with a smooth density supported on this manifold. The objective is to obtain a tractable estimator that achieves minimax-optimal convergence rates. In this way, we bridge the gap between geometric reconstruction and modern generative modeling, extending the geometric Whitney program to a statistical setting.

\subsubsection{Minimax estimators of densities on unknown manifold}
Estimating a probability distribution from samples is a fundamental problem in statistics and machine learning. Recent advances in generative modeling, particularly Generative Adversarial Networks (GANs)  \cite{goodfellow2014generative}  and Score-based Generative Models (SGMs) \cite{song2020score}, have demonstrated remarkable performance in high-dimensional settings. However, these methods often struggle when the data lie on a complex, low-dimensional manifold, as they do not produce estimators that preserve the same regularity structure as the underlying distribution. In particular, GANs face difficulties when generating data with complex topology, since their push-forward formulation can constrain the topology of the output distribution ~\citep{tanielian2020learning,stephanovitch2023wasserstein}.

SGMs also encounter limitations in this setting. When the data lie on or near a low-dimensional manifold, the score function becomes ill-behaved and diverges. This phenomenon complicates estimation and can lead to poor sample quality. Recent studies have highlighted these challenges, emphasizing the importance of careful architectural design and noise scheduling to accurate score estimation in geometrically complex settings~\citep{chen2023score,li2024good}.

Let us review the recent literature on minimax density estimation over unknown manifolds.

\begin{itemize}
    \item In~\cite{divol2022measure}, the author proposes the first minimax estimator for the Wassertein distance of densities supported on unknown manifolds. The estimator combines a local polynomial estimator from~\cite{aamari2017nonasymptotic} with a kernel density estimator. However, the local polynomial step is computationally intensive, making the overall procedure impractical in high-dimensional settings. Moreover, the estimator is not generative, as it does not provide a mechanism for sampling from the estimated distribution.

    \item In~\cite{tang2023minimax}, the authors construct an estimator that achieves the minimax rate for the Hölder Integral Probability Metric (IPM) $d_{\mathcal{H}^\gamma_1}$ for a fixed $\gamma > 0$. In particular, they highlight that when $\gamma$ is small, the distance $d_{\mathcal{H}^\gamma_1}$ becomes highly sensitive to misalignments between the supports of the measures, as opposed to discrepancies between the densities themselves. Their estimator involves manifold reconstruction, optimizing over the entire Hölder class $\mathcal{H}^\beta_1$, computing wavelet expansions of each discriminator $D \in \mathcal{H}^\gamma_1$, and evaluating all derivatives up to order $\gamma$. Despite its theoretical guarantees, the estimator is computationally infeasible in practice due to the complexity of these operations.

    \item Building on the Wasserstein GAN framework,~\cite{stephanovitch2023wasserstein} introduce the first tractable and fully generative estimator. However, this method is restricted to learning densities supported on manifolds with the topology of a $d$-dimensional torus, limiting its applicability to more realistic manifold structures.

    \item The recent works of \cite{azangulov2024convergence} and \cite{pmlr-v238-tang24a} establish that SGMs attain the minimax rate for Holder IPMs and the Wasserstein distance respectively. These results are particularly important given that SGMs have become the state-of-the-art in generative modeling. However, the practical implementation of the minimax estimator remains challenging: the score estimation procedure involves training a total of $O(\log(n))$ separate neural networks, leading to significant computational overhead. In addition, sampling from the model involves numerically solving a stochastic differential equation for each generated sample, further increasing the computational cost.
\end{itemize}

In this paper, our objective is to construct a tractable and fully generative estimator that is specifically adapted to the manifold structure of the data. This geometric adaptation enables the estimator to achieve minimax optimality simultaneously for all $\gamma$-Hölder IPMs with $\gamma \geq 1$. We propose an improvement over the classical Wasserstein GAN methodology, extending it to estimate probability measures supported on manifolds with arbitrary topology. Our approach builds upon this line of work by combining geometric reconstruction techniques with adversarial training to produce an estimator that is both computationally efficient and statistically optimal.

\section{Setting and main result}
In this section, we introduce the set of assumptions imposed on the data distribution and present the main result of the paper. We begin by defining the notion of regularity that we will work with.

\subsection{Hölder spaces, IPMs and notations}
\subsubsection{Hölder spaces}
Let us properly define the Hölder spaces. For $\eta >0$, $\mathcal{X},\mathcal{Y}$ two open subsets of Euclidean spaces and $f=(f_1,...,f_p)\in C^{\lfloor \eta \rfloor}(\mathcal{X},\mathcal{Y})$ the set of $\lfloor \eta \rfloor:=\max \{k\in \mathbb{N}_0 |\ k\leq \eta\}$  times differentiable functions (almost everywhere), denote $\partial^\nu f_i := \frac{\partial^{|\nu|}f_i}{\partial x_1^{\nu_1}...\partial x_d^{\nu_d}}$ the partial differential operator  for any multi-index $\nu = (\nu_1,...,\nu_d)\in \mathbb{N}_0^d$ with $|\nu|:=\nu_1+...+\nu_d\leq \lfloor \eta \rfloor$. Define
$$
\|f_i\|_{\eta-\lfloor \eta \rfloor}:=\sup \limits_{x\neq y}\ \frac{f_i(x)-f_i(y)}{\|x-y\|^{\eta - \lfloor \eta \rfloor}} 
$$
and let 
\begin{align*}
\mathcal{H}^{\eta}_K(\mathcal{X},\mathcal{Y}):=\Big\{& f  \in C^{\lfloor \eta \rfloor}(\mathcal{X},\mathcal{Y}) \ \big| \max \limits_{i} \sum \limits_{0\leq |\nu|\leq \lfloor \eta \rfloor} \|\partial^\nu f_i\|_{L^\infty(\mathcal{X},\mathcal{Y})} + \sum \limits_{|\nu|  = \lfloor \eta \rfloor} \|\partial^\nu f_i\|_{\eta-\lfloor \eta \rfloor}   \leq K\Big\}
\end{align*}
denote the ball of radius $K>0$ of the Hölder space $\mathcal{H}^{\eta}(\mathcal{X},\mathcal{Y})$, the set of functions $f:\mathcal{X}\rightarrow \mathcal{Y}$ of regularity $\eta$. Hölder smoothness will be used to characterize both the regularity of the data distribution and the class of metrics employed to compare probability measures. We also define the homogeneous Hölder space as
\begin{align*}
\dot{\mathcal{H}}^{\eta}_K(\mathcal{X},\mathcal{Y}):=\Big\{& f  \in C^{\lfloor \eta \rfloor}(\mathcal{X},\mathcal{Y}) \ \big| \max \limits_{i} \sum \limits_{1\leq |\nu|\leq \lfloor \eta \rfloor} \|\partial^\nu f_i\|_{L^\infty(\mathcal{X},\mathcal{Y})} + \sum \limits_{|\nu|  = \lfloor \eta \rfloor} \|\partial^\nu f_i\|_{\eta-\lfloor \eta \rfloor}   \leq K\Big\},
\end{align*}
which does not constrain the infinite norm of the functions.

\subsubsection{Integral Probability metrics}
In the context of generative modeling, many methods rely on Integral Probability Metrics (IPMs)~\citep{IPMsMuller} to assess the accuracy of the estimated distribution.
 An IPM is defined by selecting a class of functions $\mathcal{F}$ and measuring the distance between two probability measures $\mu$ and $\mu^\star$ as
\begin{equation}\label{eq:genralIPM}
    d_{\mathcal{F}}(\mu,\mu^\star) := \sup_{f \in \mathcal{F}} \left| \int f(x)\, d\mu(x) - \int f(x)\, d\mu^\star(x) \right|.
\end{equation}
Notable examples of IPMs include the total variation distance~\citep{verdu2014total} (where $\mathcal{F}$ is the class of functions taking values in $[-1,1]$), the Wasserstein distance~\citep{villani2009optimal} (where $\mathcal{F}$ is the class of 1-Lipschitz functions), and the Maximum Mean Discrepancy~\citep{smola2006maximum} (where $\mathcal{F}$ is the unit ball of a reproducing kernel Hilbert space). In this paper, we use the Hölder IPMs, which have been widely used in generative modeling~\citep{arjovsky2017wasserstein, chakraborty2024statistical} and geometric learning~\citep{tang2023minimax}.

\subsubsection{Additional notations}
We write $\langle \cdot, \cdot \rangle$ the dot product on $\mathbb{R}^p$, $\|x\|$ the Euclidean norm of a vector $x$. For $a,b\in \mathbb{R}$, $a\wedge b$ and $a\vee b$ denote the minimum and maximum value between $a$ and $b$ respectively. We write $\text{Lip}_1$ for the set of 1-Lipschitz functions. The support of a function $f$ is denoted by $supp(f)$. We  denote by $\text{Id}$ the identity application from a Euclidean space to itself.

For any map $f:\mathbb{R}^k\rightarrow \mathbb{R}^l$ we denote by $\nabla f$ the differential of $f$ and by $\|\nabla f(x)\|$ its operator norm at the point $x$. We denote $(\nabla f(x))^\top$ its transpose matrix and if $k\leq l$, $\lambda_{\min}((\nabla f(x) )^\top \nabla f(x))$ corresponds to the smallest eigenvalue of the matrix $(\nabla f(x) )^\top \nabla f(x)$. We write $\int_{\mathcal{M}}f(x)d\lambda_{\mathcal{M}}(x)$ for the integral of a function $f:\mathcal{M}\rightarrow \mathbb{R}$ with respect to the volume measure on $\mathcal{M}$ i.e. the $d$-dimensional Hausdorff measure. The Hausdorff distance between two submanifold $\mathcal{M},\mathcal{M}^\star$ is denoted by $\mathbb{H}(\mathcal{M},\mathcal{M}^\star)$. For a map $T:\mathbb{R}^d\rightarrow \mathbb{R}^d$, we write $T_{\# }\mu$ for the push forward of the measure $\mu$ by $T$.

The $k$-dimensional Lebesgue measure is denoted $\lambda^k$. For $\alpha_1,...\alpha_m\in \mathbb{R}_{>0}$, we write $ \textbf{Mult}(\{\alpha_1,...,\alpha_m\})$ for the multinomial law (or generalized Bernoulli) of parameters $(\alpha_i)_i$ i.e., if $\theta\sim \textbf{Mult}(\{\alpha_1,...,\alpha_m\})$ then for $i\in \{1,...,m\}$, $\theta=i$ with probability $\alpha_i/\sum_j \alpha_j$. The uniform multinomial $\textbf{Mult}(\{1/m,...,1/m\})$  is denoted $\mathcal{U}_m$. For any measurable set $\mathcal{X}\subset \mathbb{R}^d$ such that $\lambda^d(\mathcal{X})>0$, we write $\mathcal{U}(\mathcal{X})$ for the law of the uniform random variable on $\mathcal{X}$ i.e., if $U\sim \mathcal{U}(\mathcal{X})$ then the probability density of the law of $U$ is $\mathds{1}_{\mathcal{X}}/\lambda^d(\mathcal{X})$.

\subsection{Assumptions}

We focus on the problem of estimating probability measures that admit $\beta$-regular densities with $\beta > 1$ with respect to the volume measure of a closed (i.e., compact and without boundary) $d$-dimensional submanifold $\mathcal{M} \subset \mathbb{R}^p$, where $d \geq 2$ and $p > d$ are integers.

We begin by defining the notion of geometric regularity we impose on the underlying manifold. Let $\mathcal{M}$ be a closed $d$-dimensional submanifold embedded in $\mathbb{R}^p$, and fix $K > 1$. For any point $x \in \mathcal{M}$, define
\begin{equation}\label{eq:phix}
    U_x = B^p(x, K^{-1}) \cap \mathcal{M}, \quad \phi_x: U_x \rightarrow \mathcal{T}_x(\mathcal{M}), \quad \phi_x = \pi_{\mathcal{T}_x(\mathcal{M})} - x,
\end{equation}
where $B^p(x, K^{-1})$ denotes the Euclidean ball of radius $K^{-1}$ in $\mathbb{R}^p$ centered at $x$, and $\pi_{\mathcal{T}_x(\mathcal{M})}$ is the orthogonal projection onto the tangent space $\mathcal{T}_x(\mathcal{M})$ at $x$.

\begin{definition}[Manifold regularity condition]\label{def:manifoldcond}
    We say that a closed submanifold $\mathcal{M} \subset \mathbb{R}^p$ satisfies the $(\beta+1, K)$-manifold condition if:
    \begin{itemize}
        \item $\mathcal{M} \subset B^p(0, K)$, and
        \item For all $x \in \mathcal{M}$, the map $\phi_x: U_x \to \mathcal{T}_x(\mathcal{M})$ is a diffeomorphism whose inverse $\phi_x^{-1}$ belongs to the Hölder class $\mathcal{H}_K^{\beta+1}(\phi_x(U_x), U_x)$.
    \end{itemize}
\end{definition}

Definition~\ref{def:manifoldcond} is equivalent to requiring a lower bound on the \emph{reach} of a $(\beta+1)$-smooth submanifold~\citep{federer1959}. The reach $r_{\mathcal{M}}$ of a manifold $\mathcal{M}$ is the largest $\epsilon > 0$ such that the orthogonal projection $\pi_{\mathcal{M}}$ is well-defined on the tubular neighborhood
$
\mathcal{M}^\epsilon := \{x \in \mathbb{R}^p \mid d(x, \mathcal{M}) < \epsilon\},
$
where $d(x, \mathcal{M})$ denotes the Euclidean distance from $x$ to the manifold.
Next, we define the regularity condition imposed on the target density.

\begin{definition}[Density regularity condition]\label{def:densitycond}
    A measure $\mu$ supported on a submanifold $\mathcal{M}$ satisfies the $(\beta, K)$-density condition if:
    \begin{itemize}
        \item[i)] $\mu$ admits a density $f_\mu \in \mathcal{H}_K^\beta(\mathcal{M}, \mathbb{R})$ with respect to the volume measure on $\mathcal{M}$,
        \item[ii)] $f_\mu$ is bounded below by $K^{-1}$.
    \end{itemize}
\end{definition}
Assuming that the density is bounded below is a classical condition in the setting of density estimation on manifolds~\citep{divol2022measure, tang2023minimax}. In our context, this assumption is crucial for applying Caffarelli's regularity theory~\citep{villani2009optimal} to obtain smooth optimal transport maps between certain pushforward measures, as will be used in Proposition~\ref{prop:keydecomp}.

We can now formalize the statistical model under consideration.

\begin{assumption}[Statistical model]\label{assump:model}
There exists a submanifold $\mathcal{M}^\star$ satisfying the $(\beta+1, K)$-manifold condition such that the target measure $\mu^\star$ satisfies the $(\beta, K)$-density condition on $\mathcal{M}^\star$.
\end{assumption}

The statistical model described in Assumption~\ref{assump:model} follows the standard framework commonly adopted in the generative modeling literature~\citep{divol2022measure, tang2023minimax, azangulov2024convergence}. Our objective is to construct a tractable and generative estimator $\hat{\mu}$ of the target distribution $\mu^\star$ satisfying Assumption~\ref{assump:model}, based on an i.i.d.\ sample $X_1, \dots, X_n \sim \mu^\star$, such that $\hat{\mu}$ achieves the minimax convergence rates for the Hölder IPMs $d_{\mathcal{H}^\gamma_1}$, uniformly over all $\gamma \geq 1$.

\subsection{Main result}
Denoting $\mathcal{P}_\beta$ the set of probability measures on $\mathbb{R}^p$ satisfying Assumption \ref{assump:model}, the main result of the paper is the following.
\begin{theorem}[Informal version of Theorem \ref{theo:boundongamma}] There exists a tractable generative estimator $\hat{\mu}$ based on $n$-sample such that with probability at least $n^{-1}$, $\hat{\mu}$ satisfies Assumption \ref{assump:model}. Furthermore, for all $\gamma\geq 1$ we have that
\begin{align}\label{eq:optimdelesti}
    \sup_{\mu^\star \in \mathcal{P}_\beta}\ \mathbb{E}_{(X_i)_{i=1}^n\sim \mu^\star}[d_{\mathcal{H}^{\gamma}_1}(\hat{\mu},\mu^\star)] \underset{{\mathrm{polylog}(n)}}{\lesssim} \inf_{\hat{\theta}\in \Theta}\ \sup_{\mu^\star\in \mathcal{P}_\beta}\ \mathbb{E}_{(X_i)_{i=1}^n\sim \mu^\star}[d_{\mathcal{H}^{\gamma}_1}(\hat{\theta},\mu^\star)],
\end{align}
where $\Theta$ denotes the set of all possible estimators  of $\mu^\star$ based on $n$-sample.
\end{theorem}
Theorem~\ref{theo:boundongamma} establishes that our estimator $\hat{\mu}$ achieves the minimax convergence rate for all Hölder IPMs $d_{\mathcal{H}^\gamma_1}$ with $\gamma \geq 1$, over the class of $\beta$-regular distributions supported on unknown manifolds. Importantly, this result is achieved by a tractable estimator that is both generative and geometrically aware, in the sense that it satisfies the same regularity assumptions as the ground truth distribution (Assumption~\ref{assump:model}). The minimax rate in~\eqref{eq:optimdelesti} has been shown to be equal to $\tilde{O}\left(n^{-\frac{\beta+\gamma}{2\beta + d}} \vee n^{-1/2}\right)$ up to logarithmic factors in~\cite{tang2023minimax}. In contrast to their estimator, which is tailored to a fixed value of $\gamma > 0$, our estimator simultaneously achieves the minimax rate for all $\gamma \geq 1$ without requiring adaptation to a specific regularity level. This property is fundamental to the construction of our estimator: it is specifically designed to achieve the minimax rate of $n^{-1/2}$ for the metric $d_{\mathcal{H}^{d/2}_1}$ while satisfying Assumption~\ref{assump:model}. This, in turn, allows us to invoke the interpolation inequalities from~\cite{stephanovitch2024ipm}, ensuring that the estimator also attains the optimal rate for the other Hölder IPMs.

\section{Construction of the estimator}
In this section, we present the complete construction of our estimator. We begin in Section~\ref{sec:proptarget} by identifying key properties satisfied by probability measures that fulfill Assumption~\ref{assump:model}, and show how these properties can be leveraged in the design of the estimator. In Section~\ref{sec:Feffe}, we describe how to adapt the manifold reconstruction procedure of~\cite{fefferman2015reconstruction} to optimally approximate the data distribution. Finally, we introduce a wavelet-based parametrization of Besov spaces in Section ~\ref{sec:paramBesov}, which is then used to construct the generator and discriminator classes in Section~\ref{sec:genedis}.

\subsection{Properties of the data distribution}\label{sec:proptarget}
Throughout, $C, C_2,C_\star...$ denote constants that can vary from line to line and that only depend on $p,d,\beta$ and $K$.
Writing $\gamma_d$ for the density of $d$-dimensional standard Gaussian, the construction of the estimator is based on the following decomposition.

\begin{proposition}\label{prop:keydecomp}
    Let $\mu^\star$ a probability measure satisfying Assumption \ref{assump:model} on a manifold $\mathcal{M}^\star$. Then, there exists  a collection of maps $(\phi_i)_{i=1,...,m}\in \mathcal{H}^{\beta +1}_C(\mathbb{R}^d,\mathbb{R}^p)^m$ and weights $(\alpha_i)_{i=1,...,m}\in [C^{-1},C]^m$ such that for all  $h\in \mathcal{H}^0(\mathbb{R}^p)$ we have 
    \begin{equation*}
        \int_{\mathcal{M}^\star} h(x)d\mu^\star(x)= \sum_{i=1}^m \alpha_i \int_{\mathbb{R}^d}h(\phi_i(y))d\gamma_d(y).
    \end{equation*}
\end{proposition}
The proof of Proposition \ref{prop:keydecomp} can be found in Section \ref{sec:prop:keydecomp}. This result shows that the integration with respect to the data distribution $\mu^\star$ on the manifold $\mathcal{M}^\star$ can be decomposed into integrals over $\mathbb{R}^d$ via local charts. In particular, this decomposition captures both the $\beta$-regularity of the measure and its intrinsic dimension $d$. Building on this property, our estimator is constructed by learning a collection of maps $(\hat{g}_i)_{i=1,\dots,m} \in \mathcal{H}^{\beta+1}_{C}(\mathbb{R}^d, \mathbb{R}^p)^m$ and corresponding weights $(\hat{\alpha}_i)_{i=1,\dots,m} \in [C^{-1},C]^m$, such that the quantity
\begin{equation}\label{eq:aquantitytobound}
\sup_{h \in \mathcal{H}_1^\gamma(\mathbb{R}^p, \mathbb{R})}\left|\sum_{i=1}^m \alpha_i \int_{\mathbb{R}^d} h(\phi_i(y))  d\gamma_d(y) - \sum_{i=1}^m \hat{\alpha}_i \int_{\mathbb{R}^d} h(\hat{g}_i(y)) d\gamma_d(y)\right|
\end{equation}
achieves the minimax convergence rate.
To obtain such rates, we first establish a bound on the distance $d_{\mathcal{H}^{d/2}_1}$ and then apply an interpolation inequality to extend the result to all $d_{\mathcal{H}^{\gamma}_1}$, with $\gamma \geq 1$. The interpolation inequality is given in~\cite{stephanovitch2024ipm}, and is stated below.

\begin{theorem}\label{theo:theineqGAW}
 Let $\mathcal{M},\mathcal{M}^\star$ satisfying the $(\beta+1,K)$-manifold  condition and $\mu,\mu^\star$ two probability measures satisfying the  $(\beta,K)$-density condition on $\mathcal{M}$ and $\mathcal{M}^\star$ respectively. Then for all $1\leq \gamma\leq\eta$, we have
    \begin{align*}
        d_{\mathcal{H}^\gamma_1}(\mu,\mu^\star)\leq C\log\left(1+d_{\mathcal{H}^\eta_1}(\mu,\mu^\star)^{-1}\right)^{C_2} d_{\mathcal{H}^\eta_1}(\mu,\mu^\star)^{\frac{\beta+\gamma}{\beta+\eta}}. 
        \end{align*}
\end{theorem} 
To apply this inequality, both $\mu$ and $\mu^\star$ must satisfy the manifold and density regularity conditions from Assumption~\ref{assump:model}. Therefore, we must construct a global estimator $\hat{\mu}$ that inherits these properties by carefully gluing the local charts $(\hat{g}_i)_i$ used in~\eqref{eq:aquantitytobound}. Our approach is inspired by the manifold reconstruction method of~\cite{fefferman2015reconstruction} and the resulting estimator is trained with adversarial procedure analogous to the Wasserstein GAN.

\subsection{Adversarial manifold estimation: from Fefferman to GAN}\label{sec:Feffe}
\subsubsection{Fefferman's manifold reconstruction}
The manifold reconstruction procedure developed in~\cite{fefferman2015reconstruction} provides a constructive solution to the geometric Whitney problem: under minimal geometric assumptions, it recovers a smooth Riemannian manifold approximating a given metric space. Specifically, Theorem 1 shows that if a metric space $X$ is $\delta$-close to $\mathbb{R}^d$ at scale $r>0$ (i.e., locally Gromov--Hausdorff close) and is $\delta$-intrinsic (its metric approximates a geodesic distance), then there exists a smooth $d$-dimensional Riemannian manifold $M$ that is quasi-isometric to $X$ and has bounded sectional curvature and injectivity radius.
The reconstruction proceeds through the following steps:
\begin{enumerate}
    \item \textbf{Covering the space:} Select a maximal $r/100$-separated net $\{q_i\} \subset X$ and cover $X$ with metric balls $B_X(q_i, r)$.
    
    \item \textbf{Local parametrization:} For each $i$, construct a $\delta$-isometry $f_i: B_X(q_i, r) \to B^d_r \subset \mathbb{R}^d$, effectively approximating $X$ locally by Euclidean charts.
    
    \item \textbf{Embedding into Euclidean space:} Embed each Euclidean ball $B^d_{r/10}$ into a high-dimensional space $\mathbb{R}^p$ via smooth embeddings $F_i$, producing local patches $\Sigma_i = F_i(B^d_{r/10}) \subset \mathbb{R}^p$.
    
    \item \textbf{Gluing charts into a manifold:} Use a Whitney-type extension algorithm to glue the patches $\Sigma_i$ together into a smooth embedded submanifold $\mathcal{M} \subset \mathbb{R}^p$, ensuring compatibility and smoothness at overlaps.
    
    \item \textbf{Defining the metric:} Construct a Riemannian metric on $\mathcal{M}$ by smoothly interpolating the Euclidean metrics pulled back via the embeddings, using a partition of unity subordinate to the cover.
\end{enumerate}

This algorithm outputs a smooth manifold $M$ that approximates the original metric space $X$ both geometrically and topologically. In particular, when $X$ is a finite point cloud, it yields a smooth manifold that interpolates the data. However, for our purposes, this method has two key limitations: 
\begin{itemize}
    \item It does not provide a mechanism to generate new samples from the estimated manifold.
    \item It is not statistically optimal in the minimax sense, as the reconstruction does not adapt to the underlying regularity $\beta$ of the data distribution.
\end{itemize}

To address these limitations, we integrate this geometric reconstruction procedure with recent techniques from generative modeling.

\subsubsection{Generative modelling via Wasserstein GAN}
Generative Adversarial Networks (GANs) were first introduced in~\citep{goodfellow2014generative} as a framework for simultaneously estimating and sampling from complex target distributions. Since then, GANs have achieved remarkable success across a range of domains, including image generation~\citep{karras2021}, video synthesis~\citep{vondrick2016generating}, and text generation~\citep{SeqGANs}. The original formulation was later refined by~\cite{arjovsky2017wasserstein}, who introduced the Wasserstein GAN (WGAN), significantly improving training stability by replacing the Jensen–Shannon divergence with the Wasserstein distance.

A WGAN consists of two main components: a class of generator functions \( \mathcal{G} \) and a class of discriminators \( \mathcal{D} \). Given a simple base distribution \( \nu \) defined on a latent space \( \mathcal{Z} \), the generator \( g \in \mathcal{G} \), with \( g: \mathcal{Z} \rightarrow \mathbb{R}^p \), aims to approximate the target density \( \mu^\star \) by minimizing the $\mathcal{D}$-IPM over \( \mathcal{G} \):
\begin{equation}\label{eq:IPMgan}
    d_{\mathcal{D}}(\mu^\star, g_{\# }\nu) := \sup_{D \in \mathcal{D}} \left( \mathbb{E}_{X \sim \mu^\star}[D(X)] - \mathbb{E}_{Z \sim \nu}[D(g(Z))] \right).
\end{equation}

Each discriminator \( D \in \mathcal{D} \), where \( D: \mathbb{R}^p \rightarrow \mathbb{R} \), is trained to distinguish between real data drawn from \( \mu^\star \) and samples generated from the pushforward distribution \( g_{\#}\nu \). The generator, in turn, is updated to fool the discriminator, driving the generated distribution closer to the true one under the chosen IPM. In practice, having only acess to an iid sample $X_1,...,X_n$, the WGAN solves an empirical approximation of \eqref{eq:IPMgan}.  Specifically, for compact classes of generators $\mathcal{G}$ and discriminators $\mathcal{D}$, the WGAN estimator is defined as
\begin{equation}\label{eq:gan}
    \hat{g}\in \argmin_{g\in \mathcal{G}}\ \sup_{D\in  \mathcal{D}}\ \frac{1}{n} \sum_{i=1}^n D(g(U_i))-D(X_i),
\end{equation}
with  $U_i\sim \nu$, iid random variables. 

Assuming the existence of a \( C^{\beta+1} \) diffeomorphism \( g^\star \) from the \( d \)-dimensional torus \( \mathbb{T}^d \) onto its image in \( \mathbb{R}^p \) such that the target distribution satisfies \( \mu^\star = g^\star_{\#} U \), where \( U \) denotes the uniform measure on \( \mathbb{T}^d \),~\cite{stephanovitch2023wasserstein} show that the Wasserstein GAN estimator~\eqref{eq:gan} achieves the minimax convergence rates for estimating \( \mu^\star \). Unfortunately, in order to achieve minimax optimality, the WGAN estimator is constrained to generate distributions that share the same topology as the latent space. As a result, it cannot approximate measures supported on manifolds with arbitrary topology, limiting its applicability in more general geometric settings.

To overcome this topological limitation and obtain a more flexible estimator, our method leverages Proposition~\ref{prop:keydecomp} to construct a generalized GAN architecture with multiple generators, each serving as a local chart. These generators collectively approximate the target distribution in local coordinates. We then apply the manifold reconstruction technique of~\cite{fefferman2015reconstruction} to glue these charts together, thereby recovering a global manifold structure.

\subsubsection{Adversarial Manifold Reconstruction}\label{sec:amr}
Like in the Wasserstein GAN framework, our estimator leverages a generator–discriminator training scheme based on the decomposition provided in Proposition~\ref{prop:keydecomp}. The design of the generator class is guided by the following structural result. Let us take $\tau := \frac{1}{8K}$, which serves as a key radius determined by the regularity of  Assumption~\ref{assump:model} and a diffeomorphism $\Psi:B^d(0,\tau)\rightarrow \mathbb{R}^d$ defined by 
\begin{equation}\label{eq:diffeoloin}
    \Psi(u):= \frac{u}{\tau^2-\|u\|^2}.
\end{equation}
\begin{proposition}\label{prop:decompphi}
       Let $\mu^\star$ a probability measure satisfying Assumption \ref{assump:model} and $m\in \mathbb{N}$ given by Proposition \ref{prop:keydecomp}. Then, for each $i\in \{1,...,m\}$ there exists $g_i^1\in\mathcal{H}^{\beta+1}_{C}(B^d(0,2\tau),\mathbb{R}^p)$ satisfying $\forall u,v\in B^d(0,\tau)$ that
    $1-K\tau\leq \|\nabla g_i^1(u)\frac{v}{\|v\|}\|\leq 1+K\tau$ and there exists some diffeomorphisms $g_i^2\in\dot{\mathcal{H}}^{\beta+1}_{C}(\mathbb{R}^d,\mathbb{R}^d)$ satisfying $\forall x,y\in \mathbb{R}^d$ that
    $\|\nabla g_i^2(x)\frac{y}{\|y\|}\|\geq C^{-1}$, such that for $\phi_i$ given by Proposition \ref{prop:keydecomp} we have
    $$\phi_i = g_i^1\circ \Psi^{-1}\circ g_i^2.$$
\end{proposition}
The proof of Proposition \ref{prop:decompphi} can be found in Section \ref{sec:prop:decompphi}. This proposition shows that the transport maps \( (\phi_i )_i\), which satisfy
\[
\int_{\mathcal{M}^\star} h(x)\, d\mu^\star(x) = \sum_{i=1}^m \alpha_i \int_{\mathbb{R}^d} h(\phi_i(y))\, d\gamma_d(y),
\]
can be factored into three components with distinct roles:

\begin{itemize}
    \item \textbf{(Geometry)} The maps \( g_i^1 \) act as charts that reconstruct the local geometry of the manifold \( \mathcal{M} \); their Jacobian are constrained to ensure they do not distort volume and cannot adapt to density variations.
    \item \textbf{(Density)} The diffeomorphisms \( g_i^2 \) serve to modulate the density, allowing the composition \( \phi_i \) to locally match the target measure \( \mu \).
    \item \textbf{(Compactification)} The map \( \Psi^{-1} \) regularizes the domain by mapping \( \mathbb{R}^d \) into \( B^d(0,1) \) in a controlled fashion, which is essential for maintaining smoothness of the pushforward measures \( (\phi_i)_{\#} \gamma_d \).
\end{itemize}

This decomposition forms the backbone of our generative model architecture and guide the construction of our generator class.

\paragraph{Function classes} To properly define our estimator, we first introduce the classes of functions used in the adversarial training procedure. For $m\in \mathbb{N}$ given by Proposition \ref{prop:keydecomp},
the generator class is defined as
\begin{equation}\label{eq:G}
\mathcal{G} \subset \mathcal{H}^{\beta+1}(\mathbb{R}^d, \mathbb{R}^p)^m,
\end{equation}
to mimic the role of local weighted-charts $g_i^1\circ \Psi^{-1}\circ g_i^2$ of the unknown $(\beta+1)$-regular $d$-dimensional manifold $\mathcal{M}^\star$. To enable the gluing of the charts $g_i^1$ into a coherent global structure, we introduce a class of approximate inverse maps,
\begin{equation}\label{eq:phi}
\Phi \subset \mathcal{H}^{\beta+1}(\mathbb{R}^p, \mathbb{R}^d)^m.
\end{equation}
The weights associated with each generator range in
\begin{equation}\label{eq:A}
\mathcal{A} = [C^{-1},C]^m,
\end{equation}
with $C>0$ given by Proposition \ref{prop:keydecomp}.
The weight class $\mathcal{A}$ allows to  control the local contribution of each chart to the final reconstructed measure. Finally, the class of discriminators used for adversarial training is defined as
\begin{equation}\label{eq:D}
\mathcal{D} \subset \mathcal{H}^{d/2}(\mathbb{R}^p, \mathbb{R}),
\end{equation}
in order to match the critical smoothness $d/2$ required to achieve the rate $O(n^{-1/2})$.

\paragraph{Manifold reconstruction}
Let us now define our manifold construction based on generators $(g_i)_i\in \mathcal{G}$ and the approximate inverses $(\varphi_i)_i\in \Phi$ of the $(g_i^1)_i$. Let \( \Gamma \in \mathcal{H}^{\beta+1}_{C}([0,\infty), [0,1]) \) be a smooth cutoff function, and let \( \epsilon_\Gamma \in (0, \tau^2/4) \) be such that
\begin{equation}\label{eq:Gamma}
\Gamma\left([0, \tau(1 + (K + 2)\tau)]\right) = \{1\} \quad \text{and} \quad \Gamma\left([\tau(1 + (K + 2)\tau) + \epsilon_\Gamma, \infty)\right) = \{0\},
\end{equation}
recalling that $\tau=\frac{1}{8K}$.
The function $\Gamma$ is used to smoothly interpolate between local transformations and the identity map. For $i\in \{1,...,m\}$ and $x\in \mathbb{R}^p$, we define the local gluing operator \( F_i^{g,\varphi}: \mathbb{R}^p \rightarrow \mathbb{R}^p \) by
\begin{equation}\label{eq:Fi}
    F_i^{g,\varphi}(x) = \Gamma\left(\|x - g_i^1(0)\|\right) \cdot g_i^1 \circ \varphi_i(x) + \left(1 - \Gamma\left(\|x - g_i^1(0)\|\right)\right) \cdot x,
\end{equation}
and set the global gluing map as the composition
\begin{equation}\label{eq:Ffinal}
F^{g,\varphi} = F_m^{g,\varphi} \circ \cdots \circ F_1^{g,\varphi}.
\end{equation}
This construction is inspired by Section 5.2 of~\cite{fefferman2015reconstruction}, and is designed to ensure a smooth transition between the local charts \( g_i^1 \). It effectively merges the individual chart domains into a globally coherent manifold. We now design the adversarial training to both estimate $\mu^\star$ and constrain the maps \( \hat{g}_i^1 \in \mathcal{G} \) and \( \hat{\varphi}_i \in \Phi \) such that the reconstructed manifold
\[
\hat{\mathcal{M}} = \bigcup_{i=1}^m F^{\hat{g}, \hat{\varphi}} \circ \hat{g}_i^1\left(B^d(0, \tau)\right)
\]
satisfies the \( (\beta+1, K) \)-manifold condition of Definition \ref{def:manifoldcond}.

\paragraph{Adversarial training} Take $N\in \mathbb{N}_{>0}$, $Y_1,...,Y_N$ an iid sample of  $\mathcal{N}(0,\text{Id})$  and $\omega_1,...,\omega_N$ iid sample of the multinomial distribution $\textbf{Mult}(\{1/m,...,1/m\})$. Now,
 for $X_1,...,X_n$ the iid sample from the target measure $\mu^\star$ and $(g,\varphi,\alpha,D)\in \mathcal{G}\times \Phi \times \mathcal{A}\times \mathcal{D}$,
define
\begin{equation}\label{eq:Ln}
    L_{N,n}(g,\varphi,\alpha,D):=\frac{1}{N} \sum_{j=1}^N \alpha_{\omega_j}\  D\Big(F^{g,\varphi}\circ g_{\omega_j}(Y_j)\Big)-\frac{1}{n} \sum_{j=1}^n D(X_j),
\end{equation}
which generalizes the classical adversarial loss \eqref{eq:gan} used in the WGAN framework. In contrast to the standard setup where samples are generated from a single generator, here the generator index \( \omega_j \in \{1, \dots, m\} \) is chosen uniformly at random, and a synthetic sample is produced by applying the reconstruction map \( F^{g, \varphi} \) from \eqref{eq:Ffinal} to the chart output \( g_{\omega_j}(Y_j) \).
The term \( \alpha_{\omega_j} \) acts as a weight that modulates the contribution of each local chart, playing the role of \( \alpha_i \) introduced in Proposition~\ref{prop:keydecomp}.
In addition, taking $U_1,...,U_N$ an iid sample of $\mathcal{U}(B^d(0,2\tau))$ we define the consistency loss
\begin{equation}\label{eq:Cn}
    \mathcal{C}_N(g, \varphi) := \frac{1}{N} \sum_{j=1}^N \sum_{i,l=1}^m \left\| g_l^1(U_j) - g_i^1\circ \varphi_i \circ g_l^1(U_j) \right\|  \Gamma\big((\|g_l^1(U_j) - g_i^1(0)\|-\tau^2)\vee 0\big),
\end{equation}
which penalizes deviations from the condition \( \varphi_i \approx (g_i^1)^{-1} \).
We also introduce a regularization term
\begin{align}\label{eq:R}
    R(g) := &\sum_{z_1, z_2 \in \mathcal{Z}^1} \sum_{i=1}^m 
    \mathds{1}_{\left\{ 
        (1 + (K + \tfrac{1}{2})\tau) \|z_1 - z_2\| > \|g_i^1(z_1) - g_i^1(z_2)\| > 
        (1 - (K + \tfrac{1}{2})\tau) \|z_1 - z_2\| 
    \right\}},\\\nonumber
   & + \sum_{z_1, z_2 \in \mathcal{Z}^2} \sum_{i=1}^m 
    \mathds{1}_{\left\{\epsilon_\Gamma^{-1/2} \|z_1 - z_2\|> \|g_i^2(z_1) - g_i^2(z_2)\| > 
        \epsilon_\Gamma^{1/2} \|z_1 - z_2\| 
    \right\}}
\end{align}
where \( \mathcal{Z}^1 \) and \( \mathcal{Z}^2 \) are fixed \( \epsilon_\Gamma \)-covering of \( B^d(0, \tau) \) and \( B^d(0, 2\log(n)^{1/2}) \) respectively. The first term controls the local Lipschitz behavior of each chart-generator \( g_i^1 \) by encouraging its Jacobian to remain close to an isometry, effectively constraining the spectral radius and ensuring well-conditioned local geometry. Similarly, the second term controls the eigenvalues of the Jacobian of \( g_i^2 \) to ensure that the density of push-forward  $(g_i^2)_{\#}\gamma_d$ remains bounded above and below.

We now define our  intermediate estimators as the solution to the following constrained adversarial optimization problem:
\begin{equation}\label{eq:estimator}
    (\hat{g},\hat{\varphi},\hat{\alpha}):= \argmin_{(g,\varphi,\alpha)\in \mathcal{G}\times  \Phi\times  \mathcal{A}}\ \sup_{D\in \mathcal{D}} \ L_{N,n}(g,\varphi,\alpha,D),
\end{equation}
subject to the constraint
\begin{equation}\label{eq:opticonstrain}
\mathcal{C}_N(g, \varphi) + R(g) \leq \epsilon_\Gamma.
\end{equation}
The constraint \eqref{eq:opticonstrain} is designed to enforce geometric consistency  ensuring that the global reconstruction map $F^{\hat{g},\hat{\varphi}}$ effectively produces a manifold satisfying the desired assumptions.
Let $\gamma_d^n$ denote the probability measure with density being proportional to
\begin{align}\label{align:mdofiedgauss}
   \gamma_d^n(x)\propto \gamma_d(x)\Gamma\Big((\|x\|-\log(n)^{1/2})\mathds{1}_{\{\|x\|\geq  \log(n)^{1/2}\}}\Big),
\end{align}
with $\Gamma$ the smooth cutoff function defined in \eqref{eq:Gamma}. This measure is used to approximate the Gaussian $\gamma_d$ to an error $1/n$ while being supported on the ball $B^d(0,\log(n)^{1/2}+1)$  where the generators are well-behaved.

Using this reference measure, we define our final estimator as the pushforward distribution
\begin{equation}\label{eq:hatmu}
    \hat{\mu}:=(F^{\hat{g},\hat{\varphi}}\circ \hat{g})_{\# \hat{\alpha}\gamma_d^n},
\end{equation}
where the latter is defined by duality for any continuous test function $h\in \mathcal{H}^0(\mathbb{R}^p,\mathbb{R})$ by
\begin{equation}\label{eq:push}
\int h(x)d(F^{\hat{g},\hat{\varphi}}\circ \hat{g})_{\# \hat{\alpha}\gamma_d^n}(x):=\sum \limits_{i=1}^m \frac{\hat{\alpha}_i}{\sum_j \hat{\alpha}_j} \int_{\mathbb{R}^d}h\big(F^{\hat{g},\hat{\varphi}}\circ\hat{g}_i(y)\big)d\gamma_d^n(y).
\end{equation}
This construction enables an efficient sampling procedure from the estimator, which proceeds as follows:
\begin{enumerate}
    \item Sample $y\in B^d(0,\log(n)^{1/2}+1)$ according to the modified Gaussian distribution $\gamma_d^n$ \eqref{align:mdofiedgauss}.
    \item                  Sample $i\in \{1,...,m\}$ according to the multinomial distribution $\textbf{Mult}(\{\hat{\alpha}_1,...,\hat{\alpha}_m\})$.
    \item Compute the new sample $F^{\hat{g},\hat{\varphi}}\circ\hat{g}_i(y)$.
\end{enumerate}
In practice, sampling from $\gamma_d^n$ is straightforward: drawing $y\sim \mathcal{N}(0,\text{Id})$,  yields $y\in B^d(0,\log(n)^{1/2})$ with probability $1-O(n^{-1})$, so standard Gaussian samples can be used directly with negligible error. 
Compared to score-based generative models, which require numerically solving a stochastic differential equation for each sample, this sampling scheme is significantly more computationally efficient.

In the next section, we provide explicit constructions of the function classes \( \mathcal{G} \), \( \Phi \), \( \mathcal{A} \) and \( \mathcal{D} \), which are used to define the generator, inverse, weighting, and discriminator families respectively.

\subsection{Wavelet parametrization of adversarial classes}
\subsubsection{Parametrized approximation of Besov spaces}\label{sec:paramBesov}
To control the regularity of our function classes, we parametrize them via their wavelet decompositions. We begin by recalling the definition of Besov spaces, which provide a natural functional framework for measuring smoothness in terms of wavelet coefficients.

Let $\psi,\phi\in \mathcal{H}^{d/2+2}(\mathbb{R},\mathbb{R})$ be a compactly supported \emph{scaling} and \emph{wavelet} function respectively (see Daubechies wavelets \citep{daubechies1988orthonormal}). For ease of notation, the functions $\psi,\phi$ will be written $\psi_0,\psi_1$ respectively. Then for $j\in \mathbb{N},l \in \{1,...,2^p-1\}, w \in \mathbb{Z}^p$, the family of functions 
$$\psi_{0w}(x) = \prod \limits_{i=1}^p \psi_{0}(x_i-w_i) \ \text{ , } \ \psi_{jlw}(x) = 2^{jp/2}\prod \limits_{i=1}^p \psi_{l_i}(2^{j}x_i-w_i)
$$
form an orthonormal basis of $L^2(\mathbb{R}^p,\mathbb{R})$ (with $l_i$ the $i$-th digit of the base-2-decomposition of $l$).
Let $q_1,q_2\geq 1,s>0,b\geq 0$ be such that $d/2+2>s$. The Besov space $\mathcal{B}^{s,b}_{q_1,q_2}(\mathbb{R}^p,\mathbb{R})$ consists of functions $f$ that admit a wavelet expansion in $L^2$:
$$
f(x)=\sum \limits_{w\in \mathbb{Z}^p} \alpha_f(w)\psi_{0w}(x) + \sum \limits_{j=0}^\infty \sum \limits_{l=1}^{2^p-1}\sum \limits_{w\in \mathbb{Z}^p} \alpha_f(j,l,w)\psi_{jlw}(x)
$$
equipped with the norm
\begin{align*}
\|f\|_{\mathcal{B}^{s,b}_{q_1,q_2}}= &\Biggl(\left(\sum \limits_{w\in \mathbb{Z}^p} |\alpha_f(w)|^{q_1}\right)^{q_2/q_1}\\
& +\sum \limits_{j=0}^\infty 2^{jq_2(s+p/2-p/q_1)}(1+j)^{bq_2} \sum \limits_{l=1}^{2^p-1} \Big(\sum \limits_{w\in \mathbb{Z}^p} |\alpha_f(j,l,w)|^{q_1}\Big)^{q_2/q_1}\Biggl)^{1/q_2}.
\end{align*}
with the usual modification for $q_1,q_2=\infty$.
Note that for $b=0$, $\mathcal{B}^{s,0}_{q_1,q_2}$ coincides with the classical Besov space $\mathcal{B}^{s}_{q_1,q_2}$ \citep{gine_nickl_2015}.
For simplicity of notation, we write
$$
f(x)= \sum \limits_{j=0}^\infty \sum \limits_{l=1}^{2^p}\sum \limits_{w\in \mathbb{Z}^p} \alpha_f(j,l,w)\psi_{jlw}(x)
$$
with the convention that $\psi_{02^pw}=\psi_{0w}$ and for all $j\geq 1$, $\psi_{j2^pw}=0$. The Besov spaces can be generalized for any $s\in \mathbb{R}$ as a subspace of the space of tempered distribution $\mathcal{S}^{'}(\mathbb{R}^p)$. Indeed for $f \in \mathcal{S}^{'}(\mathbb{R}^p)$, writing $\alpha_f(j,l,w)=\langle f,\psi_{jlw}\rangle$, the Besov space for $s\in \mathbb{R}$ is defined as
$$\mathcal{B}^{s,b}_{q_1,q_2}=\{f\in \mathcal{S}^{'}(\mathbb{R}^p) | \|f\|_{\mathcal{B}^{s,b}_{q_1,q_2}}<\infty\}.$$

The same way, we write $\mathcal{B}^{s,b}_{q_1,q_2}(K)=\{f\in \mathcal{S}^{'}(\mathbb{R}^p) | \|f\|_{\mathcal{B}^{s,b}_{q_1,q_2}}\leq K\}.$
 In the following, we will use intensively the connection between Hölder and Besov spaces.
\begin{lemma}\label{lemma:inclusions} (Proposition 4.3.23 \cite{gine_nickl_2015}, (4.63)  \cite{haroske2006envelopes}) If $\eta>0$ is a non integer, then for all $\alpha\geq 0$
$$\mathcal{H}^{\eta,\alpha}(\mathbb{R}^p,\mathbb{R})=\mathcal{B}^{\eta,\alpha}_{\infty,\infty}(\mathbb{R}^p,\mathbb{R})$$
with equivalent norms. If $\eta\geq 0$ is an integer, then $\forall \epsilon>0$
$$\mathcal{B}^{\eta,1+\epsilon}_{\infty,\infty}(\mathbb{R}^p,\mathbb{R}) \xhookrightarrow{} \mathcal{H}^\eta(\mathbb{R}^p,\mathbb{R})\xhookrightarrow{} \mathcal{B}^\eta_{\infty,\infty}(\mathbb{R}^p,\mathbb{R}),$$
where we write $A\xhookrightarrow{} B$ if the function space $A$ compactly injects in the function space $B$.
\end{lemma}

Let us now describe how to approximate the Besov spaces. For $k\in \mathbb{N}_{>0}$, $\eta,\delta,R>0$, and $C_\eta>0$ such that $\mathcal{H}^\eta_1(\mathbb{R}^k,\mathbb{R})\subset \mathcal{B}^{\eta}_{\infty,\infty}(\mathbb{R}^k,\mathbb{R},C_\eta)$, define
\begin{equation}\label{eq:fdeltaper}
\mathcal{F}_{\delta,R}^{k,\eta}=\{f\in \mathcal{B}^{\eta}_{\infty,\infty}(\mathbb{R}^k,\mathbb{R},C_\eta K) | \ \langle f_i,\psi_{j,l,w}^k\rangle_{L^2}=0,  \forall j\geq \log_2(\delta^{-1}),supp(f)\subset B^k(0,R)\}
\end{equation}
the set of functions in $\mathcal{B}^{\eta}_{\infty,\infty}$ that do not have frequencies higher than $\log_2(\delta^{-1})$ and are supported in $B^k(0,R)$. The constant $C_\eta$ is chosen such that $\mathcal{H}_1^{\eta}(\mathbb{R}^k,\mathbb{R})\subset \mathcal{B}_{\infty,\infty}^{\eta}(\mathbb{R}^k,\mathbb{R},C_\eta)$. In particular, any function in $\mathcal{B}_{\infty,\infty}^{\eta}(\mathbb{R}^k,\mathbb{R},C_\eta)$ and supported in $B^k(0,2K)$, can be approximated in $\|\cdot\|_{B^{0}_{\infty,\infty}}$-norm by a function in $\mathcal{F}_\delta^{k,\eta}$ up to the error $C\delta^\eta$. 

We now describe how to approximate the function class \( \mathcal{F}_{\delta,R}^{k,\eta} \) in practice. Let \( \hat{\phi}:\mathbb{R} \rightarrow \mathbb{R} \) be an approximation of the scaling function \( \phi \), such that 
\begin{equation}\label{eq:approxipsipaper2}
\|\phi-\hat{\phi}\|_{\mathcal{H}^{\lfloor d/2 \rfloor +2}} \leq n^{-1/2}.
\end{equation}
Such function can be efficiently computed using the classical "cascade algorithm" \citep{daubechies1991two} or neural networks \citep{stephanovitch2023wasserstein}. 
Define the approximation of $\mathcal{F}_\delta^{k,\eta}$ using the approximated wavelet $\hat{\psi}$ built from $\hat{\phi}$ (see Section A.2 of \cite{stephanovitch2023wasserstein} for the construction) as the following class of functions:
\begin{align}\label{eq:fdeltaperapprox}
&\hat{\mathcal{F}}_{\delta,R}^{k,\eta}=\Big\{\sum \limits_{j=0}^{\log_2(\delta^{-1})} \sum \limits_{l=1}^{2^k} \sum \limits_{w\in \{-R2^j,...,R2^j\}^k} \hat{\alpha}(j,l,w) \hat{\psi}^k_{jlw} \text{ with } |\hat{\alpha}(j,l,w)|\leq C_\eta K(1+\|z\|)2^{-j(\eta+k/2)}\Big\},
\end{align}
The following proposition shows that \( \hat{\mathcal{F}}_\delta^{k,\eta} \) is compactly embedded in Hölder spaces.
\begin{proposition}[Proposition E.3 in~\cite{stephanovitch2023wasserstein}]\label{prop:reguofG}
    For $k,R\in \mathbb{N}_{>0}$, $\eta\in (0,d/2]$ and $\delta\in(0,1)$, we have 
    $$\hat{\mathcal{F}}_{\delta,R}^{k,\eta}\subset\mathcal{H}^{\eta}_{CR\log(\delta^{-1})^{2\mathds{1}_{\{\eta=\lfloor \eta \rfloor\}}}}(\mathbb{R}^k,\mathbb{R}).$$
    Furthermore, for all $\gamma \in (0,\eta)$, we have
     $$\hat{\mathcal{F}}_{\delta,R}^{k,\eta}\subset\mathcal{H}^{\gamma}_{C_\gamma R}(\mathbb{R}^k,\mathbb{R}).$$
\end{proposition}
The additional logarithmic factor when \(\eta\) is an integer arises from the fact that, unlike the case \(\eta \notin \mathbb{N}\), the Besov space \(B^\eta_{\infty,\infty}\) does not coincide with the Hölder space \(\mathcal{H}^\eta\) as shown in Lemma \ref{lemma:inclusions}. Proposition \ref{prop:reguofG} is crucial, as it guarantees that the functions used to build our estimator lie in the appropriate regularity class, thereby ensuring that the output manifold can satisfy the required \( (\beta+1, K) \)-manifold condition.
Furthermore, we can quantify the precision of approximation of the class $\hat{\mathcal{F}}_{\delta,R}^{k,\eta}$ writing $\alpha_f(j,l,w)=\langle f, \psi^k_{jlw}\rangle$ for any $f\in L^2(\mathbb{R}^k)$.
\begin{proposition}\label{prop:qualityofapproxfketa} Let $k,R\in \mathbb{N}_{>0}$, $\eta\in (0,d/2]$ and $\delta\in(0,1)$. Then,
    for all $f\in \mathcal{H}^{\eta}_1(\mathbb{R}^k,\mathbb{R})$, the function $\overline{f}\in \hat{\mathcal{F}}_{\delta,R}^{k,\eta}$ defined by
    $$\overline{f}:=\sum \limits_{j=0}^{\log_2(\delta^{-1})} \sum \limits_{l=1}^{2^k} \sum \limits_{z\in \{-K2^j,...,K2^j\}^k} \alpha_f(j,l,w) \hat{\psi}^k_{jlw}$$
    satisfies for all $\gamma \in (-\infty,\eta)$, $b \geq 0$ and $q_1,q_2\in [1,\infty]$ that 
    $$  \|f-\overline{f}\|_{\mathcal{B}^{\gamma,b}_{q_1,q_2}(B^k(0,R))}\leq C_{\gamma,b}R^{\frac{k}{q_1}} \big(n^{-1/2}+\log(\delta^{-1})^b\delta^{\eta-\gamma}\big).$$
\end{proposition}
The Proof of Proposition \ref{prop:qualityofapproxfketa} can be found in Section \ref{sec:prop:qualityofapproxfketa}. This Proposition shows in particular that—up to the \( n^{-1/2} \) term arising from the approximation of the scaling function in~\eqref{eq:approxipsipaper2}—the class \( \hat{\mathcal{F}}_{\delta,R}^{k,\eta} \) retains the same approximation power as the original class \( \mathcal{F}_{\delta,R}^{k,\eta} \).

\subsubsection{Generator and Discriminator classes}\label{sec:genedis}
Using the low-frequency approximation classes \( \hat{\mathcal{F}}_\delta^{k,\eta} \) introduced in the previous section, we now define the functional classes used in the construction of the estimator \( \hat{\mu} \) in~\eqref{eq:hatmu}. Let \( n \) be the number of samples drawn from the target distribution \( \mu^\star \), and \( N \) the number of fake samples used in the adversarial loss~\eqref{eq:Ln}. We define the approximation scales
\[
\delta_n := n^{-\frac{1}{2\beta + d}}, \quad \delta_N := N^{-\frac{1}{2\beta + d}}.
\]
The class of generator functions is then defined as
\begin{equation}\label{eq:classofgene}
    \mathcal{G} = \left\{\big(g_i^1\circ \Psi^{-1}\circ g_i^2\big)_{i=1}^m\ \Big|\ g_i^1\in \left(\hat{\mathcal{F}}_{\delta_N,2K+C}^{d,\beta+1}\right)^{p},\ g_i^2\in\left(\hat{\mathcal{F}}_{\delta_N,4\log(n)^{1/2}}^{d,\beta+1}\right)^{d}\right\}.
\end{equation}
This generator class is designed to approximate the decomposition of the maps $\phi_i$ given by Proposition \ref{prop:decompphi}. Similarly, the class of inverse functions  is  defined as
\begin{equation}\label{eq:classeofinv}
     \Phi:= \left\{(\varphi_i)_{i=1}^m \big|\  \varphi_i \in \left(\hat{\mathcal{F}}_{\delta_N,K}^{p,\beta+1}\right)^{d}\right\}
\end{equation}
and the discriminator class is defined as
\begin{equation}\label{eq:classesofdis}
     \mathcal{D}:= \hat{\mathcal{F}}_{\delta_n,K}^{p,d/2}.
\end{equation}
The choice of this specific function class for the discriminator is particularly convenient, as it allows one to control the gradient norm of the discriminator functions in a principled way—something that is often difficult to enforce in practice. More details on the implementation of these classes of functions can be found in Section \ref{sec:numeric}.

\section{Proof of the minimax optimality }
In this section, we outline the main steps of the proof of the minimax optimality of our estimator. We begin in Section~\ref{sec:reguofes} by establishing the regularity properties of estimators that achieve a small value of the loss function. Then, in Section~\ref{sec:statisticalbound}, we present the statistical guarantees for our estimator, showing that it attains the optimal convergence rates for Hölder IPMs.
\subsection{Regularity properties of the estimator}\label{sec:reguofes}
\subsubsection{Local regularity}
We begin by establishing that the charts used for the estimator construction possess \( \mathcal{H}^{\beta+1} \)-regularity, with a  bound on the norm depending on whether \( \beta \) is an integer or not.

\begin{proposition}\label{prop:regularityfandg}
Let \( (g, \varphi) \in \mathcal{G} \times \Phi \). Then for all \( i, j_1, \ldots, j_k \in \{1, \ldots, m\} \), with \( k \in \{0, \ldots, m\} \), and any \( \eta < \beta+1 \), we have
\[
F^{g,\varphi}_{j_k} \circ \cdots \circ F^{g,\varphi}_{j_1} \circ g_i \in \mathcal{H}^{\eta}_{C_\eta}(\mathbb{R}^d, \mathbb{R}^p).
\]
Moreover:
\begin{itemize}
    \item If \( \beta \notin \mathbb{N} \), then the composition lies in \( \mathcal{H}^{\beta+1}_C(\mathbb{R}^d, \mathbb{R}^p) \),
    \item If \( \beta \in \mathbb{N} \), it belongs to \( \mathcal{H}^{\beta+1}_{C\log(n)^2}(\mathbb{R}^d, \mathbb{R}^p) \).
\end{itemize}
\end{proposition}
The proof of Proposition~\ref{prop:regularityfandg}, provided in Section~\ref{sec:prop:regularityfandg}, follows directly from the regularity properties of the low-frequencies approximating classes \( \hat{\mathcal{F}}_{\delta,R}^{k,\eta} \) established in Proposition~\ref{prop:reguofG}. The key observation is that the compositions of the gluing maps \( F^{g,\varphi}_{j_k} \circ \cdots \circ F^{g,\varphi}_{j_1} \) with the generator \( g_i \) inherit the smoothness of their components. The distinction between the case where \( \beta \) is an integer or not arises from the fact that Besov spaces embed continuously into Hölder spaces only when the regularity parameter is not an integer. Otherwise, the embedding holds up to a logarithmic factor, which accounts for the \( \log(n)^2 \) term in the regularity bound.
  
\subsubsection{Global regularity}
To analyze the global geometry of the estimator, we consider the infinite-sample versions of the empirical losses used during training.
For \( (g, \varphi,\alpha,D) \in \mathcal{G} \times \Phi \times \mathcal{A}\times \mathcal{D} \), let us define:
$$L(g,\varphi,\alpha,D)=\frac{1}{m}\sum_{i=1}^m \alpha_i \mathbb{E}_{Y\sim \gamma_d^n}[ D(F^{g,\varphi}\circ g_{i}(Y))]-\mathbb{E}_{X\sim \mu^\star}[ D(X)],$$
$$
\mathcal{C}(g,\varphi)=\sum_{i,j=1}^m\mathbb{E}_{U\sim \mathcal{U}(B^d(0,2\tau))}\big[\|g_{j}(U)-g_i\circ \varphi_i\circ g_{j}(U)\|\Gamma\big((\|g_j^1(U) - g_i^1(0)\|-\tau^2)\vee 0\big)\big].
$$
Assuming that the quantities \( L(g, \varphi, z, D) \), \( \mathcal{C}(g, \varphi) \), and \( R(g) \) are sufficiently small, the next result shows that the support of the estimator forms a $(\beta+1)$-smooth submanifold.

 \begin{proposition}
\label{prop:isamanifold}Let $\mu^\star$ a probability measure satisfying Assumption \ref{assump:model}. There exists $C_1,C_2>0$ such that for $\epsilon_\Gamma \in (C_1^{-2},C_1^{-1})$ and $n\geq C_2$, if $(g,\varphi,\alpha)\in \mathcal{G}\times  \Phi\times \mathcal{A}$, satisfies $$\sup_{D\in \mathcal{D}} \ L(g,\varphi,\alpha,D)+\mathcal{C}(g,\varphi)+R(g)\leq C_2^{-1},$$ then the set $$\mathcal{M}=\bigcup_{i=1}^m F^{g,\varphi}\circ g_{i}(B^d(0,\log(n)^{1/2}+1))$$ satisfies the $(\beta+1,C)$-manifold condition if $\beta$ is not an integer and satisfies the $(\beta+1,C\log(n)^2)$-manifold condition if $\beta$ is  an integer.    
 \end{proposition}
The proof of Proposition~\ref{prop:isamanifold} can be found in Section~\ref{sec:prop:isamanifold}. This result is central to bridging the variational formulation of the estimator with its geometric properties. It guarantees that, when the loss is sufficiently small, the support of the estimator forms a \( (\beta+1) \)-regular manifold. As a consequence, the estimator faithfully captures both the topology and the regularity of the true data-generating manifold.

As in~\cite{fefferman2015reconstruction}, the manifold reconstruction techniques developed in Section~\ref{sec:amr} do not apply to arbitrary sets of maps. However, by enforcing the geometric structure through the losses \( L, \mathcal{C}, R \), we guarantee that the reconstructed support remains close to the true manifold, thereby yielding a globally consistent reconstruction. In addition, we verify that the estimator defines a probability distribution with a smooth density with respect to the volume measure on its support.

\begin{proposition}\label{prop:densitychecked}
 Let $\mu^\star$ a probability measure satisfying Assumption \ref{assump:model}. Then, there exists $C_1,C_2>0$ such that for $\epsilon_\Gamma \in (C_1^{-2},C_1^{-1})$ and $n\geq C_2$, if $(g,\varphi,\alpha)\in \mathcal{G}\times  \Phi\times \mathcal{A}$ satisfies $$\sup_{D\in \mathcal{D}} \ L(g,\varphi,\alpha,D)+\mathcal{C}(g,\varphi)+R(g)\leq C_2^{-1} ,$$ then for  $$\mathcal{M}=\bigcup_{i=1}^m F^{g,\varphi}\circ g_{i}(B^d(0,\log(n)^{1/2}+1)),$$
 the probability measure $(F^{g,\varphi}\circ g)_{\# \alpha\gamma_d^n}$ from \eqref{eq:push} satisfies the $(\beta+1,C)$-density condition  on $\mathcal{M}$ if $\beta$ is not an integer and satisfies the $(\beta+1,C\log(n)^{C_2})$-density condition if $\beta$ is  an integer.
\end{proposition}
The proof of Proposition \ref{prop:densitychecked} can be found in Section \ref{sec:prop:densitychecked}. This result guarantees not only the geometric correctness of the learned support, as ensured by Proposition~\ref{prop:isamanifold}, but also the density regularity of the estimator. In particular, it ensures that the pushforward density formed via the weighted sum of local charts, remains bounded below and $\beta$-smooth. In the next section, we show that with high probability, the estimator \( (\hat{g}, \hat{\varphi}, \hat{\alpha}) \in \mathcal{G} \times \Phi \times \mathcal{A} \) satisfies
\[
\sup_{D \in \mathcal{D}} L(\hat{g}, \hat{\varphi}, \hat{\alpha}, D) + \mathcal{C}(\hat{g}, \hat{\varphi}) + R(\hat{g}) \leq C_2^{-1},
\]
so that Propositions~\ref{prop:isamanifold} and~\ref{prop:densitychecked} apply to the final estimator \( \hat{\mu} \).

\subsection{Statistical bounds}\label{sec:statisticalbound}
\subsubsection{Well-posedness of the estimator}
We begin by specifying the hyperparameters involved in the construction of the estimator. Based on the generator and discriminator classes introduced in~\eqref{eq:classofgene}, \eqref{eq:classeofinv} and~\eqref{eq:classesofdis}, respectively, our final estimator takes the form
\begin{equation}\label{eq:finalest}
   \hat{\mu} =  (F^{\hat{g},\hat{\varphi}}\circ \hat{g})_{\# \hat{\alpha}\gamma_d^n},
\end{equation}
with $(F^{\hat{g},\hat{\varphi}}\circ \hat{g})_{\# \hat{\alpha}\gamma_d^n}$ defined in \eqref{eq:hatmu} and the gluing parameter \( \epsilon_\Gamma\) is set to \( C_1^{-1} \), with \( C_1 \) the constant from Proposition~\ref{prop:densitychecked}. 

The number of generated samples \( N \) is chosen depending on the smoothness parameter \( \beta \): we take \( N = n \) when \( \beta+1 \geq d/2 \), and \( N = n^{\frac{2\beta + d}{4\beta+2}} \) when \( \beta+1 < d/2 \). This adjustment compensates for a loss of smoothness in the composition \( D \circ g \): although the discriminator \( D \) may be highly regular, its composition with the generator \( g \) inherits the lower regularity of \( g \). When the generators are not smooth enough, using more generated samples helps reduce the variance term associated with the composite class \( \mathcal{D} \circ \mathcal{G} \). We now verify that the estimator \( \hat{\mu} \) is well defined, i.e., the optimization domain defined under the constraint \eqref{eq:opticonstrain} is non-empty.

\begin{lemma}\label{lemma:genenonempty}
    There exists $C>0$ such that if $N\geq C$, then there exists $(g,\varphi)\in \mathcal{G}\times  \Phi$ satisfying
    $$\mathcal{C}(g,\varphi) + \mathcal{C}_N(g,\varphi) +R(g)\leq \epsilon_\Gamma.$$
\end{lemma}
The proof of Lemma \ref{lemma:genenonempty} can be found in Section \ref{sec:lemma:genenonempty}. This ensures that the minimization problem~\eqref{eq:estimator} is well-posed for sufficiently large sample sizes.

\subsubsection{Bias-variance trade-off}

We now introduce several key quantities that appear in the  bias-variance decomposition of the estimation error. We begin by defining the approximation error associated with the generator class:
\begin{equation}\label{eq:deltaG}
    \Delta_{\mathcal{G}\times \Phi \times \mathcal{A}} := \inf_{(g,\varphi,\alpha)\in \mathcal{G}\times  \Phi\times \mathcal{A}} d_{\mathcal{H}_1^{d/2}}((F^{g,\varphi}\circ g)_{\# \alpha \gamma_d^n},\mu^\star),
\end{equation}
which measures how well the generative classes can approximate the target distribution \( \mu^\star \) for the \( d_{\mathcal{H}_1^{d/2}} \) metric. Likewise,
we define the (localized) approximation error of the discriminator class:
\begin{equation}\label{eq:deltaD}
    \Delta_{\mathcal{D}_{\hat{\mu}}} := d_{\mathcal{H}_1^{d/2}}(\hat{\mu},\mu^\star)-d_{\mathcal{D}}(\hat{\mu},\mu^\star),
\end{equation}
which quantifies the discrepancy between the full \( \mathcal{H}_1^{d/2} \)-IPM and the restricted IPM over \( \mathcal{D} \) evaluated at the estimator \( \hat{\mu} \).
To control the statistical error, we will also use covering numbers. For \( \theta > 0 \), we define a minimal \( \theta \)-covering of a function class \( \mathcal{F} \) as
\begin{equation}\label{eq:coveringnum}
|\mathcal{F}_\theta|:=\argmin \{|A|\ |\ \forall f \in \mathcal{F},\exists f_\theta \in A, \|f-f_\theta\|_\infty\leq \theta\}.
\end{equation}

Finally, for convenience, we define the class of gluing maps induced by the generator and inverse functions as
\begin{equation}\label{eq:fgphi}
    \mathcal{F}^{\mathcal{G},\Phi}:=\{F^{g,\varphi}|g\in \mathcal{G}, \varphi \in \Phi\}.
\end{equation}
Using these quantities, we establish a bias-variance trade-off for the expected estimation error of \( \hat{\mu} \) in the following result.
\begin{theorem}\label{theo:biaisvariancedec} Let 
$N\geq n$ large enough such that their exists $(g,\varphi)\in  \mathcal{G}\times  \Phi$ satisfying the constraint \eqref{eq:estimator}. Supposing that $\mu^\star$ satisfies Assumptions  \ref{assump:model}, the estimators $\hat{\mu}=(F^{\hat{g},\hat{\varphi}}\circ \hat{g})_{\# \hat{\alpha}\gamma_d^n}$ \eqref{eq:estimator} satisfies
   \begin{align*}
 \mathbb{E}_{X_1,...,X_n\sim \mu^\star}[d_{\mathcal{H}^{d/2}_1}(\hat{\mu},\mu^\star)]\leq & \mathbb{E}_{X_1,...,X_n\sim \mu^\star}[\Delta_{\mathcal{D}_{\hat{\mu}}}+\Delta_{\mathcal{G}\times \Phi \times \mathcal{A}}]\\
  &+ \frac{C\log(N)^{C_2}}{\sqrt{N}} \min \limits_{\delta\in [N^{-1},1]}  \left( \delta\log\Big(|(\mathcal{D}\circ\mathcal{F}^{\mathcal{G},\Phi}\circ \mathcal{G})_{1/N}||\mathcal{A}_{1/N}|\Big)^{1/2}+\delta^{1-\frac{d}{d\wedge(2(\beta+1))}}\right)\\
 & +\frac{C\log(n)^{C_2}}{\sqrt{n}}\left( n^{-\frac{d/2}{2\beta+d}}\log\Big(|(\mathcal{D}\circ \phi)_{1/n}|\Big)^{1/2}+1\right),
\end{align*}
for a certain $\phi\in \mathcal{H}^{\beta+1}_C(B^d(0,\tau),\mathbb{R}^p)$.
\end{theorem}
The proof of Theorem \ref{theo:biaisvariancedec} can be found Section \ref{sec:theo:biaisvariancedec}. Theorem~\ref{theo:biaisvariancedec} provides a non-asymptotic bias-variance decomposition for the expected estimation error of the proposed estimator. Its structure closely resembles classical results in the GAN literature (see Theorem~1 in~\cite{puchkin2024rates} or Theorem 5.4 in \cite{stephanovitch2023wasserstein}), where the error is split into an approximation component and a statistical estimation component governed by the complexity of the function classes.

Importantly, the complexity terms involving the discriminator class \( \mathcal{D} \) do not depend on the ambient dimension \( p \). This is due to the fact that discriminators are always composed with functions mapping from \( \mathbb{R}^d \), ensuring that only the intrinsic dimension \( d \) intervenes in the statistical complexity. This key property is formalized in the following result.

\begin{lemma}[Lemma 5.5 in \cite{stephanovitch2023wasserstein}]\label{lemma:coveringd}
    Let $g\in \text{Lip}_K(B^d(0,1),B^p(0,1))$, $\eta>0$ and $\delta,\epsilon\in (0,1)$. Then, there exists an $\epsilon$-covering $(\mathcal{N}_g)_\epsilon$ of the class $\mathcal{N}_g=\{D\circ g\ |\ D\in \hat{\mathcal{F}}^{p,\eta}_\delta\}$, such that 
    $$\log(|(\mathcal{N}_g)_\epsilon|)\leq C\delta^{-d}\log(\delta^{-1})\log(\epsilon^{-1}).$$
\end{lemma}
This bound confirms that the covering number of the composed class depends only on the intrinsic dimension $d$ rather than on the ambient dimension
$p$. In particular, taking $\delta = \delta_n = n^{-\frac{1}{2\beta+d}}$, we deduce that $\log( |(\mathcal{D}\circ \phi)_{1/n}|)\leq C \log(n)^2 n^{\frac{d}{2\beta+d}}$, which matches the scaling of the covering number of the function class $\mathcal{F}^{d,\eta}_{\delta_n}$. We now turn to bounding the approximation error of the generator class.

\begin{proposition}\label{prop:deltaG}  Let $\mu^\star$ a probability measure satisfying Assumption \ref{assump:model}. Then, for $\delta_N=N^{-\frac{1}{2\beta+d}}$ and $\gamma>0$ and $\mathcal{G}\times \Phi \times \mathcal{A}$ we have that
    $$\mathbb{E}\left[\inf_{(g,\varphi,\alpha)\in \mathcal{G}\times  \Phi\times \mathcal{A}} d_{\mathcal{H}_1^{\gamma}}((F^{g,\varphi}\circ g)_{\# \alpha \gamma_d^n},\mu^\star)\right]\leq  C\log(n)^{C_2}(\delta_N^{(\beta + \gamma)\wedge (2\beta+1)}+n^{-1/2}).$$
\end{proposition}

The proof of Proposition~\ref{prop:deltaG} can be found in Section~\ref{sec:prop:deltaG}. This result shows that the  approximation error of the generator class behaves like \( \delta_N^{\beta + \gamma} \), where \( \delta_N \) is the frequency parameter of the class \( \mathcal{G} \), \( \beta \) denotes the regularity of the target distribution \( \mu^\star \), and \( \gamma \) the regularity of the IPM under consideration. In particular, the approximation improves as the regularity of both the target and the metric increases, as soon as \( \gamma \leq \beta+1 \). For \( \gamma \geq \beta + 1 \), the error no longer decreases with \( \gamma \), as it becomes limited by the $\mathcal{H}^{\beta+1}$-regularity of the supports of the measures. Let us now give a bound on the approximation error of the discriminator class.

\begin{proposition}\label{prop:DeltaD} Let $\mu^\star$ a probability measure satisfying Assumption \ref{assump:model}, $\gamma\in [1,d/2]$ and $\mathcal{F}\subset \mathcal{H}^1(\mathbb{R}^p,\mathbb{R})$. Then, supposing $n\in \mathbb{N}_{>0}$ large enough, we have with probability at least $n^{-1}$ that the estimators $\hat{\mu}=(F^{\hat{g},\hat{\varphi}}\circ \hat{g})_{\# \hat{\alpha}\gamma_d^n}$ from \eqref{eq:estimator} satisfies
$$d_{\mathcal{H}_1^{\gamma}}(\hat{\mu},\mu^\star)-d_{\mathcal{F}}(\hat{\mu},\mu^\star)\leq C\inf_{D\in \mathcal{F}}\|\nabla h_{\hat{\mu}}-\nabla D\|_\infty\log\left(1+d_{\mathcal{H}^{\gamma}_1}(\hat{\mu},\mu^\star)^{-1}\right)^{C_2} d_{\mathcal{H}^{\gamma}_1}(\hat{\mu},\mu^\star)^{\frac{\beta+1}{\beta+\gamma}},$$
for 
$$h_{\hat{\mu}} \in \argmax_{h\in \mathcal{H}^{\gamma}_1} \int h(x)d\hat{\mu}(x)-\int h(x)d\mu^\star(x).$$
\end{proposition}
The proof of Proposition \ref{prop:DeltaD} can be found in Section \ref{sec:prop:DeltaD}. This result shows that for a function class $\mathcal{F}$ approximating $\mathcal{H}_1^{\gamma}$, the difference between the IPMs of the two classes naturally depends both on the precision of the approximation of the class $\mathcal{F}$ and on the $d_{\mathcal{H}^{\gamma}_1}$ distance between $\hat{\mu}$ and $\mu^\star$. In particular, in the case $\mathcal{F}=\hat{\mathcal{F}}^{p,\gamma}_{\delta_n}$ and $d_{\mathcal{H}^{\gamma}_1}(\hat{\mu},\mu^\star)=n^{-\frac{\beta+\gamma}{2\beta+d}}$, we obtain 
$$d_{\mathcal{H}_1^{\gamma}}(\hat{\mu},\mu^\star)-d_{\mathcal{F}}(\hat{\mu},\mu^\star)\leq C\log(n)^{C_2} \delta_n^{\beta+\gamma},$$
which matches the bound given by Proposition \ref{prop:deltaG} in the case $\gamma\leq \beta+1$ up to logarithm factors.

\subsubsection{Minimax bounds}
Building on the precedent results, we now establish that the estimator \( \hat{\mu} \) achieves the minimax optimal convergence rate for the \( d_{\mathcal{H}^{d/2}_1} \) distance.

\begin{theorem}\label{theo:boundonddsur2}  Let $\mu^\star$ a probability measure satisfying Assumption \ref{assump:model}. Then, the estimator $\hat{\mu}=(F^{\hat{g},\hat{\varphi}}\circ \hat{g})_{\# \hat{\alpha}\gamma_d^n}$ from \eqref{eq:finalest} satisfies
   $$\mathbb{E}_{X_1,...,X_n\sim \mu^\star}[d_{\mathcal{H}^{d/2}_1}(\hat{\mu},\mu^\star)]\leq C \log(n)^{C_2} n^{-1/2}.$$
\end{theorem}

The proof of Theorem \ref{theo:boundonddsur2} can be found in Section \ref{sec:theo:boundonddsur2}. As a direct consequence of this result, we obtain that the estimator \( \hat{\mu} \) achieves the minimax optimal convergence rate (up to logarithmic factors) for all \( d_{\mathcal{H}^{\gamma}_1} \) distances with \( \gamma \geq d/2 \). Moreover, leveraging the fact that with high probability the estimator itself satisfies Assumption \ref{assump:model}, applying the interpolation inequality from Theorem~\ref{theo:theineqGAW}, we conclude that \( \hat{\mu} \) also attains the optimal rate for all \( \gamma \geq 1 \).

\begin{theorem}\label{theo:boundongamma} Let $\mu^\star$ a probability measure satisfying Assumption \ref{assump:model} and $n\in \mathbb{N}$ large enough. Then, with probability at least $n^{-1}$, the  estimator $\hat{\mu}=(F^{\hat{g},\hat{\varphi}}\circ \hat{g})_{\# \hat{\alpha}\gamma_d^n}$ from \eqref{eq:finalest} satisfies Assumption \ref{assump:model} with constant parameters $(\beta+1,C\log(n)^{C_2\mathds{1}_{\beta=\lfloor \beta \rfloor}})$. Furthermore, we have for all $\gamma\geq 1$
   $$\mathbb{E}_{X_1,...,X_n\sim \mu^\star}[d_{\mathcal{H}^{\gamma}_1}(\hat{\mu},\mu^\star)]\leq C \log(n)^{C_2} \left(n^{-\frac{\beta+\gamma}{2\beta+d}}\vee n^{-1/2}\right).$$
\end{theorem}

The proof of Theorem \ref{theo:boundongamma} can be found in Section \ref{sec:theo:boundongamma}. Since the rate \( O\left(n^{-\frac{\beta+\gamma}{2\beta+d}} \vee n^{-1/2}\right) \) has been proven to be minimax optimal in~\cite{tang2023minimax}, Theorem~\ref{theo:boundongamma} establishes the statistical optimality up to logarithmic factors of our estimator \( \hat{\mu} \) for the \( \gamma \)-Hölder IPMs across all \( \gamma \geq 1 \).

\section{Numerical experiments}\label{sec:numeric}
In this section, we detail how to implement a numerical scheme to approximate the estimator $\hat{\mu}:=(F^{\hat{g},\hat{\varphi}}\circ \hat{g})_{\# \hat{\alpha}\gamma_d^n}$, from \eqref{eq:hatmu}. The idea is to learn the optimal generators using stochastic gradient descent over a parametrization of the generators and discriminators classes.
We introduce a series of simplifications to the estimator in order to improve its computational efficiency. While these modifications  relax some of the technical components required for minimax optimality, they preserve the core structure and objectives of the estimator. In particular, the simplified version remains faithful to the geometric design of the model and its ability to capture complex manifold-supported distributions.

The main simplifications are as follows:

\begin{itemize}
    \item[i)] Instead of computing the two estimators $\hat{g}_i^1:B^d(0,\tau)\rightarrow \mathbb{R}^p$ and $\Psi^{-1}\circ \hat{g}_i^2:\mathbb{R}^d\rightarrow B^d(0,\tau)$, we directly compute a single estimator $\hat{g}_i:\mathbb{R}^d\rightarrow \mathbb{R}^p$.
    
    \item[ii)] The gluing of charts \( \hat{g}_i \) via the reconstruction maps \( F^{\hat{g}, \hat{\varphi}} \) is not performed at every step of the gradient descent, but only at the end, to reduce computational overhead.
    
    \item[iii)] Rather than using the approximation \( \hat{\phi} \) of the Daubechies scaling function \( \phi \), we employ a simpler function \( \theta \) as the mother function to construct the wavelet-like basis functions \( \hat{\psi}_{jlw} \).
    
    \item[iv)] For computational efficiency, we do not use the full wavelet dictionary \( \{ \hat{\psi}_{jlw}^p \} \) to construct discriminators. Instead, we allow the function classes to learn which wavelet components are most relevant.
\end{itemize}

These simplifications preserve the core geometric and generative structure of the estimator while significantly improving its computational speed and scalability. Detailed implementation choices and their justifications are provided in the next section.

\subsection{Improving the speed of computation}\label{sec:speedofcomp}
In this section, we provide the details of our implementation strategy and explain the motivations behind the simplifications~i)--iv) introduced earlier. These choices are guided by the goal of reducing computational complexity while preserving the core generative and geometric structure of the estimator.
\paragraph{i) Simplifying the class of generators} The generator class in the theoretical construction is derived from Proposition~\ref{prop:decompphi}, which ensures that the estimator possesses a density with \(\mathcal{H}^\beta\) regularity on its support. This is achieved by decomposing each generator as \( g_i = g_i^1 \circ \Psi^{-1} \circ g_i^2 \), with constraints on both components to control the Jacobian and guarantee invertibility. However, enforcing this structure is computationally expensive, as it requires estimating two distinct maps per generator and controlling their spectral properties. To improve efficiency in practice, we simplify the generator class by directly learning unconstrained maps \( g_i: \mathbb{R}^d \to \mathbb{R}^p \), without explicitly enforcing the factorization or invertibility.

The resulting practical generator class is defined as:
\begin{equation}\label{eq:classofgenepractice}
    \mathcal{G} = \left\{\big(g_i\big)_{i=1}^m\ \Big|\ g_i^1\in \left(\hat{\mathcal{F}}_{\delta_N,4\log(n)^{1/2}}^{d,\beta+1}\right)^{p}\right\}.
\end{equation}

The limitation of using this class is that we can no longer guarantee that the estimator admits a \( \beta \)-Hölder density with respect to its support, even though the support itself remains of regularity \( (\beta+1) \). This loss of control arises because, although the transport maps $g_i$ are \( (\beta+1) \)-Hölder continuous, the regularity of the pushforward densities \( (g_i)_{\#} \gamma_d \) depends on the inverse maps \( g_i^{-1} \), for which we have no direct smoothness bounds. 
Another benefit of the simplified construction is that it eliminates the need to compute the gradient constraint term \( R(g) \) defined in~\eqref{eq:R}. 
Importantly, across all numerical experiments we conducted, the simplified estimator consistently matched the empirical performance of the more theoretically structured version, confirming its practical effectiveness.
\paragraph{ii) Constructing the manifold only in the final steps}
We observed that delaying the estimation of the inverse maps \( \hat{\varphi}_i \) until the later stages of training does not degrade the overall performance. As a result, we begin by estimating only the generators \( \hat{g}_i \), ignoring the gluing maps during the initial optimization phase. Specifically, we first compute
\begin{equation}\label{eq:estimatorapp2}
    \hat{g}:= \argmin_{g\in \mathcal{G}}\ \sup_{D\in \mathcal{D}} 
    \frac{1}{N} \sum_{j=1}^N D( g_{T_j}(Y_j))-\frac{1}{n} \sum_{j=1}^n D(X_j),
\end{equation}
via gradient descent over parametrized versions of \( \mathcal{G} \) and \( \mathcal{D} \) (details provided below). Once a good initialization for the generator has been obtained, we proceed to estimate the full model in~\eqref{eq:Ln} by resuming gradient descent, initializing both the generators and discriminators with the solution from~\eqref{eq:estimatorapp2}. This two-phase training approach significantly accelerates computation, as the gradients required for optimizing~\eqref{eq:estimatorapp2} are less expensive to compute than those involving the manifold reconstruction maps in~\eqref{eq:Ln}.

\paragraph{iii) Simplifying the wavelet.}
The Daubechies wavelet, while theoretically well-suited for representing functions in Besov and Hölder spaces, can be highly oscillatory. This oscillatory behavior makes it challenging to obtain accurate estimators for the generator class \( \mathcal{G} \) and discriminator class \( \mathcal{D} \) without requiring a large number of frequency components.
In practice, we have observed that better results can be achieved by replacing the Daubechies wavelet with a smoother, less oscillatory function. For this reason, we adopt the following Gaussian-shaped scaling function:
\begin{equation}\label{eq:scalinginpractice}
    \theta(x) := \exp(-x^2),
\end{equation}
which serves as a smooth surrogate for the Haar scaling function \( \phi_{\mathcal{H}}(x) = \mathds{1}_{\{x \in [-1,1]\}} \). This choice retains localization in space while avoiding the rapid oscillations of classical wavelets, leading to improved numerical stability and performance.

\paragraph{iv) Reducing the number of wavelets} For the discriminator, it is not necessary to control its behavior across the entire ambient domain \( B^{p}(0, K) \), since we are primarily concerned with its values on the supports of the target distribution and the estimator. Instead of using the full grid of wavelet basis functions \( \hat{\psi}_{j,l,w} \) indexed by
\[
j \in \{0, \dots, \log(\delta_n^{-1})\}, \quad l \in \{1, \dots, 2^p - 1\}, \quad w \in \{-K2^j, \dots, K2^j\}^p,
\]
we propose a data-driven alternative in which the network learns a reduced set of wavelet-like components via gradient descent. To this end, we define the following class of functions, based on the smooth scaling function \( \theta \) introduced in~\eqref{eq:scalinginpractice}: 
\begin{align}\label{eq:practical classes}
\mathcal{R}_{L_1,L_2}^{k}=\Big\{\sum \limits_{j=1}^{L_1}  \sum \limits_{i=1}^{L_2} \hat{\alpha}(j,i) \theta(\langle A_{ij},\cdot-b_{ij}\rangle) \text{ with } A_{ij},b_{ij}\in \mathbb{R}^{k} \text{ and } \|A_{ij}\|\leq 2^{j},|\hat{\alpha}_{ij}|\leq 2^{-j}\Big\}.
\end{align}
Here, \( L_1 \) represents the number of frequency scales, and \( L_2 \) the number of wavelets per scale. The parameters \( A_{ij}, b_{ij} \), and the weights \( \hat{\alpha}(j,i) \) are learned jointly via gradient descent. This formulation significantly reduces the number of parameters compared to the full wavelet class \( \hat{\mathcal{F}}_\delta^{k,\eta} \) defined in~\eqref{eq:fdeltaperapprox}, while retaining sufficient flexibility to discriminate between the distributions.

Since the support of the data lies on a \( d \)-dimensional manifold, it suffices to scale the number of wavelets accordingly. In particular, we can take \( L_2 = O(n^{\frac{d}{2\beta + d}}) \), as is done for the generator class, rather than the larger \( O(n^{\frac{p}{2\beta + d}}) \) that would be required in the full ambient space.

\subsection{Experiments}
We evaluate the performance of our estimator by comparing it against the following benchmark generative models:
\begin{itemize}
    \item \textbf{Wasserstein GAN (WGAN)} \citep{arjovsky2017wasserstein}, which serves as the foundation for our method. Our estimator builds upon WGAN to better handle more complex data topologies.
    \item \textbf{Score-based Generative Models (SGMs)} \citep{song2020score}, widely regarded as state-of-the-art generative models. SGMs have demonstrated superior performance compared to GANs, particularly on highly complex datasets \citep{dhariwal2021diffusion}.
\end{itemize}

\subsubsection{Non uniform density on the sphere}
We begin with a simple yet non-trivial example where all methods are expected to perform well. Specifically, we consider the unit \( 2 \)-dimensional sphere \( \mathbb{S}^2 \subset \mathbb{R}^3 \), and aim to estimate a probability measure \( \mu^\star \) defined on it. The density of \( \mu^\star \) with respect to the uniform surface measure \( \text{Vol}_{\mathbb{S}^2} \) is given by
\[
\frac{d\mu^\star}{d\text{Vol}_{\mathbb{S}^2}}(x, y, z) = \frac{1 + \delta(x)}{C_{\mu^\star}},
\]
where $C_{\mu^\star}$ is a normalizing constant and $\delta:[-1,1]\rightarrow [0,1]$ is a smooth cut-off function such that $\delta([-1,0])=\{0\}$ and $\delta([0.1,1])=\{1\}$. This setting introduces a mild asymmetry in the density while retaining a simple geometric structure, making it a good testbed for evaluating basic representational capabilities of each method.

We train all methods using \( n = 2000 \) samples drawn from \( \mu^\star \). To evaluate performance, we compute an approximate Wasserstein distance between each learned estimator and the ground truth distribution \( \mu^\star \), using a reference set of 10,000 points. Each experiment is repeated across three independent trials, and we report the average Wasserstein distance below. The architectures of the neural networks are chosen relatively small to ensure efficiency, while still offering sufficient capacity to accurately approximate the target distribution.

\textbf{WGAN:} The generator is a 2-layers LeakyReLU neural network with 128 neurons each. The Discriminator is a 3-layers LeakyReLU neural network with 128 neurons each. The Wasserstein distance between the final estimator and the target measure is \textbf{0.061}.

\textbf{SGM:} The score network is a 3-layers LeakyReLU neural network with 128 neurons each. The Wasserstein distance between the final estimator and the target measure is \textbf{0.082}.

\textbf{Our estimator:} The model uses two generators, each modeled as \( \left(\mathcal{R}^2_{3,500}\right)^3 \), and a discriminator  of the form $\mathcal{R}^3_{3,2000}$ with   $\mathcal{R}^k_{L_1,L_2}$ defined in \eqref{eq:practical classes}. The Wasserstein distance between the final estimator and the target measure is \textbf{0.056}.

 \begin{figure}[H]
    \centering
    \begin{tabular}{ccc}
        \includegraphics[width=0.32\textwidth]{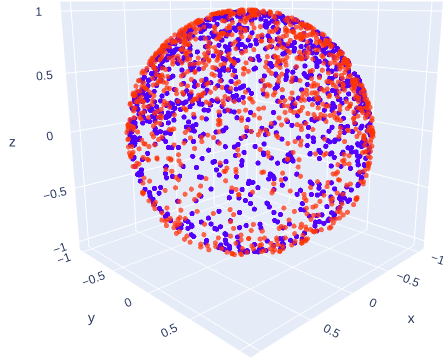} &
        \includegraphics[width=0.32\textwidth]{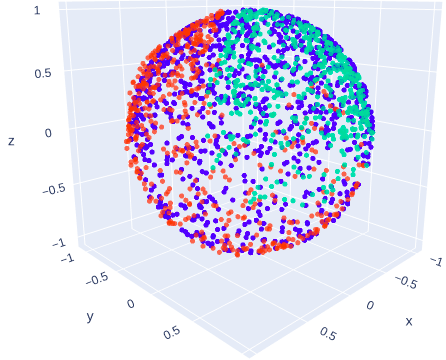} &
        \includegraphics[width=0.32\textwidth]{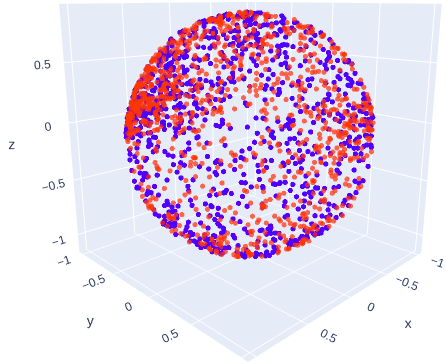} \\
        \texttt{$\quad \ $ WGAN} & \texttt{$\quad \ $ Our estimator} & \texttt{$\quad \ $ SGM} \\
    \end{tabular}
    \caption{Comparison of the estimators for a non-uniform density on the sphere. The data are represented in dark blue and the other points are generated by the estimators. In particular, our estimator generates both red and light blue points as it uses two different generators.}
    \label{fig:sawtooth}
\end{figure}

All three methods achieve excellent final performance on this task. However, tuning the noise schedule for the SGM proved to be particularly challenging, which may partially account for its slightly lower accuracy compared to our estimator. Additionally, it is worth noting that both the generator and discriminator architectures used in our method rely on significantly fewer parameters than those employed in the WGAN and SGM models.

\subsubsection{Uniform density on the torus}
To evaluate how the estimators perform on data with nontrivial topology, we consider the problem of estimating the uniform distribution on a \( 2 \)-dimensional torus embedded in \( \mathbb{R}^3 \) with a minor radius of 1 and a major radius of 2.5. We train all methods using \( n = 3000 \) samples from the uniform measure on the torus. To assess performance, we approximate the Wasserstein distance between each estimator and the true distribution \( \mu^\star \), using 10,000 reference points. Results are averaged over 3 independent trials.

\textbf{WGAN:} The generator is a 3-layers LeakyReLU neural network with 128 neurons each. The Discriminator is a 3-layers LeakyReLU neural network with 256 neurons each. The Wasserstein distance between the final estimator and the target measure is \textbf{0.35}.

\textbf{SGM:} The score network is a 3-layers LeakyReLU neural network with 256 neurons each. The Wasserstein distance between the final estimator and the target measure is \textbf{0.22}.

\textbf{Our estimator:} The two generators are of the form $\left(\mathcal{R}^2_{3,1500}\right)^3$ and the discriminator is of the form $\mathcal{R}^3_{3,3000}$ with   $\mathcal{R}^k_{L_1,L_2}$ defined in \eqref{eq:practical classes}. The Wasserstein distance between the final estimator and the target measure is \textbf{0.17}.

 \begin{figure}[H]
    \centering
    \begin{tabular}{ccc}
        \includegraphics[width=0.32\textwidth]{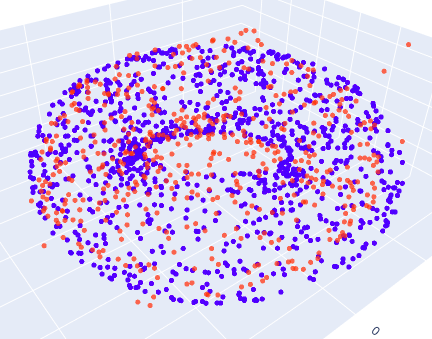} &
        \includegraphics[width=0.32\textwidth]{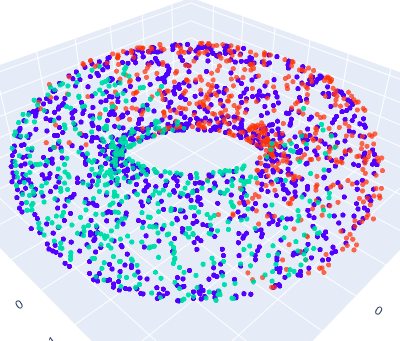} &
        \includegraphics[width=0.32\textwidth]{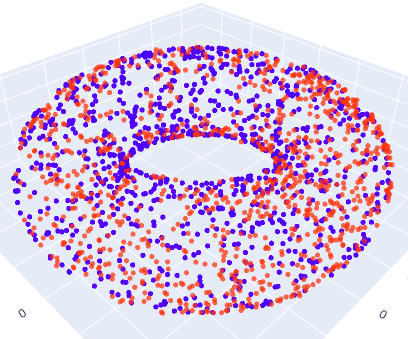} \\
        \texttt{$\quad \ $ WGAN} & \texttt{$\quad \ $ Our estimator} & \texttt{$\quad \ $ SGM} \\
    \end{tabular}
    \caption{Comparison of the estimators for the uniform measure on the torus. The data are represented in dark blue and the other points are generated by the estimators.}
    \label{fig:sawtooth2}
\end{figure}

We observe that the WGAN struggles to capture the underlying topology of the torus, underscoring its limitations in representing complex nontrivial geometries. In contrast, our estimator achieves a Wasserstein error comparable to that observed in the sphere (the data radius is 2.5 times larger). For the SGM, we encountered the same difficulty as in the sphere setting: tuning the noise schedule was challenging and likely contributed to the method's reduced performance. Although the SGM effectively recovers the support of the distribution, it exhibits difficulty in spreading the probability mass uniformly, leading to suboptimal approximation of the uniform measure.

\subsubsection{MNIST dataset}
We now evaluate our estimator on the MNIST dataset, which consists of images of handwritten digits. Our objective is to assess how well the estimator adapts to data that do not lie on a strict manifold but still exhibit a lower intrinsic dimension. To achieve strong performance, we adapt our estimator to the specific structure of image data. The following modifications are implemented to tailor the estimator to this setting:
\paragraph{Using a convolutional layer in the discriminator}  
Convolutional layers have been shown to achieve significantly better performance than fully connected neural networks \citep{krizhevsky2012imagenet,shalev2020computational}. Motivated by this, we define the discriminator class as the composition of the class $\mathcal{R}_{L_1,L_2}^{k}$ from \eqref{eq:practical classes} with a single convolutional layer. This design reduces the number of parameters compared to using only the model \eqref{eq:practical classes} while maintaining strong feature extraction capabilities.

\paragraph{Constraining the generators to certain classes}  
We have observed that during the training, some generators attempt to produce all digit classes. This leads to poor sample quality as each generator has relatively few parameters. To better constrain each generator to specific digit classes, we modify the discriminator to output values in \(\mathbb{R}^m\) instead of \(\mathbb{R}\) when they are $m$ generators. In the loss computation, we use \(D_i(g_i)\) instead of \(D(g_i)\) for each generator \(i \in \{1, \dots, s\}\), ensuring that each generator is evaluated independently based on its assigned class. To summarize, we partition the set $\{0,1,...,9\}$ into $m$ classes $(C_i)_{i=1,...,m}$ and instead of the loss function \eqref{eq:estimatorapp2}, we use \begin{equation}
    \hat{g}:= \argmin_{g_1,...,g_m\in \mathcal{R}_{L_1,L_2}^{d}}\ \sup_{D\in (\mathcal{R}_{L_1^D,L_2^D}^{p})^m\circ\ \mathcal{C}_{L_1^C,L_2^C}} \
    \frac{1}{n} \sum_{j=1}^n D_{T_j}( g_{T_j}(Y_j))- D_{T_j}(X_{T_j}),
\end{equation}
with $Y_j\sim \gamma_d$, $T_j\sim \mathcal{U}(\{1,...,m\})$, $X_{T_j}$ that follows the uniform law over the class of images of the digit in $C_{T_j}$ and $\mathcal{C}_{L_1^C,L_2^C}$ the convolutional layer with $L_1^C$ output channels,  and kernels of sizes $L_2^C$. We take the number of classes $m=3$ with
$C_1 = \{1,7,2\}$, $C_2=\{0,8,3,6\}$, $C_2 = \{4,5,9\}$ and the latent dimension $d=10$. 
 
 \begin{figure}[H]
    \centering
        \includegraphics[width=0.62\textwidth]{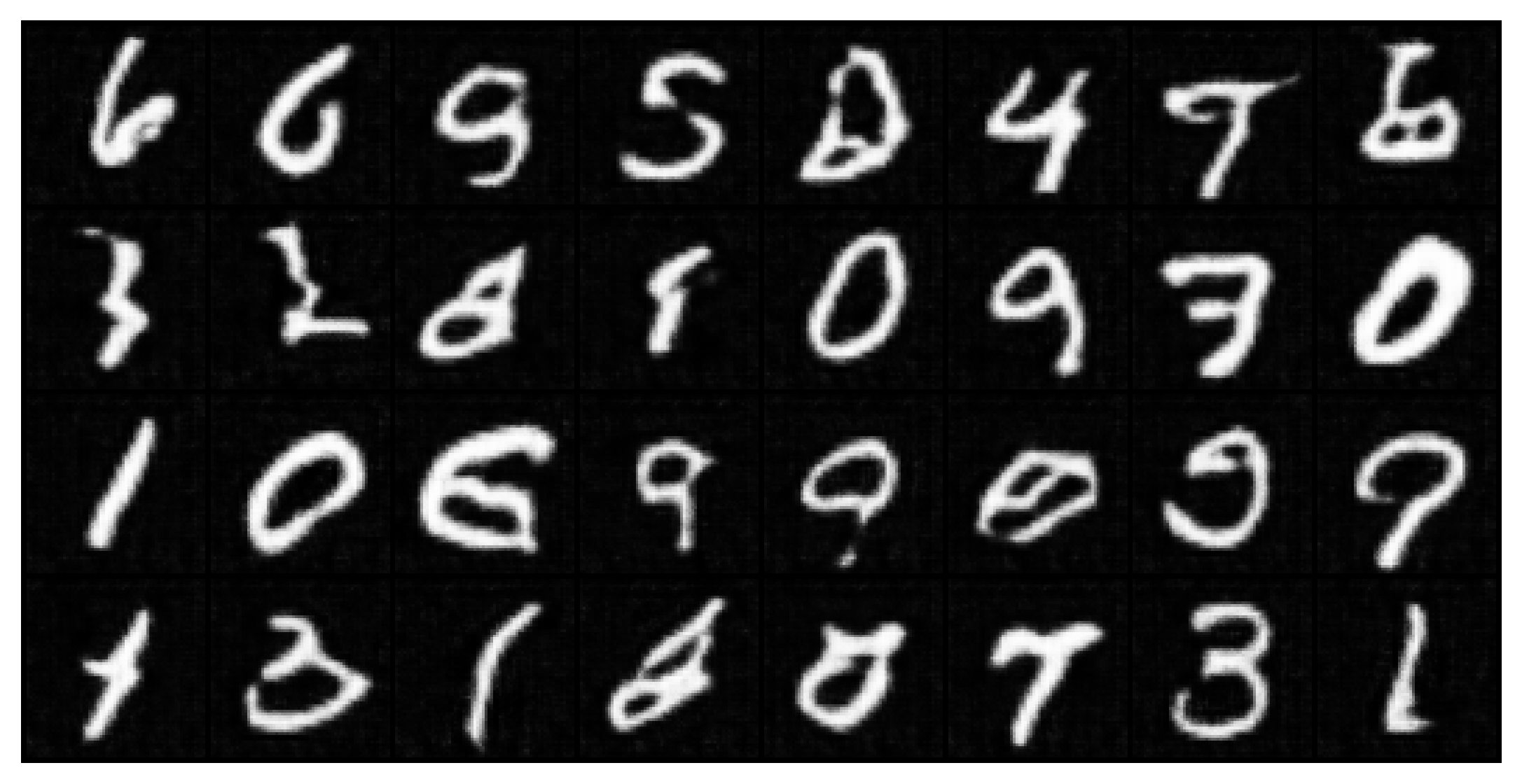} 
    \caption{Sample generated by our estimator after being trained on the MNIST dataset}
    \label{fig:sawtooth3}
\end{figure}

We observe that our estimator achieves good performance on MNIST, comparable to results previously reported for WGANs~\citep{cheng2020analysis} and diffusion models~\citep{wang2025application}. Notably, the use of multiple generators enables us to operate with a significantly smaller latent dimension \( d = 10 \) as opposed to the typical choice of \( d \approx 100 \) in standard GAN architectures. This reduction in dimensionality highlights the efficiency and representational flexibility of our approach.

\bibliography{bib}
\appendix
\vspace{0.4cm}

 
\section{Properties of the estimator}
In this Section we gather the proof of the results on the regularity of the estimators.

\subsection{Technical lemmas}
We first provide some technical results allowing to show that the measure \eqref{eq:hatmu} has a density with respect to a submanifold.
For \( (g, \varphi,\alpha,D) \in \mathcal{G} \times \Phi \times \mathcal{A}\times \mathcal{D} \), define the limit when $N,n\rightarrow \infty$ of the losses  $$L(g,\varphi,\alpha,D)=\frac{1}{m}\sum_{i=1}^m \alpha_i \mathbb{E}_{Y\sim \gamma_d^n}[ D(F^{g,\varphi}\circ g_{i}(Y))]-\mathbb{E}_{X\sim \mu^\star}[ D(X)],$$
$$
\mathcal{C}(g,\varphi)=\sum_{i,j=1}^m\mathbb{E}_{U\sim \mathcal{U}(B^d(0,2\tau))}[\|g_{j}^1(U)-g_i^1\circ \varphi_i\circ g_{j}^1(U)\|\Gamma\big((\|g_j^1(U) - g_i^1(0)\|-\tau^2)\vee 0\big)].
$$

\subsubsection{Properties arising from a bound on $\mathcal{C}(g,\varphi)$ and $R(g)$}
Let us first show that a bound on $\mathcal{C}(g,\varphi)$ gives a bound on $\|g_i^1-g_{j}^1\circ\varphi_j\circ g_{i}^1\|_\infty$ for $i,j\in \{1,...,m\}$.
\begin{lemma}\label{lemma:secondr} There exists $C_\star>0$ such that for all $\delta\in (0,C_\star^{-1})$, $(g,\varphi)\in \mathcal{G}\times  \Phi$, if $\mathcal{C}(g,\varphi)\leq C_\star^{-1}\delta^{d+1}$, then  
    for all $u\in B^d(0,2\tau)$ and $i,j\in \{1,...,m\}$, we have
    $$\|g_i^1(u)-g_{j}^1\circ\varphi_j\circ g_{i}^1(u)\|\mathds{1}_{\{\|g_{i}^1(u)-g_j^1(0)\|\leq \tau(1+(K+3)\tau)-\delta\}}\vee\|g_i^1(u)-F^{g,\varphi}_{j}\circ g_{i}^1(u)\|\leq \delta.$$
\end{lemma}

\begin{proof}
Suppose there exists  $u\in B^d(0,2\tau)$ and $i,j\in \{1,...,m\}$ such that $$\|g_i^1(u)-g_j^1\circ\varphi_j\circ g_i^1(u)\|\mathds{1}_{\{\|g_i^1(u)-g_j^1(0)\|\leq \tau(1+(K+3)\tau)-\delta\}}> \delta.$$ Take $R>0$ and $v\in B^d(0,2\tau)\cap B^d(u,\frac{\delta}{R})$, we have 
\begin{align*}
    \|g_i^1(v)-g_j^1\circ\varphi_j\circ g_i^1(v)\| & \geq \|g_i^1(u)-g_j^1\circ\varphi_j\circ g_i^1(u)\| -\|g_i^1(u)-g_i^1(v)\| - \|g_j^1\circ\varphi_j\circ g_i^1(v)-g_j^1\circ\varphi_j\circ g_i^1(u)\|\\
    & > \delta/2,
\end{align*}
taking $R>0$ large enough and recalling that $g_i^1,g_j^1$ and $\varphi_j$ are Lipschitz from Proposition \ref{prop:regularityfandg}. Then, recalling the definition of the smooth cutoff function $\Gamma$ in \eqref{eq:Gamma}, we obtain 
\begin{align*}
    \mathcal{C}(g,\varphi)& =\sum_{a,b=1}^m\mathbb{E}_{U\sim \mathcal{U}(B^d(0,2\tau))}\Big[\|g_{a}^1(U)-g_b^1\circ\varphi_b\circ g_{a}^1(U)\|\Gamma\big((\|g_a^1(U) - g_b^1(0)\|-\tau^2)\vee 0\big)\Big]\\
    & \geq C^{-1}\int_{B^d(0,2\tau)\cap B^d(u,\frac{\delta}{R})}  \|g_i^1(v)-g_j^1\circ\varphi_j\circ g_i^1(v)\|\mathds{1}_{\{\|g_i^1(v)-g_j^1(0)\|\leq \tau(1+(K+3)\tau)\}} d\lambda^d(v)\\
    & \geq C^{-1}R^{-d}\delta^{d+1}\\
    &\geq C_\star^{-1}\delta^{d+1},
\end{align*}
for a certain constant $C_\star>0$ independent of $\delta$.
We deduce that $\mathcal{C}(g,\varphi)\leq C_\star^{-1}\delta^{d+1}$ implies that
$$
\sup_{u\in B^d(0,2\tau)}\|g_i^1(u)-g_j^1\circ\varphi_j\circ g_i^1(u)\|\mathds{1}_{\{\|g_i^1(u)-g_j^1(0)\|\leq \tau(1+(K+3)\tau)-\delta\}}\leq \delta.
$$

Therefore, it also implies that for all $u\in B^d(0,2\tau)$,
\begin{align*}
    \|g_i^1(u)-F^{g,\varphi}_{j}\circ g_i^1(u)\| & = \Gamma\big(\|g_i^1(u)-g_j^1(0)\|\big) \|g_j^1\circ \varphi_j(g_i^1(u))-g_i^1(u)\|\\
    & \leq \delta.
\end{align*}
Indeed, if $\|g_i^1(u)-g_j^1(0)\|\leq \tau(1+(K+3)\tau)-\delta$ then $\|g_j^1\circ \varphi_j(g_i^1(u))-g_i^1(u)\|\leq \delta$ and otherwise, taking $\delta\leq C_\star^{-1}$ for $C_\star>0$ large enough, we have $$\Gamma\big(\|g_i^1(u)-g_j^1(0)\|\big) =0,$$  
as $\Gamma([\tau(1+(K+2)\tau)+\epsilon_\Gamma,\infty))=\{0\}$ and $\epsilon_\Gamma\in (0,\tau^2/4)$. 
\end{proof}

From Proposition \ref{prop:regularityfandg}, we have that $F^{g,\varphi}_{k}$ is Lipschitz so under the assumption of  Lemma \ref{lemma:secondr}, as
\begin{align*}
    \| g_i^1(u)-F^{g,\varphi}_{k}\circ F^{g,\varphi}_{j} \circ g_i^1(u)\|& \leq \|g_i^1(u)-F^{g,\varphi}_{k}\circ g_i^1(u)\| + \|F^{g,\varphi}_{k}\circ  g_i^1(u)-F^{g,\varphi}_{k}\circ F^{g,\varphi}_{j} \circ g_i^1(u)\|\\
    & \leq \delta+ C\delta,
\end{align*}
we deduce by induction the following result.

 \begin{lemma}\label{lemma:secondrbis} There exists $C_\star>0$ such that for all $\delta\in (0,C_\star^{-1})$, $(g,\varphi)\in \mathcal{G}\times  \Phi$, if $\mathcal{C}(g,\varphi)\leq C_\star^{-1}\delta^{d+1}$, then  
    for all $u\in B^d(0,\tau)$ and $i,j_1,...,j_k\in \{1,...,m\}$, with $k\in \{1,...,m\}$ we have
    $$\|g_i^1(u)-F^{g,\varphi}_{j_k}\circ...\circ F^{g,\varphi}_{j_1} \circ g_i(u)\|
    \leq \delta.$$
\end{lemma}

Let us now show that having $R(g)=0$ implies that the differentials of the $g_i^1$'s have upper and lower bounded eigenvalues.
\begin{lemma}\label{lemma:thirdr} There exists $C_\star>1$ such that for $\epsilon_\Gamma \in (C_\star^{-2},C_\star^{-1})$, $g\in \mathcal{G}$ satisfying $R(g)=0$ (with $R$ depending on $\epsilon_\Gamma$ defined in \eqref{eq:R}) and
     $u,v\in B^d(0,\tau)$ and $i\in \{1,...,m\}$, we have
    $$1-(K+1)\tau\leq \|\nabla g_i^1(u)\frac{v}{\|v\|}\|\leq 1+(K+1)\tau.$$
   Likewise, for all $x\in B^d(0,\log(n)^{1/2}+1)$ we have
   $$C^{-1}\leq \|\nabla g_i^2(x)\frac{v}{\|v\|}\|\leq C.$$
\end{lemma}

\begin{proof}
We give the proof of the bounds on $\nabla g_i^1$, the bounds on $\nabla g_i^2$ follow the same strategy.
Suppose that there exist $u\in B^d(0,\tau)$, $v\in \mathbb{S}^{d-1}$ and $i\in \{1,...,m\}$,  such that 
    $\|\nabla g_i^1(u)v\|\geq 1+(K+1)\tau.$ From Proposition \ref{prop:regularityfandg} we have $g_i^1\in \mathcal{H}^2_C$,   so for $x=u -2\sqrt{\epsilon_\Gamma} \frac{u }{\|u\|}$ we obtain
    \begin{align*}
        \|g_i^1(x)-g_i^1(x+ 2\sqrt{\epsilon_\Gamma} v)\| &=2\sqrt{\epsilon_\Gamma}\|\int_0^1 \nabla g_i^1(x+2s\sqrt{\epsilon_\Gamma} v) vds\|\\
        & = 2\sqrt{\epsilon_\Gamma}\|\left( \nabla g_i^1(u)+\int_0^1\int_0^1 \nabla^2 g_i^1(u+t(x+s\sqrt{\epsilon_\Gamma} v-u))(x+2s\sqrt{\epsilon_\Gamma} v-u) dt ds\right) v\|\\
        & \geq 2\sqrt{\epsilon_\Gamma}( 1+(K+1)\tau -C \sqrt{\epsilon_\Gamma}).
    \end{align*}
    
    As the set $\mathcal{A}$ from \eqref{eq:R} is an $\epsilon_\Gamma$ covering of $B^d(0,\tau)$, there exists $z_1,z_2\in \mathcal{A}$ such that $\|z_1-x\|+\|z_2-x+ 2\sqrt{\epsilon_\Gamma} v\|< 2\epsilon_\Gamma$. Therefore, 
\begin{align*}
\|g_i^1(z_1)-g_i^1(z_2)\| & \geq \|g_i^1(x)-g_i^1(x+2 \epsilon_\Gamma v)\| -C\epsilon_\Gamma \geq 2\sqrt{\epsilon_\Gamma}( 1+(K+1)\tau -C \sqrt{\epsilon_\Gamma})\\
& \geq (1+(K+1)\tau)(\|z_1-z_2\|-2\epsilon_\Gamma) -C \epsilon_\Gamma.
\end{align*}
Taking $\epsilon_\Gamma>0$ small enough, we get a contradiction with $R(g)=0$. The other inequality follows the same reasoning.
\end{proof}

From Lemma \ref{lemma:thirdr} we deduce a lower bound on the gradient of $F^{g,\varphi}\circ g_i^1$.

\begin{proposition}\label{prop:diffeofrondg}
    There exists $C_1,C_2>0$ such that for $\epsilon_\Gamma \in (0,C_1^{-1})$ and $(g,\varphi)\in \mathcal{G}\times  \Phi$, if $ \mathcal{C}(g,\varphi)+R(g)\leq C_2^{-1}$, then for all $i,j_1,...,j_k\in \{1,...,m\}$ with $k\in \{1,...,m\}$, we have that $F^{g,\varphi}_{j_k}\circ...\circ F^{g,\varphi}_{j_1} \circ g_i^1$ is a $C^{\beta+1}$ diffeomorphism from $B^d(0,\tau)$ onto its image satisfying 
    $$
    \min_{u\in B^d(0,\tau)}\min_{v\in \mathbb{S}^{d-1}} \|\nabla \big(F^{g,\varphi}_{j_k}\circ...\circ F^{g,\varphi}_{j_1} \circ g_i^1\big)(u)v\|\geq 1-(K+2)\tau.
    $$ 
\end{proposition}

\begin{proof} Let us write $G_{k,i}=F^{g,\varphi}_{j_k}\circ...\circ F^{g,\varphi}_{j_1} \circ g_i^1$, take $A>0$ and suppose that there exists $u\in B^d(0,\tau)$ and $v\in \mathbb{S}^{d-1}$ such that
$$
\|\nabla (g_i^1-G_{k,i})(u)\frac{v}{\|v\|}\|\geq A^{-1}.
$$
Then as from Proposition \ref{prop:regularityfandg}, $g_i^1$ and $G_{k,i}$ belong to $\mathcal{H}^1_C$, for $x=u-\frac{u}{A^2\|u\|}$ we obtain
\begin{align*}
    \|g_i^1(x)-G_{k,i}(x)\|& \geq \|g_i^1(x)-G_{k,i}(x)-(g_i^1(x+A^2v)-G_{k,i}(x+A^2v))\|- \|g_i^1(x+A^2v)-G_{k,i}(x+A^2v)\|\\
    & \geq \|\int_0^1 \nabla (g_i^1-G_{k,i})(x+tA^{-2}v)A^{-2}vdt\|- A^{-4},
\end{align*}
supposing $ \mathcal{C}(g,\varphi)\leq C_2^{-1}$ with $C_2>0$ large enough and using Lemma \ref{lemma:secondrbis}.
Furthermore, as $g_i^1-G_{k,i}\in \mathcal{H}^2_C$ we have
\begin{align*}
    \|\int_0^1 \nabla (g_i^1-G_{k,i})(x+tA^{-2}v)A^{-2}vdt\| & \geq A^{-2}\|\nabla (g_i^1-G_{k,i})(u)\frac{v}{\|v\|}\| - CA^{-4}\\
    \geq A^{-3}-CA^{-4},
\end{align*}
so we deduce that 
$$\|g_i^1(x)-G_{k,i}(x)\|\geq A^{-3}-CA^{-4}.$$
Taking $A>0$ and $C_2>0$ large enough, this is a contradiction with the result of Lemma \ref{lemma:secondrbis}.

 Therefore,
\begin{align*}
\inf_{u\in B^d(0,\tau)}\inf_{v\in \mathbb{R}^d}\|\nabla G_{k,i}(u)\frac{v}{\|v\|}\|&\geq\inf_{u\in B^d(0,\tau)}\inf_{v\in \mathbb{R}^d}\|\nabla g_i^1(u)\frac{v}{\|v\|}\|- A^{-1}\\
&\geq 1-(K+2)\tau,
\end{align*}
taking $A,C_1,C_2>0$ large enough and using Lemma \ref{lemma:thirdr}.
\end{proof}

As the class $\mathcal{D}$ from \eqref{eq:classesofdis}, well approximates the class $\mathcal{H}^{d/2}_1$, we then obtain a bound on the quantity $d_{\mathcal{H}^{d/2}_1}(\mu_1,\mu_2)-d_{\mathcal{D}}(\mu_1,\mu_2)$.

\begin{lemma}\label{lemma:Landdhgamma} For $\mu_1,\mu_2$ two measures on $\mathbb{R}^d$ and $\delta_n:=n^{-\frac{1}{2\beta+d}}$, we have 
$$d_{\mathcal{D}}(\mu_1,\mu_2)\geq d_{\mathcal{H}^{d/2}_1}(\mu_1,\mu_2)-C\Big(\mu_1(\mathbb{R}^d)+\mu_2(\mathbb{R}^d)\Big) \delta_n^{d/2}\log(\delta_n^{-1})^2.$$
\end{lemma}

\begin{proof}
   Let $$D^\star \in \argmax_{D\in \mathcal{H}_1^{d/2}}\int D(x)d\mu_1(x)-\int D(x)d\mu_2(x)$$
and 
$$\overline{D}=\sum \limits_{j=0}^{\log_2(\delta_n^{-1})} \sum \limits_{l=1}^{2^p} \sum \limits_{z\in \{-K2^j,...,K2^j\}^p} \alpha_{D^\star}(j,l,w) \hat{\psi}^p_{jlw}\in \mathcal{D}.$$
Using Lemma \ref{lemma:bouundcoefwav}, we obtain that 
\begin{align*}
\|\overline{D}-D^\star\|_{\infty}&\leq \|\overline{D}-D^\star\|_{\mathcal{B}^{0,2}_{\infty,\infty}}=\sup_{j,l,z} |\alpha_{D^\star}(j,l,z)-\alpha_{\overline{D}}(j,l,z)|2^{jp/2}j^2\\
& \leq C \delta_n^{d/2}\log(\delta_n^{-1})^2.
\end{align*}
Then
\begin{align*}
    d_{\mathcal{D}}(\mu_1,\mu_2)&\geq \int \overline{D}(x)d\mu_1(x)-\int \overline{D}(x)d\mu_2(x)\\
    &\geq d_{\mathcal{H}^{d/2}_1}(\mu_1,\mu_2)-C\Big(\mu_1(\mathbb{R}^d)+\mu_2(\mathbb{R}^d)\Big) \delta_n^{d/2}\log(\delta_n^{-1})^2.
\end{align*}
\end{proof}

\subsubsection{Properties arising from a bound on $L(g,\varphi,\alpha,D)$, $\mathcal{C}(g,\varphi)$ and $R(g)$}
Let us start by giving an estimate on the tails of the densities of the measures  $(\Psi^{-1}\circ g_j^2)_{\#} \gamma_d^n$.
\begin{lemma}\label{lemma:boundetaj}
For all $g\in \mathcal{G}$ and $j\in \{1,...,m\}$ let us write $\zeta_j^g$ for the density of the probability measure $(\Psi^{-1}\circ g_j^2)_{\#} \gamma_d^n$. Then, supposing that $R(g)=0$, we have that the restriction of $g_j^2$ to $B^d(0,\log(n)^{1/2}+1)$ is a diffeomorphism and $\zeta_j^g\in \mathcal{H}^\beta_{C\log(n)^{C_2\mathds{1}_{\beta=\lfloor \beta \rfloor}}}$. Furthermore,  for all $u\in B^d(0,\tau)$  we have
$$\zeta_j^g(u)\leq C\exp(-C^{-1}(\tau-\|u\|)^{-2})$$
and there exists $C_\star>0$ such that if $\|\psi(u)\|\leq C_\star^{-1}\log(n)^{1/2}$, we have
$$C^{-1}\exp(-C(\tau-\|u\|)^{-2})\leq \zeta_j^g(u).$$
\end{lemma}

\begin{proof}
The fact that the restriction of $g_j^2$ to $B^d(0,\log(n)^{1/2}+1)$ is a diffeomorphism on its image is an immediate consequence of Lemma \ref{lemma:thirdr}. 

Let us now show the upper bound on $\zeta_j^g$. In the case where $\Psi(u)\notin g_j^2(B^d(0,\log(n)^{1/2}+1))$, we have $\zeta_j^g(u)=0$ so let us focus on the case where $\Psi(u)$ belongs to $g_j^2(B^d(0,\log(n)^{1/2}+1))$. Recalling that
$$\Psi(u):= \frac{u}{\tau^2-\|u\|^2},$$
we have
\begin{align}\label{align:cvghjhgfghbp}
    \zeta_j (u)=&|\det\big(\nabla( (g_j^2)^{-1}\circ \Psi) (u)\big)|\gamma_d^n\big((g_j^2)^{-1}\circ \Psi) (u)\big)\nonumber\\
    \leq & C\frac{1}{(\tau^2-\|u\|^2)^{2d}} \gamma_d^n\big((g_j^2)^{-1} (\frac{u}{\tau^2-\|u\|^2})\big).
\end{align}
Now, as $g_j^{2}\in \left(\hat{\mathcal{F}}_{\delta_N,4\log(n)^{1/2}}^{d,\beta+1}\right)^{d}$ we have that there exist $(\hat{\alpha}(j,l,w)_i)_{j,l,w,i}$ such that for all $i\in \{1,...,d\}$
$$(g_j^2)_i=\sum \limits_{j=0}^{\log_2(\delta_N^{-1})} \sum \limits_{l=1}^{2^d} \sum \limits_{w\in \{-4\log(n)^{1/2}2^j,...,4\log(n)^{1/2}2^j\}^d} \hat{\alpha}(j,l,w)_i \hat{\psi}^d_{jlw}$$
and
$$|\hat{\alpha}(j,l,w)_i|\leq C_{\beta+1} K(1+\|z\|)2^{-j(\beta+1+d/2)}.$$
Then, for all $z\in \{-4\log(n)^{1/2}2^j,...,4\log(n)^{1/2}2^j\}^d$ such that $0 \in supp(\hat{\psi}^d_{jlz})$, we have
$$|\hat{\alpha}(j,l,z)_i|\leq CC_{\beta+1} K2^{-j(\beta+1+d/2)}$$
so we deduce that
$$\|g_j^2(0)\|\leq C.$$
Then, we obtain from Lemma  \ref{lemma:thirdr} that
\begin{align*}
    \|(g_j^2)^{-1} (0)\|=&\|(g_j^2)^{-1} (0)-(g_j^2)^{-1}\circ g_j^2 (0)\|\\
    \leq & C\|0-g_j^2 (0)\|\\
    \leq & C.
\end{align*}
Therefore, we have
\begin{align}\label{align:bvghnbghjklsop}
    \|(g_j^2)^{-1} (\frac{u}{\tau^2-\|u\|^2})\|&\geq \|(g_j^2)^{-1} (\frac{u}{\tau^2-\|u\|^2})-(g_j^2)^{-1} (0)\|-\|(g_j^2)^{-1} (0)\|\nonumber\\
    &\geq C^{-1}\|\frac{u}{\tau^2-\|u\|^2}\|-C,
\end{align}
so from \eqref{align:cvghjhgfghbp} we have
\begin{align*}
        \zeta_j^g (u)
    \leq & C\frac{1}{(\tau^2-\|u\|^2)^{2d}} \gamma_d^n\big((g_j^2)^{-1} (\frac{u}{\tau^2-\|u\|^2})\big)\\
    \leq & C\frac{1}{(\tau^2-\|u\|^2)^{2d}} \exp\big(-C^{-1}\|\frac{u}{\tau^2-\|u\|^2}\|^2\big)\\
    \leq & C \exp\big(-C^{-1}(\tau^2-\|u\|^2)^{-2}\big).
\end{align*}
The proof of the lower bound on \( \zeta_j^g(u) \) follows a similar line of reasoning, focusing on points \( u \in B^d(0, \tau) \) such that \( \|\psi(u)\| \leq C_\star^{-1} \log(n)^{1/2} \), where \( C_\star > 0 \) is chosen large enough to ensure that \( \Psi(u) \in g_j^2(B^d(0, \log(n)^{1/2})) \). This condition guarantees that
\[
\gamma_d^n\left((g_j^2)^{-1} \circ \Psi(u)\right) = \gamma_d\left((g_j^2)^{-1} \circ \Psi(u)\right),
\]
allowing the bound to hold within the truncated Gaussian domain.

Let us now compute the estimate on the regularity of $\zeta_j^g$.
The map $g_2^j$ belongs to $\dot{\mathcal{H}}^{\beta}_C$ and  we have from Lemma  \ref{lemma:thirdr} that for all $x,y\in B^d(0,\log(n)^{1/2}+1)$,
$$\|\nabla g_j^2(x)\frac{y}{\|y\|}\|\geq C^{-1},$$
so we deduce that $\zeta^g_j\in \mathcal{H}^\beta$. Furthermore, using \eqref{align:bvghnbghjklsop} we obtain that for all $i\in \{0,...,\lfloor \beta\rfloor\}$,
\begin{align*}
\|\nabla^i \zeta_j^g(u)\|& = \|\nabla^i\Big(y\mapsto |\det\big(\nabla( (g_j^2)^{-1}\circ \Psi) (y)\big)|\gamma_d^n\big((g_j^2)^{-1}\circ \Psi) (y)\big)\Big)(u)\|\\
&  \leq C\log(n)^{2C_2\mathds{1}_{\beta=\lfloor \beta \rfloor}}
    ((\tau^2-\|u\|^2)^{-i-1}\vee 1 )^d\gamma_d^n\big((g_j^2)^{-1} \circ\Psi(u)\big)\\
    &  \leq C\log(n)^{2C_2\mathds{1}_{\beta=\lfloor \beta \rfloor}}
    ((\tau^2-\|u\|^2)^{-i-1}\vee 1 )^d\sup_{v\in \mathbb{S}^{d-1}}\gamma_d^n\big(C^{-1}\|\frac{u}{\tau^2-\|u\|^2}\|v\big)\\
    &\leq C\log(n)^{C_2\mathds{1}_{\beta=\lfloor \beta \rfloor}},
\end{align*}
so we deduce that $\zeta^g_j\in \mathcal{H}^{\lfloor \beta \rfloor}_{C\log(n)^{C_2\mathds{1}_{\beta=\lfloor \beta \rfloor}}}.$ Similarly, we obtain
\begin{align*}
\|\nabla^i \zeta_j^g(u)-\nabla^i \zeta_j^g(v)\|
    &\leq C\log(n)^{C_2\mathds{1}_{\beta=\lfloor \beta \rfloor}}\|u-v\|^{\beta-\lfloor \beta \rfloor},
\end{align*}
so we conclude that $\zeta^g_j\in \mathcal{H}^{ \beta }_{C\log(n)^{C_2\mathds{1}_{\beta=\lfloor \beta \rfloor}}}.$

\end{proof}

Let us now show that we can obtain a bound on the Hausdorff distance between the support of the target measure and the support of the estimator.

\begin{proposition}\label{prop:distanceofthees}
There exists $C_1,C_2>0$ such that for $(g,\varphi,\alpha)\in \mathcal{G}\times  \Phi\times \mathcal{A}$, if $\sup_{D\in \mathcal{D}} \ L(g,\varphi,\alpha,D)+\mathcal{C}(g,\varphi)\leq C_2^{-1}$, we have
$$\mathbb{H}\left(\bigcup_{i=1}^m F^{g,\varphi}\circ g_i^1(B^d(0,\tau)),\mathcal{M}^\star\right)\leq C_1\log\Big(C_1^{-1}\Big(C_2^{-1}+\delta_n^{d/2}\log(\delta_n^{-1})^2\Big)^{-1}\Big)^{-1/2}\vee (C\log(n)^{-1/2}).$$
\end{proposition}

\begin{proof}
Let us take
$$u \in \argmax_{v\in B^d(0,\tau)} \max_{i\in \{1,...,m\}}d(F^{g,\varphi}\circ g_i^1(v),\mathcal{M}^\star)$$
and 
$$r=d(F^{g,\varphi}\circ g_i^1(u),\mathcal{M}^\star).$$
From Proposition \ref{prop:regularityfandg} there exists $A\geq 1$ such that $F^{g,\varphi} \circ g_i^1\in \mathcal{H}^{1}_{A}(B^d(0,2\tau),\mathbb{R}^p)$, so taking $w=(1-\frac{r\wedge 4}{4A})u$, we have for all $v\in B^d(w,\frac{r\wedge 4\tau}{4A})$, $\min_{z\in \mathcal{M}^\star} \|F^{g,\varphi}\circ g_i^1(v)-z\|\geq r/2$. Writing $\zeta_j$ for the density of the measure $(\Psi^{-1}\circ g_j^2)_{\#} \gamma_d^n$ and using the dual formulation of the $1$-Wasserstein distance we have 
\begin{align}\label{align:jdhdfgfizpodf}
    &W_1((F^{g,\varphi}\circ g)_{\# \alpha \gamma_d^n},\mu^\star) \nonumber \\
    & =\sup_{f\in \text{Lip}_1}\int f(x)d(F^{g,\varphi}\circ g)_{\# \alpha \gamma_d^n}(x)-\int f(x)d\mu^\star(x)\nonumber\\
    & \geq \int 0\vee (r/2-\|F^{g,\varphi} \circ g_i^1(w)-x\|)d(F^{g,\varphi}\circ g)_{\# \alpha \gamma_d^n}(x)-\int 0\vee (r/2-\|F^{g,\varphi} \circ g_i^1(w)-x\|)d\mu^\star(x)\nonumber\\
    & \geq C^{-1} \frac{r}{4}\sum_{j=1}^m\int \mathds{1}_{\{\|F^{g,\varphi} \circ g_i^1(w)-F^{g,\varphi}\circ g_j^1(v)\|\leq r/4\}}\zeta_j(v)d\lambda^d(v)\nonumber\\
    & \geq C^{-1}\min_{s\in B^d(0,1-1\frac{r\wedge 4}{4A})}\sum_{j=1}^m\zeta_j(\tau s) \frac{r}{4}\int_{B^d(0,1-\frac{r\wedge 4}{4A})} \mathds{1}_{\{\|F^{g,\varphi} \circ g_i^1(w)-F^{g,\varphi}\circ g_i^1(v)\|\leq r/4\}}d\lambda^d(v).
\end{align}
Let $C_\star>0$ given by Lemma \ref{lemma:boundetaj} such that for $u\in B^d(0,\tau)$ satisfying $\|\psi(u)\|\leq C_\star^{-1}\log(n)^{1/2}$, we have
$$C^{-1}\exp(-C(\tau-\|u\|)^{-2})\leq \zeta_j^g(u).$$
Now, for $v\in B^d(0,\tau-\tau\frac{r\wedge 4}{4A})$ we have
\begin{align*}
    \|\Psi(v)\|=&\|\frac{v}{\tau^2-\|v\|^2}\|\\
    \leq & C(1-(1-\frac{r\wedge 4}{4A})^2)^{-1}\\
    \leq & C \left(\frac{r\wedge 4}{A}\right)^{-1},
\end{align*}
so there exists $C_{\star,2}>0$ such that $r\geq C_{\star,2}\log(n)^{-1/2}$ implies that for all $v\in B^d(0,\tau-\tau\frac{r\wedge 4}{4A})$ we have $\|\psi(v)\|\leq C_\star^{-1}\log(n)^{1/2}$. First, in the case where $r\leq C_{\star,2}\log(n)^{-1/2}$, recalling that
$$u \in \argmax_{v\in B^d(0,\tau)} \max_{i\in \{1,...,m\}}d(F^{g,\varphi}\circ g_i^1(v),\mathcal{M}^\star),$$
we directly obtain that
$$\mathbb{H}\left(\bigcup_{i=1}^m F^{g,\varphi}\circ g_i^1(B^d(0,\tau)),\mathcal{M}^\star\right)=r\leq C_{\star,2}\log(n)^{-1/2}.
$$
Now in the case $r\geq C_{\star,2}\log(n)^{-1/2}$, we obtain applying Lemma \ref{lemma:boundetaj} in the lower bound \eqref{align:jdhdfgfizpodf} that
\begin{align*}
    W_1((F^{g,\varphi}\circ g)_{\# \alpha \gamma_d^n},\mu^\star)
    & \geq C^{-1}\exp(-C^{-1}A^2r^{-2})\frac{r}{4}\int_{B^d(0,1-\frac{r\wedge 4}{4A})} \mathds{1}_{\{\|w-v\|\leq r/(4A)\}}d\lambda^d(v)\\
    & \geq C^{-1} \exp(-C^{-1}A^2r^{-2})A^{-d}r^{d+1}.
\end{align*}
 Then, from Corollary 20 in \cite{stephanovitch2024ipm} we have that
\begin{align*}
    d_{\mathcal{H}_1^{d/2}}\left((F^{g,\varphi}\circ g)_{\# \zeta^z},\mu^\star\right) \geq C^{-1} \exp(-C^{-1}r^{-2}).
\end{align*}
Therefore, using Lemma \ref{lemma:Landdhgamma} we conclude that 

\begin{align*}
    \sup_{D\in \mathcal{D}} \ L(g,\varphi,\alpha,D)\geq C^{-1} \exp(-Cr^{-2})-C \delta_n^{d/2}\log(\delta_n^{-1})^2.
\end{align*}
Supposing that $\sup_{D\in \mathcal{D}} \ L(g,\varphi,\alpha,D)\leq C_2^{-1}$, we finally obtain
$$r\leq C\log\Big(C^{-1}\Big(C_2^{-1}+\delta_n^{d/2}\log(\delta_n^{-1})^2\Big)^{-1}\Big)^{-1/2}.$$
\end{proof}

Finally, let us give a key technical lemma for the reconstruction of the manifold.
 \begin{lemma}\label{lemma:firstr} There exists $C_1,C_2>0$ such that for $\epsilon_\Gamma \in (C_1^{-2},C_1^{-1})$, $(g,\varphi,\alpha)\in \mathcal{G}\times  \Phi\times \mathcal{A}$, if $\sup_{D\in \mathcal{D}} \ L(g,\varphi,\alpha,D)+\mathcal{C}(g,\varphi)+R(g)\leq C_2^{-1}$, then  
    for all $u\in B^d(0,\tau)$ and $i\in \{1,...,m\}$, there exists $j\in \{1,...,m\}$ such that
    $$
    \|g_i^1(u)-g_j^1(0)\|\leq \tau(1+(K+1)\tau)\quad \text{ and }\quad \|\varphi_j(g_i^1(u))\|\leq \tau -\epsilon_\Gamma.
    $$
\end{lemma}

\begin{proof}
Let $C_2>0$ and suppose that 
\begin{equation}\label{eq:bghidodhzpfjdoss}
\sup_{D\in \mathcal{D}} \ L(g,\varphi,\alpha,D)+\mathcal{C}(g,\varphi)+R(g)\leq C_2^{-1}
\end{equation}
and that there exists $u\in B^d(0,\tau)$, $i\in \{1,...,m\}$ such that for all $j\in \{1,...,m\}$ we have either
    $$\|g_i^1(u)-g_j^1(0)\|> \tau(1+(K+1)\tau) \quad \text{ or }\quad \|\varphi_j(g_i^1(u))\|> \tau -\epsilon_\Gamma.$$
Taking $C_2>0$ large enough, we are going to show that the Wasserstein distance between $(F^{g,\varphi}\circ g)_{\# \alpha \gamma_d^n}$ and $\mu^\star$ is bounded below which will give us by Corollary 20 in \cite{stephanovitch2024ipm} that the IPM $d_{\mathcal{H}_1^{d/2}}$ is also bounded below. Finally using Lemma \ref{lemma:Landdhgamma}, we will get that the IPM $d_{\mathcal{D}}$ is also bounded below, which is a contradiction.

First, supposing $C_2>0$ large enough, we have by Proposition \ref{prop:distanceofthees} and  Lemma \ref{lemma:secondrbis} that $\min_{z\in \mathcal{M}^\star} \|g_i^1(u)-z\|\leq \epsilon_\Gamma$. Let $y\in \argmin_{z\in \mathcal{M}^\star} \|g_i^1(u)-z\|$, and write $\zeta_j$ for the density of the measure $(\Psi^{-1}\circ g_j^2)_{\#} \gamma_d^n$.
Using the dual formulation of the Wasserstein distance we get
\begin{align}\label{align:w&boundbelow}
W_1((F^{g,\varphi}\circ g)_{\# \alpha \gamma_d^n},\mu^\star) & \geq \int 0\vee (\epsilon_\Gamma-\|y-x\|)d\mu^\star(x)-\int 0\vee (\epsilon_\Gamma-\|y-x\|)d(F^{g,\varphi}\circ g)_{\# \alpha \gamma_d^n}(x)\nonumber\\
    & \geq C^{-1}\epsilon_\Gamma^{d+1}-C\epsilon_\Gamma\sum_{j=1}^m\int \mathds{1}_{\{\|y-F^{g,\varphi}\circ g_j^1(v)\|\leq \epsilon_\Gamma\}}\zeta_j(v)d\lambda^d(v).
\end{align}
Supposing $C_2>0$ large enough, from Lemma \ref{lemma:secondrbis} we have for all $v\in B^d(0,\tau)$ that $\|g_j^1(v)-F^{g,\varphi}\circ g_j^1(v)\|\leq  \epsilon_\Gamma$ so 
\begin{align*}
    \int \mathds{1}_{\{\|y-F^{g,\varphi}\circ g_j^1(v)\|\leq \epsilon_\Gamma\}}\zeta_j(v)d\lambda^d(v)&\leq \int \mathds{1}_{\{\|y-g_j^1(v)\|\leq 2\epsilon_\Gamma\}}\zeta_j(v)d\lambda^d(v)\\
    &\leq \int \mathds{1}_{\{\|g_i^1(u)-g_j^1(v)\|\leq 3\epsilon_\Gamma\}}\zeta_j(v)d\lambda^d(v).
\end{align*}
Let us treat separately the cases $\|g_i^1(u)-g_j^1(0)\|> \tau(1+(K+1)\tau)$ or not.
\begin{itemize}
    \item In the case $\|g_i^1(u)-g_j^1(0)\|> \tau(1+(K+1)\tau)$, using Lemma \ref{lemma:thirdr} we obtain 
\begin{align*}
    \int \mathds{1}_{\{\|g_i^1(u)-g_j^1(v)\|\leq 3\epsilon_\Gamma\}}\zeta_j(v)d\lambda^d(v)&= \int \mathds{1}_{\{\|g_i^1(u)-(g_j^1(0)+\int_0^1\nabla g_j^1(tv)vdt)\|\leq 3\epsilon_\Gamma\}}\zeta_j(v)d\lambda^d(v)\\
    &\leq \int \mathds{1}_{\{\|\int_0^1\nabla g_j^1(tv)vdt)\|\geq \tau(1+(K+1)\tau)-3\epsilon_\Gamma\}}\zeta_j(v)d\lambda^d(v)
    \\
    &\leq \int \mathds{1}_{\{\|v\|\geq \tau-\frac{3\epsilon_\Gamma}{1+(K+1)\tau}\}}\zeta_j(v)d\lambda^d(v)\\
    & \leq C\max_{v\in B^d(0,\tau-\frac{3\epsilon_\Gamma}{1+(K+1)\tau})^c}\zeta_j (v),
\end{align*}
so from Lemma \ref{lemma:boundetaj} we have 
$$\int \mathds{1}_{\{\|g_i^1(u)-g_j^1(v)\|\leq 3\epsilon_\Gamma\}}\zeta_j(v)d\lambda^d(v)\leq C\exp(-C^{-1}\epsilon_\Gamma^{-1}).$$
\item
In the case $\|\varphi_j(g_i^1(u))\| > \tau -\epsilon_\Gamma$ and $\|g_i^1(u)-g_j^1(0)\|\leq \tau(1+(K+1)\tau)$ we have from Lemma \ref{lemma:secondr} that
\begin{align*}
    \int \mathds{1}_{\{\|g_i^1(u)-g_j^1(v)\|\leq 3\epsilon_\Gamma\}}\zeta_j(v)d\lambda^d(v)& \leq \int \mathds{1}_{\{\|g_j^1\circ \varphi_j(g_i^1(u))-g_j^1(v)\|\leq 3\epsilon_\Gamma+\|g_i^1(u)-g_j^1\circ \varphi_j(g_i^1(u))\|\}}\zeta_j(v)d\lambda^d(v)\\
    & \leq \int \mathds{1}_{\{\|g_j^1\circ \varphi_j(g_i^1(u))-g_j^1(v)\|\leq 4\epsilon_\Gamma\}}\zeta_j(v)d\lambda^d(v)\\
    & \leq \int \mathds{1}_{\{\|\varphi_j(g_i^1(u))-v\|\leq C\epsilon_\Gamma\}}\zeta_j(v)d\lambda^d(v),
\end{align*}
as $R(g)=0$ taking $C_2>1$ in \eqref{eq:bghidodhzpfjdoss}. Using the fact that $\|\varphi_j(g_i^1(u))\| > \tau -\epsilon_\Gamma$, we obtain
$$\int \mathds{1}_{\{\|\varphi_j(g_i^1(u))-v\|\leq C\epsilon_\Gamma\}}\zeta_j(v)d\lambda^d(v)\leq C\exp(-C^{-1}\epsilon_\Gamma^{-1}),$$
which also gives 
$$\int \mathds{1}_{\{\|g_i^1(u)-g_j^1(v)\|\leq 3\epsilon_\Gamma\}}\zeta_j(v)d\lambda^d(v)\leq C\exp(-C^{-1}\epsilon_\Gamma^{-1}).$$
\end{itemize}
Finally, in both cases, from \eqref{align:w&boundbelow} we get

\begin{align}\label{align:boundwassepsi}
    W_1((F^{g,\varphi}\circ g)_{\# \alpha \gamma_d^n},\mu^\star)
    & \geq C^{-1}\epsilon_\Gamma^{d+1}-\epsilon_\Gamma\sum_{j=1}^m\int \mathds{1}_{\{\|y-F^{g,\varphi}\circ g_j^1(v)\|\leq \epsilon_\Gamma\}}\zeta_j(v)d\lambda^d(v)\nonumber\\
    & \geq C^{-1}\epsilon_\Gamma^{d+1}-C\epsilon_\Gamma\exp(-C^{-1}\epsilon_\Gamma^{-1}).
\end{align}
Using Corollary 20 in \cite{stephanovitch2024ipm} and Lemma \ref{lemma:Landdhgamma}, we obtain a lower bound on $\sup_{D\in \mathcal{D}} \ L(g,\varphi,\alpha,D).$ 
We deduce 
that this is not possible for $\sup_{D\in \mathcal{D}} \ L(g,\varphi,\alpha,D)\leq C_2^{-1}$ with $C_2>0$ large enough depending on $\epsilon_\Gamma$.
\end{proof}

\subsection{Proofs of  the regularity properties}
In this section we give the proofs of the results of Sections \ref{sec:proptarget}, \ref{sec:Feffe} and  \ref{sec:reguofes}. let us start by giving a first property of the measures satisfying Assumption \ref{assump:model}.
\begin{proposition}\label{prop:keydecompunedeplus}
    Let $\mu$ a probability measure satisfying Assumption \ref{assump:model} on a manifold $\mathcal{M}$. Then, there exists  a collection of maps $(\phi_i)_{i=1,...,m}\in\mathcal{H}^{\beta+1}_{C}(B^d(0,\tau),\mathbb{R}^p)^m$ satisfying $\forall u,v\in B^d(0,\tau)$, $i\in \{1,...,m\}$, 
    $1-K\tau\leq \|\nabla \phi_i(u)\frac{v}{\|v\|}\|\leq 1+K\tau$ and there exists non-negative functions $(\zeta_i)_{i=1,...,m}\in\mathcal{H}^{\beta}_{C}(\mathbb{R}^d,\mathbb{R})^m$ with $supp(\zeta_i)\subset B^d(0,\tau)$, such that for all bounded continuous function $h:\mathbb{R}^p\rightarrow \mathbb{R}$ we have 
    \begin{equation*}
        \int_\mathcal{M} h(x)d\mu(x)= \sum \limits_{i=1}^m \int_{\mathbb{R}^d}h(\phi_i(u))\zeta_i(u)d\lambda^d(u).
    \end{equation*}
\end{proposition}

\begin{proof}
    The proof follows exactly the proof of Proposition 4 in \cite{stephanovitch2024ipm} until a slight change at the end. The change is the following:
    Let $(x_i)_{i=1,...,m}$ be a finite sequence such that $\mathcal{M}\subset \bigcup \limits_{i=1}^m B^p(x_i,\frac{\tau}{2})$ and $(\rho_i)_{i=1,...,m}$ an approximation of unity subordinated to $(\phi_{x_i}^{-1}(B^d(0,\tau)))_i$. We have 
    \begin{align*}
        \mathbb{E}_{\mu}[h(X)] & =\sum \limits_{i=1}^m\mathbb{E}_{\mu}[h(X)\rho_i(X)]=\sum \limits_{i=1}^m\int_{B^d(0,\tau)}h(\phi_{x_i}^{-1}(u))\rho_i(\phi_{x_i}^{-1}(u))d\phi_{x_i\# \mu}(u)\\
        & =\sum \limits_{i=1}^m\int_{B^d(0,\tau)}h(\phi_{x_i}^{-1}(u))\rho_i(\phi_{x_i}^{-1}(u))f_{\phi_{x_i\# \mu}}(u)d\lambda^d(u),
    \end{align*}
    for $f_{\phi_{x_i\# \mu}}$ the density of $\phi_{x_i\# \mu}$ with respect to the Lebesgue measure. Let us take $\phi_i(u)=\phi_{x_i}^{-1}(u)$ and $\zeta_i(u)=\rho_i(\phi_{x_i}^{-1}(u))f_{\phi_{x_i\# \mu}}(u)$. As $\nabla \phi_{x_i}(x_i)$ is equal to the identity on the tangent space of $\mathcal{M}$ at $x_i$, for all $u\in B^d(0,\tau)$ and $v\in \mathbb{R}^d$ we have 
$$\|\nabla \phi_i^{-1}(u)v\|=\|\nabla \phi_i^{-1}(0)v+\int_0^1 \nabla^2 \phi_i^{-1}(tu)uvdt\|\geq \|v\|-K\|u\| \|v\|\geq (1-K\tau)\|v\|$$
and likewise 
$$\|\nabla \phi_i^{-1}(u)v\|\leq (1+K\tau)\|v\|.$$
Finally, as $\phi_i\in \mathcal{H}^{\beta+1}_C$ and $\zeta_i\in \mathcal{H}^{\beta}_C$, we have the result.
\end{proof}

\subsubsection{Proof of Proposition \ref{prop:keydecomp}}\label{sec:prop:keydecomp}
\begin{proof}
\textbf{Decomposition through charts}\\
Let us fix $\tau = \frac{1}{8K}$. Let $(x_i)_{i=1,...,m}$ be a finite sequence such that $\mathcal{M}^\star\subset \bigcup \limits_{i=1}^m B^p(x_i,\frac{\tau}{2})$ and let $(\rho_i)_{i=1,...,m}\in \mathcal{H}^{\beta+1}_C(\mathbb{R}^p,\mathbb{R})$ an approximation of unity defined by
$$\rho_i(x)=\chi_i(x)/\sum_{j=1}^m\chi_j(x),$$
with
$$\chi_i(x) = \det\big(\nabla \Psi(\phi_{x_i}(x))\big) \gamma_d\big(\Psi(\phi_{x_i}(x)\big)\mathds{1}_{\{\|\phi_{x_i}(x)\|\leq \tau\}},$$
for $\phi_x$ the orthogonal projection on $\mathcal{T}_x$ defined in \eqref{eq:phix} and $\Psi:B^d(0,\tau)\rightarrow \mathbb{R}^d$ defined by 
    $\Psi(u):= \frac{u}{\tau^2-\|u\|^2}.$ Having taken the sequence $(x_j)_j$ satisfying that for all $y\in \mathcal{M}^\star$, there exists $i_y\in \{1,...,m\}$ such that $\|y-x_{i_y}\|\leq \tau/2$, we have in particular $\|\phi_{x_{i_y}}(y)\|\leq \tau/2$ so
\begin{align}\label{eq:cjisiqjdikdd}
    \sum_{j=1}^m\chi_j(y)&\geq \chi_{i_y}(y)= \det\big(\nabla \Psi(\phi_{x_{i_y}}(y))\big) \gamma_d\big(\Psi(\phi_{x_{i_y}}(y)\big)\nonumber\\
    & \geq C^{-1}\inf_{\|v\|=\tau/2}\gamma_d\left(\frac{v}{\tau^2-(\tau/2)^2}\right)\nonumber\\
   & \geq C^{-1}.
\end{align}
    Using the approximation of unity $(\rho_i)_i$, we can write 
    \begin{align}\label{align:cfvghuefji}
       \mathbb{E}_{\mu^\star}[h(X)] & =\sum \limits_{i=1}^m \mathbb{E}_{\mu^\star}[h(X)\rho_i(X)]=\sum \limits_{i=1}^m\int_{B^d(0,\tau)}h(\phi_{x_i}^{-1}(u))\rho_i(\phi_{x_i}^{-1}(u))d\phi_{x_i\# \mu^\star}(u)\\\nonumber
        & =\sum \limits_{i=1}^m\int_{B^d(0,\tau)}h(\phi_{x_i}^{-1}(u))\rho_i(\phi_{x_i}^{-1}(u))f_{\phi_{x_i\# \mu^\star}}(u)d\lambda^d(u),
    \end{align}
    for $f_{\phi_{x_i\# \mu^\star}}$ the density of $\phi_{x_i\# \mu^\star}$ with respect to the Lebesgue measure. Let us take 
    $$g_i^1(u)=\phi_{x_i}^{-1}(u) \quad \text{ and }\quad \zeta_i(u)=\rho_i(\phi_{x_i}^{-1}(u))f_{\phi_{x_i\# \mu^\star}}(u),$$
    so \eqref{align:cfvghuefji} can be written as
\begin{align*}
 \mathbb{E}_{\mu^\star}[h(X)]  &=\sum \limits_{i=1}^m\int_{B^d(0,\tau)}h(g_i^1(u))\zeta_i(u)d\lambda^d(u)\\
 &=\sum \limits_{i=1}^m \alpha_i\int_{B^d(0,\tau)}h(g_i^1(u))\bar{\zeta}_i(u)d\lambda^d(u),
\end{align*}
with $\alpha_i>0$ a normalizing constant such that the function
\begin{align}\label{eq:fnkidjdojdjdnbsos}
\bar{\zeta}_i(u)&=\rho_i(\phi_{x_i}^{-1}(u))f_{\phi_{x_i\# \mu^\star}}(u)\alpha_i^{-1}\nonumber\\
&=\det\big(\nabla \Psi(u)\big)\gamma_d( \Psi(u))\frac{f_{\phi_{x_i\# \mu^\star}}(u)}{\alpha_i\sum_{j=1}^m\chi_j(g_i^1(u))},
\end{align}
is a probability density.

\textbf{Regularity properties of $g_i^1$}\\
As $\nabla \phi_{x_i}(x_i)$ is equal to the identity on the tangent space of $\mathcal{M}^\star$ at $x_i$, for all $u\in B^d(0,\tau)$ and $v\in \mathbb{R}^d$ we have 
$$\|\nabla (g_i^1)^{-1}(u)v\|=\|\nabla (g_i^1)^{-1}(0)v+\int_0^1 \nabla^2 (g_i^1)^{-1}(tu)uvdt\|\geq \|v\|-K\|u\| \|v\|\geq (1-K\tau)\|v\|$$
and likewise 
$$\|\nabla (g_i^1)^{-1}(u)v\|\leq (1+K\tau)\|v\|.$$

In particular, it gives us that $f_{\phi_{x_i\# \mu^\star}}(u)\leq C$ as $\mu^\star$ is bounded on $\mathcal{M}^\star$. Then, using \eqref{eq:cjisiqjdikdd} we obtain $$\alpha_i=\int_{B^d(0,\tau)} \zeta_i(u)d\lambda^d(u)
\leq C.$$ On the other hand, as $\mu^\star$ is bounded below on $\mathcal{M}^\star$ we also have that $\alpha_i\geq C^{-1}$.

\textbf{Existence of a transport map from $\gamma_d$ to $\bar{\zeta}_i$}\\
From \eqref{eq:fnkidjdojdjdnbsos}, we can rewrite the probability density $\bar{\zeta}_i$ as
$$\bar{\zeta}_i(u)=\det\big(\nabla \Psi(u)\big)\gamma_d( \Psi(u))\exp(z_i(u))\alpha_i^{-1},$$
with
$$
z_i(u)=\log\left(f_{\phi_{x_i\# \mu^\star}}(u)/\sum_{j=1}^m\chi_j(g_i^1(u))\right)
$$
which satisfies $z_i\in \mathcal{H}^\beta_C(B^d(0,\tau))$ as $\mu^\star$ satisfies the density regularity condition (Definition \ref{def:densitycond}) and both $f_{\phi_{x_i\# \mu^\star}}$ and $\sum_{j=1}^m\chi_j(g_j^1(\cdot))$ are bounded above and below on $B^d(0,\tau)$. 
 Then, taking $$\xi_i=\Psi_{\#} \bar{\zeta}_i,$$ we have that its probability density is equal to
$\xi_i(x)=\gamma_d(x)\exp(z_i(\Psi^{-1}(x))\alpha_i^{-1},$
so from Corollary 1 in \cite{stephanovitch2024smooth}, there exists a map $g_i^2\in \mathcal{H}^{\beta+1}_C(\mathbb{R}^d, \mathbb{R}^d)$  such that $$(g_i^2)_{\# }\gamma_d=\xi_i\quad \text{ and }\quad \lambda_{\min}(\nabla g_i^2)\geq C^{-1}.$$ Therefore, we have
\begin{align*}
 \mathbb{E}_{\mu^\star}[h(X)]&=\sum \limits_{i=1}^m \alpha_i\int_{B^d(0,\tau)}h(g_i^1(u))\bar{\zeta}_i(u)d\lambda^d(u)\\
 &=\sum \limits_{i=1}^m \alpha_i\int_{\mathbb{R}^d}h(g_i^1\circ \Psi^{-1}(y))d\xi_i(y)\\
&=\sum \limits_{i=1}^m\alpha_i\int_{\mathbb{R}^d} h(g_i^1\circ \Psi^{-1}\circ g_i^2(y))d\gamma_d(y).
\end{align*}
\end{proof}

\subsubsection{Proof of Proposition \ref{prop:decompphi}}\label{sec:prop:decompphi}
The proof is identical to the proof of Proposition \ref{prop:keydecomp} in Section \ref{sec:prop:keydecomp}. 

In particular,
the existence of the maps $g_i^{1}$ is given in the paragraph "\textbf{Decomposition through charts}", the existence the maps $g_i^{2}$ is given in the paragraph "\textbf{Existence of a transport map}".

\subsubsection{Proof of Proposition \ref{prop:qualityofapproxfketa}}\label{sec:prop:qualityofapproxfketa}
Let us first show the following lemma.
\begin{lemma}\label{lemma:bouundcoefwav}
    Under the conditions of Proposition \ref{prop:qualityofapproxfketa}, we have for all
    $j\in \mathbb{N},l\in \{0,...,2^k-1\},z\in \mathbb{Z}^k$ that 
$$
|\alpha_{\overline{f}}(j,l,z)-\alpha_{f}(j,l,z)\mathds{1}_{j\leq \log_2(\delta^{-1})}|\leq Cn^{-1/2}2^{-j(\eta+k/2)}.
$$

\end{lemma}

\begin{proof}
Let $(j,l,z)\in \{1,...,\log(\delta^{-1})\}\times \{1,...,2^k\}\times \{-R2^{j},...,R2^j\}^k$, we have 
\begin{align*}
    \|\hat{\psi}_{jlw}-\psi_{jlw}\|_{\mathcal{H}^{\lfloor d/2 \rfloor +2}} & \geq C \|\hat{\psi}_{jlw}-\psi_{jlw}\|_{\mathcal{B}^{\lfloor d/2 \rfloor +2}_{\infty,\infty}}\\
    & = C\sup \limits_{(j^{'},l^{'},w^{'})}2^{j^{'}(\lfloor d/2 \rfloor +2+k/2)}|\alpha_{\hat{\psi}_{jlw}}(j^{'},l^{'},w^{'})-\mathds{1}_{\{(j^{'},l^{'},w^{'})=(j,l,w)\}}|.
\end{align*}
Furthermore
\begin{align*}
    \|& \hat{\psi}_{jlw}-\psi_{jlw}\|_{\mathcal{H}^{\lfloor d/2 \rfloor +2}} \\
    &=2^{jk/2}\|\prod \limits_{i=1}^k \hat{\psi}_{l_i}(2^j\cdot_i-w_i)-\prod \limits_{i=1}^p \psi_{l_i}(2^j\cdot_i-w_i)\|_{\mathcal{H}^{\lfloor d/2 \rfloor +2}}\\
    & =2^{jp/2}\|\sum \limits_{i=1}^p (\hat{\psi}_{l_i}(2^j\cdot_i-w_i)-\psi_{l_i}(2^j\cdot_i-w_i))\prod \limits_{r>i}^p \hat{\psi}_{l_r}(2^j\cdot_r-w_r)\prod \limits_{s<i}^p \psi_{l_s}(2^j\cdot_s-w_s)\|_{\mathcal{H}^{\lfloor d/2 \rfloor +2}}\\
    & \leq 2^{j(\lfloor d/2 \rfloor +2+k/2)}\sum \limits_{i=1}^p \| (\hat{\psi}_{l_i}(\cdot_i-w_i)-\psi_{l_i}(\cdot_i-w_i))\prod \limits_{r>i}^p \hat{\psi}_{l_r}(\cdot_r-w_r)\prod \limits_{s<i}^p \psi_{l_s}(\cdot_s-w_s)\|_{\mathcal{H}^{\lfloor d/2 \rfloor +2}}\\
    & \leq C 2^{j(\lfloor d/2 \rfloor +2+k/2)}n^{-1/2},
\end{align*}
as $\|\hat{\psi}_{l_i}-\psi_{l_i}\|_{\mathcal{H}^{\lfloor d/2 \rfloor +2}} \leq Cn^{-1/2}$ for $l_i\in \{0,1\}$.
We deduce that for all $(j^{'},l^{'},w^{'})\neq(j,l,w)$ we have
$$
|\alpha_{\hat{\psi}_{jlw}}(j^{'},l^{'},w^{'})|\leq C2^{(j-j^{'})(\lfloor d/2 \rfloor +2+k/2)}n^{-1/2}
$$
and 
$$
|\alpha_{\hat{\psi}_{jlw}}(j,l,w)-1|\leq Cn^{-1/2}.
$$
Furthermore, for $(j^{'},l^{'},w^{'})\neq(j,l,w)$ we have
\begin{align*}
    |\alpha_{\hat{\psi}_{jlw}}(j^{'},l^{'},w^{'})|=&|\int \langle \hat{\psi}_{jlw}(x)-\psi_{jlw}(x),\psi_{j^{'}l^{'}w^{'}}(x) \rangle d\lambda^k(x) |\\
    \leq & 2^{j^{'}k/2}\int \| \hat{\psi}_{jlw}(x)-\psi_{jlw}(x)\| \mathds{1}_{\{x\in supp(\psi_{jlw})\cup supp(\hat{\psi}_{jlw}) \}}d\lambda^k(x) \\
    \leq & 2^{k/2(j^{'}-j)}\left(\int \| \hat{\psi}_{jlw}(x)-\psi_{jlw}(x)\|^2d\lambda^k(x)\right)^{1/2} \\
    \leq & 2^{k/2(j^{'}-j)}\left(\int \| \hat{\psi}_{0lw}(x)-\psi_{0lw}(x)\|^2d\lambda^k(x)\right)^{1/2}  \\
    \leq & C 2^{k/2(j^{'}-j)}n^{-1/2}.
\end{align*}

Now
taking $\alpha_f(j,l,w):=0$ for $j> \log_2(\delta^{-1})$, we have 
\begin{align*}
    |&\alpha_{\overline{f}}(j^{'},l^{'},w^{'})-\alpha_f(j^{'},l^{'},w^{'})| \\
    = & | \sum \limits_{j=0}^{\log_2(\delta^{-1})} \sum \limits_{l=1}^{2^p} \sum \limits_{w\in \{-K2^{j},...,K2^j\}^p} \alpha_f(j,l,w) \langle\hat{\psi}_{jlw},\psi_{j^{'}l^{'}w^{'}}\rangle-\alpha_f(j^{'},l^{'},w^{'})|\\
     \leq & | \sum \limits_{(j,l,w)\neq (j^{'},l^{'},w^{'})} |\alpha_f(j,l,w)\langle \hat{\psi}_{jlw},\psi_{j^{'}l^{'}w^{'}}\rangle|\\
    & +|\alpha_f(j^{'},l^{'},w^{'})||\langle \hat{\psi}_{j^{'}l^{'}w^{'}} , \psi_{j^{'}l^{'}w^{'}}\rangle-1|\mathds{1}_{\{j^{'}\leq\log_2(\delta^{-1})\}}\\
     \leq & \sum \limits_{\substack{(j,l,w)\neq (j^{'},l^{'},w^{'})\\ supp(\hat{\psi}_{jlw})\cap supp(\psi_{j^{'}l^{'}w^{'}})\neq \varnothing}} C2^{-j(\eta+k/2)}n^{-1/2}\Big(2^{(j-j^{'})(\lfloor d/2 \rfloor +2+k/2)}\mathds{1}_{\{j\leq j^{'}\}}+2^{k/2(j^{'}-j)}\mathds{1}_{\{j> j^{'}\}}\Big)\\
     & + C2^{-j^{'}(\eta+k/2)}n^{-1/2}\mathds{1}_{\{j^{'}\leq\log_2(\delta^{-1})\}}.
\end{align*}
Now,
\begin{align*}
      &\sum \limits_{\substack{(j,l,w)\in \{1,...,\log_2(\delta^{-1})\}\times \{1,...,2^k\}\times \{-R2^{j},...,R2^j\}^k\\ supp(\hat{\psi}_{jlw})\cap supp(\psi_{j^{'}l^{'}w^{'}})\neq \varnothing}} 2^{-j(\eta+k/2)}\Big(2^{(j-j^{'})(\lfloor d/2 \rfloor +2+k/2)}\mathds{1}_{\{j\leq j^{'}\}}+2^{k/2(j^{'}-j)}\mathds{1}_{\{j> j^{'}\}}\Big)\\
    &\leq C\sum_{j=1}^{j^{'}} 2^{-j(\eta+k/2)}2^{(j-j^{'})(\lfloor d/2 \rfloor +2+k/2)}+C\sum_{j=j^{'}+1}^{\log(\delta^{-1})} 2^{-j(\eta+k/2)}2^{k/2(j^{'}-j)}2^{k(j-j^{'})}
    \\
    &\leq C 2^{-j^{'}(\lfloor d/2 \rfloor +2+k/2)}\sum_{j=1}^{j^{'}} 2^{j(\lfloor d/2 \rfloor +2-\eta)}+C 2^{-j^{'}k/2}\sum_{j=j^{'}+1}^{\log(\delta^{-1})} 2^{-j\eta}
    \\
    &\leq C 2^{-j^{'}(\eta+k/2)}.
\end{align*}
\end{proof}

We can now give the Proof of Proposition \ref{prop:qualityofapproxfketa}.

\begin{proof}[Proof of Proposition \ref{prop:qualityofapproxfketa}]
Using Lemma \ref{lemma:bouundcoefwav} we have, 
\begin{align*}
    \|f-\overline{f}\|_{\mathcal{B}^{\gamma,b}_{\infty,\infty}}  =& \sup_{j,l,z} 2^{j(\gamma+k/2)}j^b|\alpha_{\overline{f}}(j,l,z)-\alpha_{f}(j,l,z)|\\
     \leq &\sup_{j\leq \log_2(\delta^{-1}),l,z} 2^{j(\gamma+k/2)}j^b|\alpha_{\overline{f}}(j,l,z)-\alpha_{f}(j,l,z)|+\sup_{j>\log_2(\delta^{-1}),l,z} 2^{j(\gamma+k/2)}j^b(|\alpha_{\overline{f}}(j,l,z)|+|\alpha_{f}(j,l,z)|)\\
     \leq& \sup_{j\leq \log_2(\delta^{-1}),l,z} 2^{j(\gamma+k/2)}j^bCn^{-1/2}2^{-j(\eta+k/2)} + \sup_{j> \log_2(\delta^{-1}),l,z} 2^{j(\gamma+k/2)}j^bC2^{-j(\eta+k/2)}(1+n^{-1/2})\\
     \leq & C \big(n^{-1/2}+\log(\delta^{-1})^b\delta^{\eta-\gamma}\big).
\end{align*}
On the other hand,
\begin{align*}
    \|f-\overline{f}\|_{\mathcal{B}^{\gamma,b}_{q_1,q_2}}^{q_2}  =& \sum_{j,l} 2^{jq_2(\gamma+k/2-k/q_1)}j^{bq_2}\left(\sum_{z}|\alpha_{\overline{f}}(j,l,z)-\alpha_{f}(j,l,z)|^{q_1}\right)^{q_2/q_1}\\
    \leq & C\sum_{j\leq \log_2(\delta^{-1}) } 2^{jq_2(\gamma+k/2-k/q_1)}j^{bq_2}\left(\sum_{z}n^{-q_1/2}2^{-jq_1(\eta+k/2)}\right)^{q_2/q_1}\\
    &+ C\sum_{j> \log_2(\delta^{-1}) } 2^{jq_2(\gamma+k/2-k/q_1)}j^{bq_2}\left(\sum_{z}2^{-jq_1(\eta+k/2)}\right)^{q_2/q_1}
    \\
    \leq & Cn^{-q_2/2}\log(\delta^{-1})^{bq_2}\sum_{j\leq \log_2(\delta^{-1}) } 2^{jq_2(\gamma+k/2-k/q_1)}\left(R^k2^{jk}2^{-jq_1(\eta+k/2)}\right)^{q_2/q_1}\\
    &+ C\sum_{j> \log_2(\delta^{-1}) } 2^{jq_2(\gamma+k/2-k/q_1)}j^{bq_2}\left(R^k2^{jk}2^{-jq_1(\eta+k/2)}\right)^{q_2/q_1}\\
    \leq &  CR^{kq_2/q_1}n^{-q_2/2}\sum_{j\leq \log_2(\delta^{-1}) } 2^{jq_2(\gamma-\eta)}j^{bq_2}+ CR^{kq_2/q_1}\sum_{j> \log_2(\delta^{-1}) } 2^{jq_2(\gamma-\eta)}j^{bq_2}\\
    \leq & CR^{kq_2/q_1}n^{-q_2/2}+ CR^{kq_2/q_1} \log(\delta^{-1})^{bq_2}\delta^{q_2(\eta-\gamma)}.
\end{align*}
We then deduce that
\begin{align*}
    \|f-\overline{f}\|_{\mathcal{B}^{\gamma,b}_{q_1,q_2}(B^k(0,R))}\leq CR^{k/q_1}(n^{-1/2}+  \log(\delta^{-1})^{b}\delta^{\eta-\gamma})
\end{align*}
\end{proof}

\subsubsection{Proof of Proposition \ref{prop:regularityfandg}}\label{sec:prop:regularityfandg}
\begin{proof}
     Let $\xi\in (\eta,\beta+1)$ a non integer. From Proposition \ref{prop:reguofG}, we have $F^{g,\varphi}_{j_k}\circ...\circ F^{g,\varphi}_{j_1} \circ g_{i}\in \mathcal{H}^{\xi}_{C_\xi}(B^d(0,2\tau),\mathbb{R}^p) \subset \mathcal{H}^{\eta}_{C_\xi}(B^d(0,2\tau),\mathbb{R}^p)$. Furthermore, we have $G_{k,i}:=F^{g,\varphi}_{j_k}\circ...\circ F^{g,\varphi}_{j_1} \circ g_{i}\in \mathcal{H}^{\beta+1}_C(B^d(0,2\tau),\mathbb{R}^p)$ if $\beta$ is not an integer and $G_{k,i}\in \mathcal{H}^{\beta+1}_{C\log(n)^2}(B^d(0,2\tau),\mathbb{R}^p)$ if $\beta$ is an integer.
\end{proof}

\subsubsection{Proof of Proposition \ref{prop:isamanifold}}\label{sec:prop:isamanifold}
\begin{proof}
Let $C_2>0$ and $(g,\varphi,\alpha)\in \mathcal{G}\times  \Phi\times \mathcal{A}$ such that \begin{equation}\label{eq:ghdheozhnfhyudeikhds}
\sup_{D\in \mathcal{D}} \ L(g,\varphi,\alpha,D)+\mathcal{C}(g,\varphi)+R(g)\leq C_2^{-1}.
\end{equation}
Taking $x=F^{g,\varphi}\circ g_{i}(y)\in \mathcal{M}$ and supposing $C_2>0$ large enough, we have from Lemma \ref{lemma:firstr} that there exists $j\in \{1,...,m\}$ such that
$$
\|g_i(y)-g_j^1(0)\|\leq \tau(1+(K+1)\tau),\quad \|\varphi_j(g_i(y))\|\leq \tau -\epsilon_\Gamma
$$
and from Lemma \ref{lemma:secondrbis}  we have that
\begin{align}\label{eq:distxgj}
\|x-g_j^1(0)\|& \leq  \|g_i(y)-g_j^1(0)\|+ \|x-g_i(y)\|\nonumber \\
& \leq \tau(1+(K+1)\tau)+\frac{\epsilon_\Gamma}{4\|\varphi_j\|_{\mathcal{H}^1}},
\end{align}
taking $C_2>0$ large enough in \eqref{eq:ghdheozhnfhyudeikhds} such that
$$\|x-g_i(y)\|\leq \frac{1}{4C_3} \epsilon_\Gamma\leq  \frac{\epsilon_\Gamma}{4\|\varphi_j\|_{\mathcal{H}^1}},$$
using that $\varphi_j\in \left(\hat{\mathcal{F}}_{\delta_N,K}^{p,\beta+1}\right)^{d} \subset \mathcal{H}^1_{C_3}$ from Proposition \ref{prop:reguofG}.
Let us show that $F^{g,\varphi}_{m}\circ...\circ F^{g,\varphi}_{j+1}\circ g_j$ is a $C^{\beta+1}$ diffeomorphism between an open subset of $\mathbb{R}^d$ and $\mathcal{M}\cap B^p(x,(4C_3)^{-1}\epsilon_\Gamma)$. 
Let $w\in \mathcal{M}\cap B^p(x,(4C_3)^{-1}\epsilon_\Gamma)$, $l\in \{1,...,m\}$  and $v\in B^d(0,\tau)$ such that $ w=F^{g,\varphi}\circ g_{l}^1(v) $. From \eqref{eq:distxgj}  we have 
\begin{align*}
\|F^{g,\varphi}\circ g_{l}^1(v)-g_j^1(0)\|& \leq \|F^{g,\varphi}\circ g_{l}^1(v)-x\| + \|x-g_j^1(0)\| \\
& \leq \tau(1+(K+1)\tau)+\frac{\epsilon_\Gamma}{2\|\varphi_j\|_{\mathcal{H}^1}}
\end{align*}
and from  Lemma \ref{lemma:secondrbis} 
\begin{align}\label{algin:fjfhfkdejhdfkhfddd}
\|F^{g,\varphi}_{j-1}\circ...\circ F^{g,\varphi}_{1}\circ g_{l}^1(v)-g_j^1(0)\| &\leq C \|F^{g,\varphi}_{j-1}\circ...\circ F^{g,\varphi}_{1}\circ g_{l}^1(v)-F^{g,\varphi}\circ g_{l}^1(v)\| +\|F^{g,\varphi}\circ g_{l}^1(v)-g_j^1(0)\|\nonumber\\
&\leq \tau(1+(K+1)\tau)+\frac{3\epsilon_\Gamma}{4\|\varphi_j\|_{\mathcal{H}^1}}.
\end{align}
On the other hand, taking $C_2>0$ large enough in \eqref{eq:ghdheozhnfhyudeikhds} we have
\begin{align*}
    \|\varphi_j \circ F^{g,\varphi}_{j-1}\circ...\circ F^{g,\varphi}_{1}\circ g_{l}(v)\|\leq &\|\varphi_j(g_i(y))\| + \|\varphi_j \circ F^{g,\varphi}_{j-1}\circ...\circ F^{g,\varphi}_{1}\circ g_{l}^1(v)-\varphi_j(g_i(y))\| \\
    \leq &\|\varphi_j(g_i(y))\| + \|\varphi_j\|_{\mathcal{H}^1}\| F^{g,\varphi}_{j-1}\circ...\circ F^{g,\varphi}_{1}\circ g_{l}^1(v)-g_i(y)\| \\
    \leq &\|\varphi_j(g_i(y))\| + \|\varphi_j\|_{\mathcal{H}^1}\| F^{g,\varphi}_{j-1}\circ...\circ F^{g,\varphi}_{1}\circ g_{l}^1(v)-F^{g,\varphi}\circ g_{l}^1(v)\| \\
    & +\|\varphi_j\|_{\mathcal{H}^1}\|F^{g,\varphi}\circ g_{l}^1(v)-F^{g,\varphi}\circ g_i(y)\|+\|\varphi_j\|_{\mathcal{H}^1}\|F^{g,\varphi}\circ g_i(y)-g_i(y)\|\\
     \leq & \tau-\epsilon_\Gamma +\epsilon_\Gamma/4+\epsilon_\Gamma/4+\epsilon_\Gamma/4\\
     \leq & \tau -\epsilon_\Gamma/4.
\end{align*}
Now, recalling the Definition of the smooth cutoff function $\Gamma>0$ from \eqref{eq:Gamma}, for $\epsilon_\Gamma>0$ small enough we get from \eqref{algin:fjfhfkdejhdfkhfddd} that
$$\Gamma(\|F^{g,\varphi}_{j-1}\circ...\circ F^{g,\varphi}_{1}\circ g_{l}^1(v)-g_j^1(0)\|)=1,$$
which gives
\begin{align*}
w=&F^{g,\varphi}_{m}\circ...\circ F^{g,\varphi}_{j+1}\big(F^{g,\varphi}_{j}\circ...\circ F^{g,\varphi}_{1}\circ g_{l}^1(v) \big)\\
=& F^{g,\varphi}_{m}\circ...\circ F^{g,\varphi}_{j+1}\big(g_j^1\circ\varphi_j \circ F^{g,\varphi}_{j-1}\circ...\circ F^{g,\varphi}_{1}\circ g_{l}^1(v)\big),
\end{align*}
so in particular $$w\in F^{g,\varphi}_{m}\circ...\circ F^{g,\varphi}_{j+1}\circ g_j^1(B^d(0,\tau-\epsilon_\Gamma/4)).$$

Now supposing taking $C_2>0$ large enough in \eqref{eq:ghdheozhnfhyudeikhds} and taking $\epsilon_\Gamma>0$ small enough, we get from Proposition \ref{prop:diffeofrondg}  that $F^{g,\varphi}_{m}\circ...\circ F^{g,\varphi}_{j+1}\circ g_j^1$ is a $C^{\beta+1}$ diffeomorphism from $B^d(0,\tau)$ onto its image and from Proposition \ref{prop:regularityfandg} we have that it belongs to $\mathcal{H}^{\beta+1}_{C\log(n)^{2\mathds{1}_{\beta=\lfloor \beta \rfloor}}}$, which gives the result.

\end{proof}

\subsubsection{Proof of Proposition \ref{prop:densitychecked}}\label{sec:prop:densitychecked}
Let us write $T_j:=F^{g,\varphi}\circ g_j^1$ and $\zeta_j$ for the density of the probability measure $(\Psi^{-1}\circ g_j^2)_{\#} \gamma_d^n$. We have from Proposition \ref{prop:diffeofrondg} that for all $j\in \{1,...,m\}$, $T_j$ is a diffeomorphism from $B^d(0,\tau)$ onto its image. Then, for $x\in \mathcal{M}$, we have 
\begin{equation}
    f_{\hat{\mu}}(x)=\sum_{j=1}^m \frac{\alpha_j}{\sum_{l} \alpha_l} \zeta_j(T_j^{-1}(x))\text{det}\Big(\big(\nabla T_j(T_j^{-1}(x)\big)^\top\nabla T_j(T_j^{-1}(x))\Big)^{-1/2}\mathds{1}_{\{x\in T_j(B^d(0,\tau))\}}.
\end{equation}
Now, from Lemma \ref{lemma:boundetaj} we have that $\zeta_j$ belongs to $\mathcal{H}^{\beta}_C$ if $\beta$ is not an integer and belongs to $\mathcal{H}^{\beta}_{C\log(n)^{C_2}}$ otherwise. Furthermore,  combining Propositions \ref{prop:regularityfandg} and \ref{prop:diffeofrondg} we obtain that $T_j,T_j^{-1}\in \mathcal{H}^{\beta+1}_{C\log(n)^{C_2\mathds{1}_{\beta=\lfloor \beta \rfloor}}}$ and as 
$$\min_{u\in B^d(0,\tau)}\min_{v\in \mathbb{S}^{d-1}} \|\nabla T_j(u)v\|\geq 1-(K+2)\tau,$$
we conclude that $f_{\hat{\mu}}$ belongs to $\mathcal{H}^{\beta}_{C\log(n)^{C_2\mathds{1}_{\beta=\lfloor \beta \rfloor}}}$.

Let us now show the lower bound on the density. Let $\theta>0$ and suppose that there exist $w=F^{g,\varphi}\circ g_i(y)\in \mathcal{M}$ such that $f_{\hat{\mu}}(w)\leq \theta$. For $A\geq 4K$, supposing $C_2>0$ large enough we have by Proposition \ref{prop:distanceofthees} that $\min_{z\in \mathcal{M}^\star} \|x-z\|\leq A^{-1}.$ Then, as $\beta>1$, we have $f_{\hat{\mu}}\in \mathcal{H}^1_C$ so using the dual formulation of the Wasserstein distance and the fact that $\mathcal{M}$ and $\mathcal{M}^\star$ satisfy the $(2,C)$-manifold condition, we get
\begin{align}\label{align:w&boundbelowaa}
W_1((F^{g,\varphi}\circ g)_{\# \alpha \gamma_d^n},\mu^\star) & \geq \int 0\vee (\frac{1}{2A}-\|w-x\|)d\mu^\star(x)-\int 0\vee (\frac{1}{2A}-\|w-x\|)d(F^{g,\varphi}\circ g)_{\# \alpha \gamma_d^n}(x)\nonumber\\
    & \geq C^{-1}A^{-d}-C\theta A^{-d},
\end{align}
for $A>0$ small enough.
Using Corollary 20 in \cite{stephanovitch2024ipm} and Lemma \ref{lemma:Landdhgamma}, $\eqref{align:w&boundbelowaa}$ gives a lower bound on the IPM $d_{\mathcal{D}}((F^{g,\varphi}\circ g)_{\# \alpha \gamma_d^n},\mu^\star)$
which is in contradiction with $\sup_{D\in \mathcal{D}} L(g,\varphi,\alpha,D)\leq C_2^{-1}$ for $C_2$ and $\theta^{-1}$ large enough.

\section{Proofs of the statistical results}
\subsection{Bias variance decomposition}
\subsubsection{Rates of approximation of empirical means}
Let us give a first general concentration bound that will be used several times.
\begin{lemma}\label{lemma:keyboundclass}
 Let $\eta>0$, $A\geq 1$, $k\in \mathbb{N}_{>0}$,  $\zeta:B^d(0,A)\rightarrow \mathbb{R}$ a probability density, $\Lambda\in \mathcal{H}^1_A(B^k(0,A),\mathbb{R})$ such that $\|\Lambda(x)\|\leq A \|x\|$  and $\mathcal{F}\subset \mathcal{H}^\eta_A(B^d(0,A),\mathbb{R}^k)$. Then, for all $\delta>0$ and $U_1,...,U_n\sim \zeta$ iid random variables, we have
 \begin{align*}
    &\mathbb{E}\Big[\sup \limits_{\substack{f\in \mathcal{F} \\ \|F\|_\infty\leq \delta}}  \frac{1}{n}\sum_{j=1}^n  \int_{B^d(0,A)} \Lambda(f(u))\zeta(u) d\lambda^d(u)-\Lambda(f(U_j))\Big]+\frac{16}{\sqrt{n}}\int_{\delta/4}^{2}\sqrt{\log(|(\mathcal{F})_\epsilon|)}d\epsilon\\
 &\leq \frac{C_k}{\sqrt{n}}\left(A\sqrt{\frac{(\delta+1/n)^2\log(|\mathcal{F}_{1/n}|n)}{n} }
  +A^{C_2}\big(1+\delta^{(1-\frac{d}{2\eta})}+\log(\delta^{-1}/4)\mathds{1}_{\{2\eta= d\}}\big)\right).\end{align*}
\end{lemma}

\begin{proof}
Let us take care of the first term. Let $(\mathcal{F})_{1/n}$ be a minimal $1/n$ covering of the class $\mathcal{F}$ for the distance $\|\cdot\|_\infty$. Therefore, for any $f \in \mathcal{F}$, there exists $f_n\in (\mathcal{F})_{1/n}$ such that $\|f_n-f\|_\infty\leq \frac{1}{n}$. We have for $\alpha>2A/n$ 
\begin{align}\label{align:boundonprobacovering}
     & \mathbb{P}\Big(\sup \limits_{\substack{f\in \mathcal{F} \\ \|f\|_\infty\leq \delta}}  \frac{1}{n}\sum_{j=1}^n  \int_{\mathbb{R}^d}\Lambda(f(u))\zeta(u)d\lambda^d(u)-\Lambda(f(U_j)) \geq \alpha \Big)\nonumber\\
     & \leq \mathbb{P}\Big( \sup \limits_{\substack{f\in \mathcal{F} \\ \|f\|_\infty\leq \delta}}  \frac{1}{n}\sum_{j=1}^n  \int_{\mathbb{R}^d}\Lambda(f_n(u))\zeta(u)d\lambda^d(u)-\Lambda(f_n(U_j)) \geq \alpha -\frac{2A}{n} \Big)\nonumber\\
    & \leq \mathbb{P}\Big( \substack{\exists f_n\in (\mathcal{F})_{1/n}\\ \|f_n \|_\infty\leq \delta+\frac{1}{n}} \text{ with } \frac{1}{n}\sum_{j=1}^n  \int_{\mathbb{R}^d}\Lambda(f_n(u))\zeta(u)d\lambda^d(u)-\Lambda(f_n(U_j)) \geq \alpha -\frac{2A}{n} \Big)\nonumber\\
    & \leq \sum \limits_{\substack{f_n\in (\mathcal{F})_{1/n}\\ \|f_n\|_\infty\leq \delta+\frac{1}{n}}}  \mathbb{P}\Big( \frac{1}{n}\sum_{j=1}^n  \int_{\mathbb{R}^d}\Lambda(f_n(u))\zeta(u)d\lambda^d(u)-\Lambda(f_n(U_j)) \geq \alpha -\frac{2A}{n} \Big)\nonumber\\
    & \leq |(\mathcal{F})_{1/n}|  \exp\left(-\frac{n(\alpha -\frac{2A}{n})^2}{4A^2(\delta+\frac{1}{n})^2}\right),
\end{align}
applying Hoeffding's inequality to the random variables $\Lambda(f_n(U_j)) $, $j=1,...,n$. Then for $\tau \in (0,1]$, solving $|(\mathcal{F})_{1/n}|  \exp\left(-\frac{n(\alpha -\frac{2A}{n})^2}{4A^2(\delta+\frac{1}{n})^2}\right)=\tau$ with respect to $\alpha$, we obtain that with probability at least $1-\tau$
\begin{align*}
& \sup \limits_{\substack{f\in \mathcal{F} \\ \|f\|_\infty\leq \delta}}  \frac{1}{n}\sum_{j=1}^n  \int_{\mathbb{R}^d}\Lambda(f(u))\zeta(u)d\lambda^d(u)-f(U_j) \\
& \leq CA\left(\sqrt{\frac{(\delta+1/n)^2\log(|(\mathcal{F})_{1/n}|/\tau)}{n} }+\frac{1}{n}\right).
\end{align*}
Taking $\tau=1/n$ we deduce that
\begin{align*}
& \mathbb{E}\Big[\sup \limits_{\substack{f\in \mathcal{F} \\ \|f\|_\infty\leq \delta}}  \frac{1}{n}\sum_{j=1}^n  \int_{\mathbb{R}^d}\Lambda(f(u))\zeta(u)d\lambda^d(u)-f(U_j) \Big]\\
& \leq CA\left(\sqrt{\frac{(\delta+1/n)^2\log(|(\mathcal{F})_{1/n}|n)}{n} }+\frac{1}{n}\right).
\end{align*}
Now, for the second term, as $\mathcal{F}\subset\mathcal{B}^{ \eta}_{\infty,\infty}(B^d(0,A),\mathbb{R},A)$ we have $\log(|(\mathcal{F})_\epsilon|)\leq CA^{C_2} \log(|(\mathcal{B}^{ \eta}_{\infty,\infty}([0,1]^d,\mathbb{R},1)_\epsilon|)$ and as  $\log(|(\mathcal{B}^{ \eta}_{\infty,\infty}([0,1]^d,\mathbb{R},1)_\epsilon|)\leq C \epsilon ^{-\frac{d}{ \eta}}$ we deduce that
\begin{align*}
\frac{16}{\sqrt{n}}\int_{\delta/4}^{2}\sqrt{\log(|(\mathcal{F})_\epsilon|)}d\epsilon \leq & \frac{CA^{C_2}}{\sqrt{n}}\int_{\delta/4}^{2}\epsilon ^{-\frac{d}{2 \eta}}d\epsilon\\
 = & \frac{CA^{C_2}\mathds{1}_{\{2 \eta\neq d\}}}{\sqrt{n}(1-\frac{d}{2 \eta})}\left(2^{1-\frac{d}{2 \eta}}-\left(\frac{\delta}{4}\right)^{(1-\frac{d}{2 \eta})}\right) +\frac{CA^{C_2}\mathds{1}_{\{2 \eta= d\}}}{\sqrt{n}}(\log(2)-\log(\delta/4)).
\end{align*}
Therefore we have $$\frac{16}{\sqrt{n}}\int_{\delta/4}^2\sqrt{\log(|(\mathcal{F})_\epsilon|)}d\epsilon\leq  \frac{CA^{C_2}}{\sqrt{n}}(1+\delta^{(1-\frac{d}{2 \eta})}+\log(\delta^{-1}/4)\mathds{1}_{\{2 \eta= d\}}).$$
\end{proof}

Using this result, let us now give a bound on the expected difference between $\mathbb{E}_{X\sim \mu^\star}[D(X)]$ and its empirical approximation $\frac{1}{n}\sum \limits_{i=1}^n D(X_j)$ (with $X_j\sim \mu^\star$), for all  $D\in \mathcal{D}$.
\begin{proposition}\label{prop:statisticalerrorformu} Let $n,A\in \mathbb{N}_{>0}$, $\mathcal{M}$ satisfying the $(\beta+1,K)$-manifold condition, $\mu$ a probability measure satisfying the $(\beta,K)$-density condition on $\mathcal{M}$ and $(X_1,...,X_n)$ an i.i.d. sample of law $\mu$. Then for all $\gamma\geq 1$ and $\mathcal{D}\subset \mathcal{H}_A^\gamma(\mathbb{R}^p,\mathbb{R})$, we have
\begin{align*}
\mathbb{E}&\Big[\sup_{D\in \mathcal{D}} \mathbb{E}_{X\sim \mu}[ D(X)]-\frac{1}{n} \sum_{j=1}^n D(X_j)\Big]\\
& \leq  C \min \limits_{\delta \in [0,1]}  \Biggl\{A \sqrt{\frac{(\delta+1/n)^2\log(n|(\mathcal{D}\circ \phi)_{1/n}|)}{n}}+\frac{A^{C_2}}{\sqrt{n}}\Big(1+\delta^{(1-\frac{d}{2(\gamma\wedge(\beta+1))})}+\log(\delta^{-1})\mathds{1}_{\{2(\gamma\wedge(\beta+1))= d\}}\Big)\Biggl\},
\end{align*}
for a certain $\phi\in \mathcal{H}^{\beta+1}_C(B^d(0,\tau),\mathbb{R}^p)$.
\end{proposition}

\begin{proof}
Let us first remark that using Proposition \ref{prop:keydecompunedeplus}, there exists a collection of maps $(\phi_i)_{i=1,...,m}\in\mathcal{H}^{\beta+1}_{C}(\mathbb{R}^d,\mathbb{R}^p)^m$ and non-negative weight functions $(\zeta)_{i=1,...,m}\in\mathcal{H}^{\beta+1}_{C}(B^d(0,\tau),\mathbb{R})^m$ such that
$$
\frac{1}{n}\sum_{j=1}^n \mathbb{E}_{X\sim \mu}[ D(X)]-  D(X_j)=\frac{1}{n}\sum \limits_{i=1}^m\sum_{j=1}^n  \int_{\mathbb{R}^d}D(\phi_i(u))\zeta_i(u)d\lambda^d(u)-D(\phi_i(Y_j))\mathds{1}_{\{T_j=i\}},
$$
with $T_1, ..., T_n$ iid sample of multinomial law $\textbf{Mult}(\{\int \zeta_1,...,\int\zeta_m\})$ and $Y_j$ of law $\zeta_{T_j}/{\int \zeta_{T_j}}$.

We are going to treat each case $i=1,...,m$ separately using Dudley's inequality bound (\cite{vaart2023empirical}, chapter 2). Let us take $$\mathcal{D}_{\phi_i}=\{D\circ \phi_i \ | \ D\in \mathcal{D}\}$$ equipped with the metric $M=\|\cdot\|_\infty$. For $\theta:=f \in \mathcal{D}_{\phi_i}$, we define the zero-mean sub-Gaussian process (e.g. \citep{Van2016prob}, chapter 5) as $$
X_\theta=\frac{1}{\sqrt{n}}\sum_{j=1}^n  \int_{\mathbb{R}^d}f(u)\zeta_i(u)d\lambda^d(u)-f(Y_j)\mathds{1}_{\{T_j=i\}}.
$$
 Then, as $\sup \limits_{f,f^{'}\in \mathcal{f}_{\phi_i}} \|f-f^{'}\|_\infty\leq 2$, using Dudley's bound we obtain for all $ \delta \in [0,2]$,
\begin{align}\label{eq:dudleyf1}
\mathbb{E}&\Big[\sup \limits_{D\in \mathcal{D}} \frac{1}{n}\sum_{j=1}^n  \int_{\mathbb{R}^d}D(\phi_i(u))\zeta_i(u)d\lambda^d(u)-D(\phi_i(Y_j))\mathds{1}_{\{T_j=i\}} \Big]\nonumber\\
\leq & 2 \mathbb{E}\Big[\sup \limits_{\substack{D\in \mathcal{D} \\ \|D\circ \phi_i\|_\infty\leq \delta}}  \frac{1}{n}\sum_{j=1}^n  \int_{\mathbb{R}^d}D(\phi_i(u))\zeta_i(u)d\lambda^d(u)-D(\phi_i(Y_j))\mathds{1}_{\{T_j=i\}}\Big]+\frac{16}{\sqrt{n}}\int_{\delta/4}^{2}\sqrt{\log(|(\mathcal{D}_{\phi_i})_\epsilon|)}d\epsilon.
\end{align}
Finally, as  $\mathcal{D}_{\phi_i}\subset \mathcal{H}^{(\beta+1)\wedge \gamma}_{CA}(B^d(0,\tau),\mathbb{R})$,
applying Lemma \ref{lemma:keyboundclass} to the class $\mathcal{D}_{\phi_i}$ we get the result.

\end{proof}

\begin{proposition}\label{prop:statisticalerrorC} Let $\mathcal{G}$ and $\Phi$ the classes of functions defined in \eqref{eq:classofgene} and \eqref{eq:classeofinv} respectively. Then, for independent random variables  $U_1,...,U_N\sim \mathcal{U}(B^d(0,\tau))$, $T_1,...,T_N\sim \mathcal{U}_m$, we have
\begin{align*}
& \mathbb{E}\Big[\sup_{(g,\varphi)\in \mathcal{G}\times  \Phi} \mathcal{C}(g,\varphi)-C_N(g,\varphi)\Big]\\
 \leq &  C\log(N)^{C_2}\min \limits_{\delta \in [0,1]}  \Biggl\{ \sqrt{\frac{(\delta_1+1/N)^2\log(N|\mathcal{G}_{1/N}||(\Phi \circ \mathcal{G})_{1/N}|)}{N}}+\frac{1}{\sqrt{N}}(1+\delta^{(1-\frac{d}{2(\beta+1)})}+\log(\delta^{-1})\mathds{1}_{\{2(\beta+1)= d\}})\Biggl\}.
\end{align*}
\end{proposition}

\begin{proof}
For $(g,\varphi)\in \mathcal{G}\times \Phi$ and   independent random variables  $U_1,...,U_N\sim \mathcal{U}(B^d(0,2\tau))$, we have
\begin{align*}
    \mathcal{C}(g,\varphi)-C_N(g,\varphi) = &\frac{1}{N} \sum_{j=1}^N  \sum_{i=1}^m\sum_{k=1}^m \mathbb{E}_{U\sim \mathcal{U}(B^d(0,2\tau))}[\|g_{k}^1(U)-g_i^1\circ \varphi_i\circ g_{k}^1(U)\|\Gamma\big((\|g_k^1(U) - g_i^1(0)\|-\tau^2)\vee 0\big)] \\
    & -  \|g_{k}^1(U_j)-g_i^1\circ \varphi_i\circ g_{k}^1(U_j)\|\Gamma\big((\|g_k^1(U_j) - g_i^1(0)\|-\tau^2)\vee 0\big).
\end{align*}
We are going to use Dudley's inequality bound (\cite{vaart2023empirical}, chapter 2) like in Proposition \ref{prop:statisticalerrorformu}. Define
$$
\mathcal{A}_{i,k}=\{(g_k^1,g_i^1,\varphi_i\circ g_k^1)| \ g\in \mathcal{G},\varphi \in \Phi\},
$$
equipped with the infinite norm.
For $\theta=(g_k^1,g_i^1,\varphi_i\circ g_k^1) \in \mathcal{A}_{i,k}$, we define the zero-mean sub-Gaussian process  as 
\begin{align*}
X_\theta=&\frac{1}{\sqrt{N}}\sum_{j=1}^N \mathbb{E}_{U\sim \mathcal{U}(B^d(0,2\tau))}[\|g_{k}^1(U)-g_i^1\circ \varphi_i\circ g_{k}^1(U)\|\Gamma\big((\|g_k^1(U) - g_i^1(0)\|-\tau^2)\vee 0\big)] \\
    & -  \|g_{k}^1(U_j)-g_i^1\circ \varphi_i\circ g_{k}^1(U_j)\|\Gamma\big((\|g_k^1(U_j) - g_i^1(0)\|-\tau^2)\vee 0\big). 
\end{align*}
 Then, using Dudley's bound we obtain for all $ \delta \in [0,2]$,
\begin{align*}
\mathbb{E}&\Big[\sup \limits_{(g_k^1,g_i^1,\varphi_i\circ g_k^1)\in\mathcal{A}_{i,k}} \frac{1}{N}\sum_{j=1}^N  \int_{B^d(0,2\tau)}\|g_{k}^1(u)-g_i^1\circ \varphi_i\circ g_{k}^1(u)\|\Gamma\big((\|g_k^1(u) - g_i^1(0)\|-\tau^2)\vee 0\big)d\lambda^d(u)\\
&-\|g_{k}^1(U_j)-g_i^1\circ \varphi_i\circ g_{k}^1(U_j)\|\Gamma\big((\|g_k^1(U_j) - g_i^1(0)\|-\tau^2)\vee 0\big)) \Big]\nonumber\\
\leq & 2 \mathbb{E}\Big[\sup \limits_{\substack{(g_k^1,g_i^1,\varphi_i\circ g_k^1)\in\mathcal{A}_{i,k} \\ \|(g_k^1,g_i^1,\varphi_i\circ g_k^1)\|_\infty\leq \delta}}  \frac{1}{N}\sum_{j=1}^N  \int_{B^d(0,2\tau)}\|g_{k}^1(u)-g_i^1\circ \varphi_i\circ g_{k}^1(u)\|\Gamma\big((\|g_k^1(u) - g_i^1(0)\|-\tau^2)\vee 0\big)d\lambda^d(u)\\
&-\|g_{k}^1(U_j)-g_i^1\circ \varphi_i\circ g_{k}^1(U_j)\|\Gamma\big((\|g_k^1(U_j) - g_i^1(0)\|-\tau^2)\vee 0\big))\Big]+\frac{16}{\sqrt{N}}\int_{\delta/4}^{2}\sqrt{\log(|(\mathcal{A}_{i,k})_\epsilon|)}d\epsilon.
\end{align*}
Finally, from Proposition \ref{prop:reguofG} we have $\Phi \subset \mathcal{H}^{\beta+1}_{C\log(N)^2}(B^p(0,K),B^d(0,C))$ and for all $g\in \mathcal{G}$ and $i\in \{1,....,m\}$ we have $g_i^1\in \mathcal{H}^{\beta+1}_{C\log(N)^2}(B^d(0,2\tau),B^p(0,C))$. Then, applying Lemma \ref{lemma:keyboundclass} to the class $\mathcal{A}_{i,k}$ we get the result.
\end{proof}
Using this result and  applying the same proof technique, we obtain a bound on the difference between the empirical loss and the true loss. Let
\begin{equation}\label{eq:fgphi2}
    \mathcal{F}^{\mathcal{G},\Phi}:=\{F^{g,\varphi}|g\in \mathcal{G}, \varphi \in \Phi\}.
\end{equation}

\begin{proposition}\label{prop:statisticalerror} Let 
$\mathcal{G},
\Phi,
\mathcal{A},
\mathcal{D}$ be defined in \eqref{eq:classofgene}, \eqref{eq:classeofinv}, \eqref{eq:A} and \eqref{eq:classesofdis} respectively. Then, for iid independent random variables  $Y_1,...,Y_N\sim \gamma_d^n$, $\omega_1,...,\omega_N\sim \mathcal{U}_m$, $X_1,...,X_n\sim \mu^\star$, we have
\begin{align*}
& \mathbb{E}\Big[\sup_{(g,\varphi,\alpha)\in \mathcal{G}\times  \Phi\times \mathcal{A},D\in \mathcal{D}} L(g,\varphi,\alpha,D)-L_{N,n}(g,\varphi,\alpha,D)\Big]\\
 \leq &  C\log(N)^{C_2} \min \limits_{\delta_1 [0,1]}  \Biggl\{ \sqrt{\frac{(\delta_1+1/N)^2\log(N|(\mathcal{D}\circ\mathcal{F}^{\mathcal{G},\Phi}\circ \mathcal{G})_{1/N}||\mathcal{A}_{1/N}|)}{N}}\\
 &+\frac{1}{\sqrt{N}}(1+\delta_1^{(1-\frac{d}{2(\frac{d}{2}\wedge(\beta+1))})}+\log(\delta^{-1}_1)\mathds{1}_{\{\frac{d}{2}\wedge(\beta+1)= \frac{d}{2}\}})\Biggl\}\\
& + C\log(n)^{C_2}\min \limits_{\delta_2 \in [0,1/2]}  \Biggl\{  \sqrt{\frac{(\delta_2+1/n)^2\log(n|(\mathcal{D} \circ \phi_1)_{1/n}|)}{n}}+\frac{C}{\sqrt{n}}\log(\delta_2^{-1})\Biggl\}.
\end{align*}
\end{proposition}

\begin{proof}
Let us first remark that
\begin{align*}
\mathbb{E}&\Big[\sup_{(g,\varphi,\alpha)\in \mathcal{G}\times  \Phi\times \mathcal{A},D\in \mathcal{D}} L(g,\varphi,\alpha,D)-L_{N,n}(g,\varphi,\alpha,D)]\\
& -\mathbb{E}_{X\sim \mu^\star}[ D(X)]-\left(\frac{1}{N} \sum_{j=1}^N \alpha_{\omega_j}\  D\Big(F^{g,\varphi}\circ g_{\omega_j}(Y_j)\Big)\right) \Big]\\
\leq &\mathbb{E}\Big[\sup_{\substack{g\in \mathcal{G},\varphi\in \Phi,\\\alpha\in \mathcal{A},D\in \mathcal{D}}} \frac{1}{m}\sum_{i=1}^m \alpha_i \mathbb{E}_{Y\sim \gamma_d^n}[ D(F^{g,\varphi}\circ g_{i}(Y))]-\frac{1}{N} \sum_{j=1}^N \alpha_{\omega_j}\  D\Big(F^{g,\varphi}\circ g_{\omega_j}(Y_j)\Big) \Big]\\
&+\mathbb{E}\Big[\sup_{D\in \mathcal{D}} \mathbb{E}_{X\sim \mu^\star}[ D(X)]-\frac{1}{n} \sum_{j=1}^n D(X_j)\Big].
\end{align*}
Now, the second term can be controlled by Proposition \ref{prop:statisticalerrorformu} as from Proposition \ref{prop:reguofG} we have $\mathcal{D}\subset \mathcal{H}^{\beta+1}_{C\log(n)^2}(B^p(0,K),\mathbb{R})$. To bound the first term, we are going to use Dudley's inequality bound (\cite{vaart2023empirical}, chapter 2) like in Proposition \ref{prop:statisticalerrorformu}. Define
$$
\mathcal{W}_i=\{ \alpha_i(D\circ F^{g,\varphi}\circ g_i)\ | \ D\in \mathcal{D},g\in \mathcal{G},\varphi \in \Phi,\alpha\in \mathcal{A}\},
$$
equipped with the infinite norm.
For $\theta=\alpha_i (D\circ F^{g,\varphi}\circ g_i) \in \mathcal{W}_i$, we define the zero-mean sub-Gaussian process  as $$
X_\theta=\frac{\alpha_i}{\sqrt{N}}\sum_{j=1}^N \int_{\mathbb{R}^d}D(F^{g,\varphi}\circ g_i(y))d\gamma_d^n(u)- D(F^{g,\varphi}\circ g_{i}(Y_j))\mathds{1}_{\{\omega_j=i\}}
$$
with $\omega_j\sim \mathcal{U}_m$ and $Y_j\sim \gamma_d^n$ independent. Then, applying Dudley's bound we get for all $\delta\in [0,2]$ that
\begin{align*}
&\mathbb{E}\Big[\sup_{(g,\varphi,\alpha)\in \mathcal{G}\times  \Phi\times \mathcal{A},D\in \mathcal{D}} L(g,\varphi,\alpha,D)-L_{N,n}(g,\varphi,\alpha,D)\Big]\\
& \leq \mathbb{E}\Big[\sup_{\substack{f\in \mathcal{W}_i\\\|f\|_{\infty\leq \delta}}}\frac{\alpha_i}{\sqrt{N}}\sum_{j=1}^N \int_{\mathbb{R}^d}D(F^{g,\varphi}\circ g_i(y))d\gamma_d^n(u)- D(F^{g,\varphi}\circ g_{i}(Y_j))\mathds{1}_{\{\omega_j=i\}} \Big]+\frac{16}{\sqrt{N}}\int_{\delta/4}^{2}\sqrt{\log(|(\mathcal{W}_{i})_\epsilon|)}d\epsilon.
\end{align*}
    Finally, from Proposition \ref{prop:reguofG} we have $\mathcal{W}_i \subset \mathcal{H}^{\frac{d}{2}\wedge(\beta+1)}_{C\log(N)^{C_2}}(B^d(0,2\log(n)^{1/2}),\mathbb{R})$ so applying Lemma \ref{lemma:keyboundclass} to the class $\mathcal{W}_{i}$ we get the result.
\end{proof}

Let us now give a first bias-variance decomposition bound on the error.
\begin{lemma}\label{lemma:decomptheo}
Under the assumptions of Theorem \ref{theo:biaisvariancedec},   supposing that there exist $(g,\varphi)\in \mathcal{G}\times  \Phi$ satisfying the constraint \eqref{eq:opticonstrain}, we have
    $$d_{\mathcal{H}^{d/2}_1}(\hat{\mu},\mu^\star)\leq \Delta_{\mathcal{D}_{\hat{\mu}}}+ \Delta_{\mathcal{G}\times \Phi \times \mathcal{A}} + 2\sup_{(g,\varphi,\alpha)\in \mathcal{G}\times  \Phi\times \mathcal{A},D\in \mathcal{D}} L(g,\varphi,\alpha,D)-L_{N,n}(g,\varphi,\alpha,D),    $$
    with $\Delta_{\mathcal{G}\times \Phi \times \mathcal{A}}$ and $\Delta_{\mathcal{D}_{\hat{\mu}}}$ defined in \eqref{eq:deltaG} and \eqref{eq:deltaD} respectively.
\end{lemma}

\begin{proof}

   Define
   $$(\overline{g},\overline{\varphi},\overline{z}) \in \argmin \limits_{(g,\varphi,\alpha)\in \mathcal{G}\times  \Phi\times \mathcal{A}} d_{\mathcal{H}_1^{d/2}}((F^{g,\varphi}\circ g)_{\# \alpha \gamma_d^n},\mu^\star)$$
   and for all $\mu \in \mathcal{F}$
$$h_\mu \in \argmax_{h\in \mathcal{H}^{d/2}_1} \int h(x)d\mu(x)-\int h(x)d\mu^\star(x),\quad D_\mu \in \argmax_{D\in \mathcal{D}} \int D(x)d\mu(x)-\int D(x)d\mu^\star(x)$$
and
$$h_\mu^n \in \argmax_{h\in \mathcal{H}^{d/2}_1} \int h(x)d\mu(x)-\frac{1}{n}\sum_{i=1}^n h(X_i),\quad D_\mu^n \in \argmax_{h\in \mathcal{D}} \int D(x)d\mu(x)-\frac{1}{n}\sum_{i=1}^n D(X_i).$$
   We have
\begin{align*}
    d_{\mathcal{H}^{d/2}_1}(\hat{\mu},\mu^\star) =& \int h_{\hat{\mu}}(x)d\hat{\mu}(x)-\int h_{\hat{\mu}}(x)d\mu^\star(x)\\
    \leq & \int D_{\hat{\mu}}(x)d\hat{\mu}(x)-\int D_{\hat{\mu}}(x)d\mu^\star(x) +\Delta_{\mathcal{D}_{\hat{\mu}}}
    \end{align*}
    and 
\begin{align*} 
 \int D_{\hat{\mu}}(x)d\hat{\mu}(x)-\int D_{\hat{\mu}}(x)d\mu^\star(x)    = &  \int D_{\hat{\mu}}(x)d\hat{\mu}(x)-\frac{1}{N} \sum_{j=1}^N \hat{\alpha}_{\omega_j} D_{\hat{\mu}}(F^{\hat{g},\hat{\varphi}}\circ \hat{g}_{\omega_j}(Y_j))  \\
     &+  \frac{1}{N} \sum_{j=1}^N \hat{\alpha}_{\omega_j} D_{\hat{\mu}}(F^{\hat{g},\hat{\varphi}}\circ \hat{g}_{\omega_j}(Y_j))-\frac{1}{n}\sum_{i=1}^n D_{\hat{\mu}}(X_i)  \\
     & +  \frac{1}{n}\sum_{i=1}^n D_{\hat{\mu}}(X_i)-\int D_{\hat{\mu}}(x)d\mu^\star(x)  \\
     \leq &  \sup_{(g,\varphi,\alpha)\in \mathcal{G}\times  \Phi\times \mathcal{A},D\in \mathcal{D}} L(g,\varphi,\alpha,D)-L_{N,n}(g,\varphi,\alpha,D)  \\
     &+  \frac{1}{N} \sum_{j=1}^N \hat{\alpha}_{\omega_j} D_{\hat{\mu}}(F^{\hat{g},\hat{\varphi}}\circ \hat{g}_{\omega_j}(Y_j))-\frac{1}{n}\sum_{i=1}^n D_{\hat{\mu}}(X_i)  .
\end{align*}
Furthermore, writing $\overline{\mu}=(F^{\overline{g},\overline{\varphi}}\circ \overline{g})_{\# \overline{\alpha}\gamma_d^n}$, by definition of $\hat{\mu}$ we have 
\begin{align*}
   & \frac{1}{N} \sum_{j=1}^N \hat{\alpha}_{\omega_j} D_{\hat{\mu}}(F^{\hat{g},\hat{\varphi}}\circ \hat{g}_{\omega_j}(Y_j))-\frac{1}{n}\sum_{i=1}^n D_{\hat{\mu}}(X_i)  \\
   &\leq  \frac{1}{N} \sum_{j=1}^N \bar{\alpha}_{\omega_j} D_{\bar{\mu}}(F^{\bar{g},\bar{\varphi}}\circ \bar{g}_{\omega_j}(Y_j))-\frac{1}{n}\sum_{i=1}^n D_{\overline{\mu}}(X_i)  \\
    & \leq   \sup_{(g,\varphi,\alpha)\in \mathcal{G}\times  \Phi\times \mathcal{A},D\in \mathcal{D}} L(g,\varphi,\alpha,D)-L_{N,n}(g,\varphi,\alpha,D)   +  \int D_{\overline{\mu}}(x)d\overline{\mu}(x)-\int D_{\overline{\mu}}(x)d\mu^\star(x)   \\
    & \leq  \sup_{(g,\varphi,\alpha)\in \mathcal{G}\times  \Phi\times \mathcal{A},D\in \mathcal{D}} L(g,\varphi,\alpha,D)-L_{N,n}(g,\varphi,\alpha,D)   + \Delta_{\mathcal{G}\times \Phi \times \mathcal{A}}  .
\end{align*}
\end{proof}

\subsubsection{Proof of Lemma \ref{lemma:genenonempty}}\label{sec:lemma:genenonempty}
\begin{proof}
    Let $(g_i^1)_{i=1,...,m}\in\mathcal{H}^{\beta+1}_{C}(B^d(0,2\tau),\mathbb{R}^p)^m$, $(g_i^2)_{i=1,...,m}\in\dot{\mathcal{H}}^{\beta+1}_{C}(\mathbb{R}^d,\mathbb{R}^d)^m$ given by the decomposition of Propositions \ref{prop:keydecomp} and \ref{prop:decompphi}. In particular, we have that $g_i^1=\phi_{g_i^1(0)}^{-1}$ the inverse of the orthogonal projection on the tangent space $\mathcal{T}_{g_i^1(0)}$ of $\mathcal{M}^\star$ defined in \eqref{eq:phix}. Let $\pi_i$ be the projection on the tangent space of $\mathcal{M}^\star$ at the point $g_i^1(0)$ and $\Phi: \pi_i \circ g_i^1(B^d(0,2\tau))\rightarrow B^d(0,2\tau)$ defined by 
    $$\Phi(u)=(\pi_i \circ g_i^1)^{-1}(u).$$
    Then, from Theorem 4 in chapter 6 of \cite{SingularStein}, $\Phi_i$ can be extended into a map $\overline{\Phi}_i\in \mathcal{H}^{\beta+1}_C(\mathbb{R}^d,\mathbb{R}^d)$. Likewise, $g_i^1$ can be extended into a map $\overline{g}_i^1\in \mathcal{H}^{\beta+1}_C(\mathbb{R}^d,\mathbb{R}^p)$. Let $C_{\beta+1}>0$  such that $\mathcal{H}_1^{\beta+1}(\mathbb{R}^d,\mathbb{R})\subset \mathcal{B}_{\infty,\infty}^{\beta+1}(\mathbb{R}^d,\mathbb{R},C_{\beta+1})$ and $\mathcal{H}_1^{\beta+1}(\mathbb{R}^p,\mathbb{R})\subset \mathcal{B}_{\infty,\infty}^{\beta+1}(\mathbb{R}^p,\mathbb{R},C_{\beta+1})$. We have that there exists  $$|\alpha_{(\overline{g}_i^1)_k}(j,l,w)|\leq KC_{\beta+1} 2^{-j(\beta+1+d/2)}\quad \text{ with }\quad k\in \{1,...,p\},$$ $$|\alpha_{(\overline{\Phi}_i\circ\pi_i)_s}(j,l,w)|\leq KC_{\beta+1} 2^{-j(\beta+1+p/2)}\quad \text{ with }\quad s\in \{1,...,d\}$$ and $$|\alpha_{(g_i^2)_r}(j,l,w)|\leq KC_{\beta+1} 2^{-j(\beta+1+d/2)}\quad \text{ with }\quad r\in \{1,...,d\}$$
    such that
for all $u\in B^d(0,2\tau)$,
    $$\overline{g}_i^1(u)_k=\sum \limits_{j=0}^{\infty} \sum \limits_{l=1}^{2^d} \sum \limits_{w\in \{-(2K+C)2^j,...,(2K+C)2^j\}^d} \alpha_{(\overline{g}_i^1)_k}(j,l,w)\psi^d_{jlw}(u),$$
    
for all $y\in B^d(0,2\log(n)^{1/2})$,
    $$g_i^2(y)_r=\sum \limits_{j=0}^{\infty} \sum \limits_{l=1}^{2^d} \sum \limits_{w\in \{-4\log(n)^{1/2}2^j,...,4\log(n)^{1/2}2^j\}^d} \alpha_{(g_i^2)_r}(j,l,w)\psi^d_{jlw}(y),$$

and for all $x\in B^p(\overline{g}_i^1(0),2\tau)$,
$$\overline{\Phi}_i\circ\pi_i(x)_s=\sum \limits_{j=0}^{\infty} \sum \limits_{l=1}^{2^p} \sum \limits_{w\in \{-(2K+C)2^j,...,(2K+C)2^j\}^p} \alpha_{(\overline{\Phi}_i\circ\pi_i)_s}(j,l,w)\psi^p_{jlw}(x).$$

Let us now take
    $$\tilde{g}_i^1(u)_k=\sum \limits_{j=0}^{\log_2(\delta^{-1}_N)} \sum \limits_{l=1}^{2^d} \sum \limits_{w\in \{-(2K+C)2^j,...,(2K+C)2^j\}^d} \alpha_{(\overline{g}_i^1)_k}(j,l,w)\hat{\psi}^d_{jlw}(u),$$
        $$\tilde{g}_i^2(y)_r=\sum \limits_{j=0}^{\log_2(\delta^{-1}_N)} \sum \limits_{l=1}^{2^d} \sum \limits_{w\in \{-4\log(n)^{1/2}2^j,...,4\log(n)^{1/2}2^j\}^d} \alpha_{(g_i^2)_r}(j,l,w)\hat{\psi}^d_{jlw}(y)$$

and
    
\begin{equation*}
\varphi_i(x)_s=\sum \limits_{j=0}^{\log_2(\delta^{-1}_N)} \sum \limits_{l=1}^{2^d} \sum \limits_{w\in \{-(2K+C)2^j,...,(2K+C)2^j\}^p} \alpha_{(\overline{\Phi}_i\circ\pi_i)_s}(j,l,w)\hat{\psi}^p_{jlw}(x),
\end{equation*}
the low frequencies approximation of $\bar{g}_i^1,g_i^2$ and $\overline{\Phi}_i\circ\pi_i$ using the approximated wavelets $\hat{\psi}$.
Then, from Proposition \ref{prop:qualityofapproxfketa}, we have for all $u\in B^d(0,2\tau)$, $y\in B^d(0,2\log(n)^{1/2})$ and $x\in B^p(g_i^1(0),3\tau)$ that
$$\|\varphi_i(x)-\pi\circ\overline{\Phi}(x)\| + \|\tilde{g}_i^1(u)-g_i^1(u)\|+ \|\tilde{g}_i^2(y)-g_i^2(y)\|\leq C\left(n^{-1/2}+\delta_N^{\beta+1}\log(N)^2\right)$$
and
$$\|\nabla \tilde{g}_i^1(u)-\nabla g_i^1(u)\|+\|\nabla \tilde{g}_i^2(y)-\nabla g_i^2(y)\|\leq C\left(n^{-1/2}+\delta_N^{\beta}\log(N)^2\right).$$
Then, as from Proposition \ref{prop:decompphi} we have $\mathcal{C}(g^1,\pi\circ\overline{\Phi})=0$,  $1-K\tau\leq \|\nabla g_i^1(u)\frac{v}{\|v\|}\|\leq 1+K\tau$ and $C^{-1}\leq \|\nabla g_i^2(y)\frac{v}{\|v\|}\|\leq C$, we deduce that for $N$ large enough 
 $$\mathcal{C}(g,\varphi) + \mathcal{C}_N(g,\varphi) +R(g)\leq \epsilon_\Gamma.$$
    
\end{proof}

\subsubsection{Proof of Theorem \ref{theo:biaisvariancedec}}\label{sec:theo:biaisvariancedec}

\begin{proof}
From Lemma \ref{lemma:genenonempty} we have that for $N\geq n$ large enough, there exist $(g,\varphi)\in \mathcal{G}\times  \Phi$ satisfying
    $$\mathcal{C}(g,\varphi) + \mathcal{C}_N(g,\varphi) +R(g)\leq \epsilon_\Gamma.$$ Then, putting together Lemma \ref{lemma:decomptheo} and Proposition \ref{prop:statisticalerror}, we obtain
   \begin{align*}
& \mathbb{E}_{X_1,...,X_n\sim \mu^\star}[d_{\mathcal{H}^{d/2}_1}(\hat{\mu},\mu^\star)-\Delta_{\mathcal{D}_{\hat{\mu}}}-\Delta_{\mathcal{G}\times \Phi \times \mathcal{A}}]\\
 \leq &  C\log(N)^{C_2} \min \limits_{\delta_1 [0,1]}  \Biggl\{ \sqrt{\frac{(\delta_1+1/N)^2\log(N|(\mathcal{D}\circ\mathcal{F}^{\mathcal{G},\Phi}\circ \mathcal{G})_{1/N}||\mathcal{A}_{1/N}|)}{N}}\\
 &+\frac{1}{\sqrt{N}}(1+\delta_1^{(1-\frac{d}{2(\frac{d}{2}\wedge(\beta+1))})}+\log(\delta^{-1}_1)\mathds{1}_{\{\frac{d}{2}\wedge(\beta+1)= \frac{d}{2}\}})\Biggl\}\\
& + C\log(n)^{C_2}\min \limits_{\delta_2 \in [0,1]}  \Biggl\{  \sqrt{\frac{(\delta_2+1/n)^2\log(n|(\mathcal{D} \circ \phi_1)_{1/n}|)}{n}}+\frac{C}{\sqrt{n}}\log(\delta_2^{-1})\Biggl\}.
\end{align*}

Now, taking $\delta_2=n^{-\frac{d/2}{2\beta+d}}$
we get the result.

\end{proof}

\subsection{Proofs of the bounds on the approximation errors}

\subsubsection{Proof of Proposition \ref{prop:deltaG}}\label{sec:prop:deltaG}

\begin{proof}
    Let $(g_i^1)_{i=1,...,m}\in\mathcal{H}^{\beta+1}_{C}(B^d(0,2\tau),\mathbb{R}^p)^m$, $(g_i^2)_{i=1,...,m}\in\mathcal{H}^{\beta}_{C}(\mathbb{R}^d,\mathbb{R}^d)^m$ given by the decomposition of Propositions \ref{prop:keydecomp} and \ref{prop:decompphi}. In particular, we have that $g_i^1=\phi_{g_i^1(0)}^{-1}$ is the inverse of the orthogonal projection on the tangent space $\mathcal{T}_{g_i^1(0)}$ of $\mathcal{M}^\star$ defined in \eqref{eq:phix}. Let $\pi_i$ be the projection on the tangent space of $\mathcal{M}^\star$ at the point $g_i^1(0)$ and $\Phi_i: \pi_i \circ g_i^1(B^d(0,2\tau))\rightarrow B^d(0,2\tau)$ defined by 
    $$\Phi_i(x)=(\pi_i \circ g_i^1)^{-1}(x).$$
    Then, from Theorem 4 in chapter 6 of \cite{SingularStein}, $\Phi_i$ can be extended into a map $\overline{\Phi}_i\in \mathcal{H}^{\beta+1}_C(\mathbb{R}^d,\mathbb{R}^d)$ and $g_i^1$ can be extended into a map $\overline{g}_i^1\in \mathcal{H}^{\beta+1}_C(\mathbb{R}^d,\mathbb{R}^p)$. We have that there exists, $|\alpha_{g_i^2}(j,l,w)|\leq KC_{\beta+1} 2^{-j(\beta+d/2)}$, $|\alpha_{(\overline{g}_i^1)_k}(j,l,w)|\leq KC_{\beta+1} 2^{-j(\beta+1+d/2)}$ with $k\in \{1,...,p\}$, $|\alpha_{(\overline{\Phi}_i\circ\pi_i)_s}(j,l,w)|\leq KC_{\beta+1} 2^{-j(\beta+1+d/2)}$ with $s\in \{1,...,d\}$, such that
for all $u\in B^d(0,2\tau)$
    $$\overline{g}_i^1(u)_k=\sum \limits_{j=0}^{\infty} \sum \limits_{l=1}^{2^d} \sum \limits_{w\in \{-K2^j,...,K2^j\}^d} \alpha_{(\overline{g}_i^1)_k}(j,l,w)\psi^d_{jlw}(u),$$
for all $y\in B^d(0,2\log(n)^{1/2})$
    $$g_i^2(y)_r=\sum \limits_{j=0}^{\infty} \sum \limits_{l=1}^{2^d} \sum \limits_{w\in \{-4\log(n)^{1/2}2^j,...,4\log(n)^{1/2}2^j\}^d} \alpha_{(g_i^2)_r}(j,l,w)\psi^d_{jlw}(y),$$
and for all $x\in B^p(\overline{g}_i^1(0),3\tau)$
$$\overline{\Phi}_i\circ\pi_i(x)_s=\sum \limits_{j=0}^{\infty} \sum \limits_{l=1}^{2^d} \sum \limits_{w\in \{-K2^j,...,K2^j\}^p} \alpha_{(\overline{\Phi}_i\circ\pi_i)_s}(j,l,w)\psi^p_{jlw}(x).$$

Let us now take
    $$\tilde{g}_i^1(u)_k=\sum \limits_{j=0}^{\log_2(\delta^{-1}_N)} \sum \limits_{l=1}^{2^d} \sum \limits_{w\in \{-(2K+C)2^j,...,(2K+C)2^j\}^d} \alpha_{(\overline{g}_i^1)_k}(j,l,w)\hat{\psi}^d_{jlw}(u),$$
        $$\tilde{g}_i^2(y)_r=\sum \limits_{j=0}^{\log_2(\delta^{-1}_N)} \sum \limits_{l=1}^{2^d} \sum \limits_{w\in \{-4\log(n)^{1/2}2^j,...,4\log(n)^{1/2}2^j\}^d} \alpha_{(g_i^2)_r}(j,l,w)\hat{\psi}^d_{jlw}(y)$$

and
    
\begin{equation*}
\varphi_i(x)_s=\sum \limits_{j=0}^{\log_2(\delta^{-1}_N)} \sum \limits_{l=1}^{2^d} \sum \limits_{w\in \{-(2K+C)2^j,...,(2K+C)2^j\}^p} \alpha_{(\overline{\Phi}_i\circ\pi_i)_s}(j,l,w)\hat{\psi}^p_{jlw}(x).
\end{equation*}

Then from Lemma \ref{lemma:bouundcoefwav}, for all $ j\leq \log_2(\delta^{-1}_N)$, we have that 
$$
|\alpha_{(\tilde{g}_i^1)_k}(j,l,w)-\alpha_{(\overline{g}_i^1)_k}(j,l,w)\mathds{1}_{j\leq \log_2(\delta^{-1})}|\leq Cn^{-1/2}2^{-j(\beta+1+d/2)},
$$
$$
|\alpha_{(\tilde{g}_i^2)_k}(j,l,w)-\alpha_{(g_i^2)_k}(j,l,w)\mathds{1}_{j\leq \log_2(\delta^{-1})}|\leq Cn^{-1/2}2^{-j(\beta+1+d/2)},
$$
$$
|\alpha_{\varphi_i}(j,l,w)-\alpha_{(\overline{\Phi}_i\circ\pi_i)_s}(j,l,w)\mathds{1}_{j\leq \log_2(\delta^{-1})}|\leq Cn^{-1/2}2^{-j(\beta+1+p/2)}.
$$
Furthermore, taking $g=\tilde{g}_i^1\circ \Psi^{-1}\circ \tilde{g}_i^2$,  we have that 
\begin{equation}\label{eq:hfoehdkdhhdburbfur}
\mathcal{C}(g,\varphi) + \mathcal{C}_N(g,\varphi) +R(g)\leq \epsilon_\Gamma,
\end{equation}
as $(g,\varphi)$ are the function used in the proof of Lemma \ref{lemma:genenonempty} Section \ref{sec:lemma:genenonempty}  to show the existence of a couple satisfying the constraint \eqref{eq:hfoehdkdhhdburbfur}.

Let us now give bound on the approximation error. Writing, $g^\star=\overline{g}_i^1\circ \Psi^{-1}\circ g_i^2$, we have for all $x\in \mathcal{M}^\star$ that $F^{g^\star,\overline{\Phi}_i\circ \pi_i}(x)=x$, so we deduce that
\begin{align}\label{align:cfghjgyujpmlk}
    &\mathbb{E}_{\substack{\omega\sim \mathcal{U}_m,X\sim \mu^\star\\Y\sim \gamma_d^n}}[ D(F^{g,\varphi}\circ g_{\omega}(Y))-D(X)]\nonumber\\
     = &\sum_{i=1}^m  \int_{\mathbb{R}^d}D\big(F^{g,\varphi}\circ \tilde{g}_i^1\circ \Psi^{-1}\circ \tilde{g}_i^2(y)\big)d\gamma_d^n(y) - \int_{\mathbb{R}^d}D\big(F^{g^\star,\overline{\Phi}_i\circ \pi_i}\circ \overline{g}_i^1\circ \Psi^{-1}\circ g_i^2(y)\big)d\gamma_d(y)\nonumber\\
     \leq &Cn^{-1}+\sum_{i=1}^m  \int_{\mathbb{R}^d}D\big(F^{g,\varphi}\circ \tilde{g}_i^1\circ \Psi^{-1}\circ \tilde{g}_i^2(y)\big)d\gamma_d^n(y) - \int_{\mathbb{R}^d}D\big(F^{g^\star,\overline{\Phi}_i\circ \pi_i}\circ \overline{g}_i^1\circ \Psi^{-1}\circ g_i^2(y)\big)d\gamma_d^n(y)\nonumber\\
    \leq &Cn^{-1}+ \sum_{i=1}^m \int_{\mathbb{R}^d}\Biggl(D\big(F^{g,\varphi}\circ \tilde{g}_i^1\circ \Psi^{-1}\circ \tilde{g}_i^2(y)\big)-D\big(F^{g,\varphi}\circ \tilde{g}_i^1\circ \Psi^{-1}\circ g_i^2(y)\big)\nonumber\\
    &+D\big(F^{g,\varphi}\circ \tilde{g}_i^1\circ \Psi^{-1}\circ g_i^2(y)\big)- D\big(F^{g^\star,\overline{\Phi}_i\circ \pi_i}\circ \overline{g}_i^1\circ \Psi^{-1}\circ g_i^2(y)\big)\Biggl)d\gamma_d^n(y).
\end{align}
On one hand we have
\begin{align*}
    &\int_{\mathbb{R}^d}\Big(D\big(F^{g,\varphi}\circ \tilde{g}_i^1\circ \Psi^{-1}\circ \tilde{g}_i^2(y)\big)-D\big(F^{g,\varphi}\circ \tilde{g}_i^1\circ \Psi^{-1}\circ g_i^2(y))\Big)  \gamma_d^n(y)d\lambda^d(y)\\
    = &\int_{\mathbb{R}^d}\int_0^1\Big\langle \nabla \big(D\circ F^{g,\varphi}\circ \tilde{g}_i^1\circ \Psi^{-1}\big)\big(g_i^2(y)+t(\tilde{g}_i^2(y)-g_i^2(y))\big), \tilde{g}_i^2(y)-g_i^2(y)\Big\rangle dt  \gamma_d^n(y)d\lambda^d(y)
    \\
    =  & \sum_{j=1}^p\int_0^1\int_{\mathbb{R}^d} \partial_j \big(D\circ F^{g,\varphi}\circ \tilde{g}_i^1\circ \Psi^{-1}\big)\big(g_i^2(y)+t(\tilde{g}_i^2(y)-g_i^2(y))\big)\gamma_d^n(y) (\tilde{g}_i^2(y)-g_i^2(y))_jd\lambda^d(y)dt.
\end{align*}
Fixing $t\in [0,1]$ and writing $(\alpha_{\bar{D}_{j}}(k,l,z))_{k,l,z}$ for the wavelet coefficients of the function 
$$D_{j}(y):= \partial_j \big(D\circ F^{g,\varphi}\circ \tilde{g}_i^1\circ \Psi^{-1}\big)\big(g_i^2(y)+t(\tilde{g}_i^2(y)-g_i^2(y))\big)\gamma_d^n(y),$$ 
we have
\begin{align}\label{align:mlkiuhbvghjh}
   & \int_{\mathbb{R}^d} \partial_j \big(D\circ F^{g,\varphi}\circ \tilde{g}_i^1\circ \Psi^{-1}\big)\big(g_i^2(y)+t(\tilde{g}_i^2(y)-g_i^2(y))\big)\gamma_d^n(y) (\tilde{g}_i^2(y)-g_i^2(y))_jd\lambda^d(y)\nonumber\\
   &= \sum \limits_{k=0}^{\infty} \sum \limits_{l=1}^{2^d} \sum \limits_{z\in \{-4\log(n)^{1/2}2^k,...,4\log(n)^{1/2}2^k\}^d} \alpha_{\bar{D}_{j}}(k,l,z)(\alpha_{(\tilde{g}_i^2)_{j}}(k,l,z)-\alpha_{(g_i^2)_{j}}(j,l,z))\nonumber\\
   &\leq \|\tilde{g}_i^2-g_i^2\|_{\mathcal{B}^{-((\gamma-1)\wedge\beta)}_{1,1}}\nonumber\\
   & \leq C\log(n)^{C_2} (\delta_N^{(\gamma-1)\wedge\beta+\beta+1}+n^{-1/2}),
\end{align}
using that $D_j\in \mathcal{H}^{(\gamma-1)\wedge\beta}_{C\log(n)^{2}}$ by Proposition \ref{prop:regularityfandg} and using Proposition \ref{prop:qualityofapproxfketa} to obtain the bound on the Besov norm. Therefore, we conclude that
\begin{align*}
    \int_{\mathbb{R}^d}\Big(D\big(F^{g,\varphi}\circ \tilde{g}_i^1\circ \Psi^{-1}\circ \tilde{g}_i^2(y)\big)-D\big(F^{g,\varphi}\circ \tilde{g}_i^1\circ \Psi^{-1}\circ g_i^2(y))\Big)  \gamma_d^n(y)d\lambda^d(y)\leq C\log(n)^2 (\delta_N^{(\gamma-1)\wedge\beta+\beta+1}+n^{-1/2}).
\end{align*}
Let us now bound the second term of \eqref{align:cfghjgyujpmlk}. We have
\begin{align*}
   & \int_{\mathbb{R}^d}\Big(D\big(F^{g,\varphi}\circ \tilde{g}_i^1\circ \Psi^{-1}\circ g_i^2(y)\big)- D\big(F^{g^\star,\overline{\Phi}_i\circ \pi_i}\circ \overline{g}_i^1\circ \Psi^{-1}\circ g_i^2(y)\big)\Big)d\gamma_d^n(y)\\
    &= \int_{B^d(0,\tau)}\Big(D\big(F^{g,\varphi}\circ \tilde{g}_i^1(u)\big)- D\big(F^{g^\star,\overline{\Phi}_i\circ \pi_i}\circ \overline{g}_i^1 (u)\big)\Big)\zeta_i(u)d\lambda^d(u)
\end{align*}
with $\zeta_i\in \mathcal{H}^\beta_C(\mathbb{R}^d,\mathbb{R})$ the density of the probability measure $(\Psi^{-1}\circ g_i^2)_{\#}\gamma_d^n$. Then, doing like for the other term, we have
\begin{align*}
    &\int_{\mathbb{R}^d}\Big(D\big(F^{g,\varphi}\circ \tilde{g}_i^1(u)\big)- D\big(F^{g^\star,\overline{\Phi}_i\circ \pi_i}\circ \overline{g}_i^1 (u)\Big)  \zeta_i(u)d\lambda^d(u)\\
    = &\int_{\mathbb{R}^d}\int_0^1\Big\langle \nabla D\Big(F^{g^\star,\overline{\Phi}_i\circ \pi_i}\circ \overline{g}_i^1(u)+ t(F^{g,\varphi}\circ \tilde{g}_i^1(u)-F^{g^\star,\overline{\Phi}_i\circ \pi_i}\circ \overline{g}_i^1(u))\Big),F^{g,\varphi}\circ \tilde{g}_i^1(u)-F^{g^\star,\overline{\Phi}_i\circ \pi_i}\circ \overline{g}_i^1(u)\Big\rangle dt  \zeta_i(u)d\lambda^d(u)
    \\
    =  & \sum_{j=1}^p\int_0^1\int_{\mathbb{R}^d} \partial_j D\Big(F^{g^\star,\overline{\Phi}_i\circ \pi_i}\circ \overline{g}_i^1(u)+ t(F^{g,\varphi}\circ \tilde{g}_i^1(u)-F^{g^\star,\overline{\Phi}_i\circ \pi_i}\circ \overline{g}_i^1(u))\Big)\zeta_i(u)\\
    & \times (F^{g,\varphi}\circ \tilde{g}_i^1(u)-F^{g,\varphi}\circ g_i^1(u)+F^{g,\varphi}\circ g_i^1(u)-F^{g^\star,\overline{\Phi}_i\circ \pi_i}\circ g_i^1(u))_jd\lambda^d(u)dt\\
    =  &  \sum_{j_1,j_2=1}^p\int_0^1\int_0^1\int_{\mathbb{R}^d} \partial_{j_1} D\Big(F^{g^\star,\overline{\Phi}_i\circ \pi_i}\circ \overline{g}_i^1(u)+ t(F^{g,\varphi}\circ \tilde{g}_i^1(u)-F^{g^\star,\overline{\Phi}_i\circ \pi_i}\circ \overline{g}_i^1(u))\Big)\zeta_i(u)\\
    &\times \partial_{j_2} (F^{g,\varphi})_{j_1}\Big( \tilde{g}_i^1(u)+s(g_i^1(u)-\tilde{g}_i^1(u)\Big)(\tilde{g}_i^1(u)-g_i^1(u))_{j_2}d\lambda^d(u)dsdt\\
   & +   \sum_{j=1}^p\int_0^1\int_{\mathbb{R}^d} \partial_j D\Big(F^{g^\star,\overline{\Phi}_i\circ \pi_i}\circ \overline{g}_i^1(u)+ t(F^{g,\varphi}\circ \tilde{g}_i^1(u)-F^{g^\star,\overline{\Phi}_i\circ \pi_i}\circ \overline{g}_i^1(u))\Big)\zeta_i(u)(F^{g,\varphi}\circ g_i^1(u)-F^{g^\star,\overline{\Phi}_i\circ \pi_i}\circ g_i^1(u))_jd\lambda^d(u)dt.
\end{align*}
Now both these terms can be treated like \eqref{align:mlkiuhbvghjh} by writing the $L^2$ scalar product as the sum of the product of the wavelet coefficients and using the regularity of each function.

\end{proof}

\subsubsection{Proof of Proposition \ref{prop:DeltaD}}\label{sec:prop:DeltaD}

\begin{proof}
Let $\pi_{\hat{\mu}}$ be an optimal transport plan for the Euclidean distance between $\hat{\mu}=(F^{\hat{g},\hat{\varphi}}\circ \hat{g})_{\# \hat{\alpha}\gamma_d^n}$  and $\mu^\star$. For
$D \in \mathcal{D}$
we have
\begin{align*}
    &\mathbb{E}_{\substack{Y\sim \hat{\mu}\\
    X\sim \mu^\star}}[h_{\hat{\mu}}(Y)-h_{\hat{\mu}}(X)-(D(Y)-D(X))]\\
    & = \int_{\mathbb{R}^p\times \mathbb{R}^p} \left(h_{\hat{\mu}}(Y)-h_{\hat{\mu}}(X)-(D(Y)-D(X))\right)d\pi_{\hat{\mu}}(x,y)\\
    & = \int_{\mathbb{R}^p\times \mathbb{R}^p}  \int_0^1\Big\langle \nabla h_{\hat{\mu}}(y+t(x-y))-\nabla D(y+t(x-y)),x-y\Big\rangle dtd\pi_{\hat{\mu}}(x,y)\\
    & \leq \|\nabla h_{\hat{\mu}}-\nabla D\|_\infty \int_{\mathbb{R}^p\times \mathbb{R}^p}  \|x-y\| d\pi_{\hat{\mu}}(x,y)\\
    &  = \|\nabla h_{\hat{\mu}}-\nabla D\|_\infty \sup \limits_{f\in \text{Lip}_1}\mathbb{E}_{\substack{Y\sim \hat{\mu}\\
    X\sim \mu^\star}}[  f(Y)-f(X)].
\end{align*}
On the other hand, from Lemma \ref{lemma:firstbadbound} and Proposition \ref{prop:statisticalerrorC} we have for $n\in \mathbb{N}_{>0}$ large enough, 
\[
\sup_{D \in \mathcal{D}} L(\hat{g}, \hat{\varphi}, \hat{\alpha}, D) + \mathcal{C}(\hat{g}, \hat{\varphi}) + R(\hat{g}) \leq C_2^{-1},
\]
with $C_2>0$ given by Proposition \ref{prop:densitychecked}. Then, applying Proposition \ref{prop:densitychecked}, we have that with probability at least $1-1/n$, for  $\hat{\mathcal{M}}=\bigcup_{i=1}^m F^{\hat{g},\hat{\varphi}}\circ \hat{g}_{i}(B^d(0,\log(n)^{1/2}+1))$, $\hat{\mu}$  satisfies the $(\beta,K)$-density condition on $\hat{\mathcal{M}}$. Therefore, using Theorem \ref{theo:theineqGAW} we get
\begin{align*}
\sup \limits_{f\in \text{Lip}_1}\mathbb{E}_{\substack{Y\sim \hat{\mu}\\
    X\sim \mu^\star}}[  f(Y)-f(X)]
&\leq (2K+1)\sup \limits_{f\in \mathcal{H}^1_1}\mathbb{E}_{\substack{Y\sim \hat{\mu}\\
    X\sim \mu^\star}}[  f(Y)-f(X)]\\
& \leq C\log\left(1+d_{\mathcal{H}^{\gamma}_1}(\hat{\mu},\mu^\star)^{-1}\right)^{C_2} d_{\mathcal{H}^{\gamma}_1}(\hat{\mu},\mu^\star)^{\frac{\beta+1}{\beta+\gamma}}.
\end{align*}    
\end{proof}

\subsection{Minimax bounds}

\subsubsection{A first suboptimal bound}
In order to apply the geometric and regularity results from Section~\ref{sec:reguofes}, we require the estimator \( \hat{\mu} \) to already be sufficiently close to the target distribution \( \mu^\star \). To this end, we first establish a preliminary bound showing that \( \hat{\mu} \) achieves a sub-optimal rate of convergence under the \( d_{\mathcal{H}^{d/2}_1} \) IPM.

\begin{lemma}\label{lemma:firstbadbound}
 Let $\mu^\star$ a probability measure satisfying Assumption \ref{assump:model}. Then, with probability at least $1-1/n$ we have
    $$d_{\mathcal{H}^{d/2}_1}(\hat{\mu},\mu^\star)\leq C\log(n)^{C_2} n^{-\frac{\beta\wedge d/2}{2\beta+d}}.$$
\end{lemma}

This intermediate result ensures that the estimator lies within a controlled neighborhood of \( \mu^\star \), enabling the application of the regularity-based interpolation arguments of Theorem \ref{theo:theineqGAW}. 

\begin{proof}[Proof of lemma \ref{lemma:firstbadbound}]
Let $A>0$ be the constant given by Lemma \ref{lemma:genenonempty} such that if $N\geq A$, there exists $(g,\varphi)\in \mathcal{G}\times  \Phi$ satisfying the constraint \eqref{eq:estimator}. Supposing that $\mu^\star$ satisfies Assumptions  \ref{assump:model}, in the case $N<A$, we have 
\begin{align}\label{align:trivresult}
    d_{\mathcal{H}^{d/2}_1}(\hat{\mu},\mu^\star)\leq 2 \leq 2AN^{-1},
\end{align}
so we have the result. Let us then focus on the case $N\geq A$. From Lemma \ref{lemma:decomptheo} we have
\begin{align*}
    \mathbb{P}\Big(d_{\mathcal{H}^{d/2}_1}(\hat{\mu},\mu)\geq \alpha \Big)\leq& \mathbb{P}\Big(\Delta_{\mathcal{D}_{\hat{\mu}}}\geq \alpha/3 \Big)+\mathbb{P}\Big(\Delta_{\mathcal{G}\times \Phi \times \mathcal{A}}\geq \alpha/3 \Big)\\    &+\mathbb{P}\Big(\sup_{(g,\varphi,\alpha)\in \mathcal{G}\times  \Phi\times \mathcal{A},D\in \mathcal{D}} L(g,\varphi,\alpha,D)-L_{N,n}(g,\varphi,\alpha,D)\geq  \alpha/6 \Big).
\end{align*}
Now,
\begin{align*}
    \mathbb{P}\Big(\Delta_{\mathcal{D}_{\hat{\mu}}}\geq \alpha/3 \Big)& \leq \mathbb{P}\Big(\sup \limits_{(g,\varphi,\alpha)\in \mathcal{G}\times  \Phi\times \mathcal{A}} \bigl\{ d_{\mathcal{H}_1^{d/2}}((F^{g,\varphi}\circ g)_{\# \alpha \gamma_d^n},\mu^\star)-d_{\mathcal{D}}((F^{g,\varphi}\circ g)_{\# \alpha \gamma_d^n},\mu^\star)\bigl\}\geq \alpha/3 \Big)\\
    & \leq \mathbb{P}\Big(2\sup_{D^\star\in \mathcal{H}_1^{d/2}(B^p(0,K),\mathbb{R})} \inf\limits_{D\in \mathcal{D}} \|D-D^\star\|_\infty\geq \alpha/3 \Big)\\
    & \leq \mathbb{P}\Big(2\sup_{D^\star\in \mathcal{H}_1^{d/2}(B^p(0,K),\mathbb{R})} \inf\limits_{D\in \mathcal{D}} \|D-D^\star\|_{\mathcal{B}^{0,2}_{\infty,\infty}}\geq \alpha/3 \Big)\\
    & \leq \mathbb{P}\Big( C \log(n)^2 n^{-\frac{d/2}{2\beta+d}}\geq \alpha/3 \Big),
\end{align*}
using Proposition \ref{prop:qualityofapproxfketa} to obtain the bound on the Besov norm. so for $\alpha=C\log(n)^{C_2} n^{-\frac{d/2}{2\beta+d}}$, we have with probability $1$ that $\Delta_{\mathcal{D}_{\hat{\mu}}}\leq \alpha/3$. On the other hand, from Proposition \ref{prop:deltaG} we have 
\begin{equation}\label{eq:approxGprecise}
\Delta_{\mathcal{G}\times \Phi \times \mathcal{A}}\leq C\log(N)^2N^{-\frac{(\beta + d/2)\wedge (2\beta+1)}{2\beta+d}}.
\end{equation}
Now, from Proposition \ref{prop:keydecompunedeplus}, there exists a collection of maps $(\phi_i)_{i=1,...,m}\in\mathcal{H}^{\beta+1}_{C}(\mathbb{R}^d,\mathbb{R}^p)^m$, non-negative weight functions $(\zeta)_{i=1,...,m}\in\mathcal{H}^{\beta+1}_{C}(B^d(0,\tau),\mathbb{R})^m$ such that
$$
\frac{1}{n}\sum_{j=1}^n \mathbb{E}_{X\sim \mu}[ D(X)]-  D(X_j)=\frac{1}{n}\sum \limits_{i=1}^m\sum_{j=1}^n  \int_{\mathbb{R}^d}D(\phi_i(u))\zeta_i(u)d\lambda^d(u)-D(\phi_i(Y_j))\mathds{1}_{\{T_j=i\}},
$$
with $T_1, ..., T_n$ iid sample of multinomial distribution $\textbf{Mult}(\{\int \zeta_1,...,\int\zeta_m\})$ and $Y_j$ of law $\zeta_{T_j}/{\int \zeta_{T_j}}$. Then, doing as for \eqref{align:boundonprobacovering}, we have 
\begin{align}\label{align:petitsupprob}
    \mathbb{P}&\Big(\sup_{(g,\varphi,\alpha)\in \mathcal{G}\times  \Phi\times \mathcal{A},D\in \mathcal{D}} L(g,\varphi,\alpha,D)-L_{N,n}(g,\varphi,\alpha,D)\geq  \alpha/6 \Big)\nonumber\\
    \leq &  \mathbb{P}\Big(\sup \limits_{D\in \mathcal{D} }  \frac{1}{n}\sum_{j=1}^n  \sum_{i=1}^m\int_{\mathbb{R}^d}D(\phi_i(u))\zeta_i(u)d\lambda^d(u)-D(\phi_i(Y_j))\mathds{1}_{\{T_j=i\}}\geq \alpha/12 \Big)\nonumber\\
    &+ \mathbb{P}\Big(\sup_{\substack{g\in \mathcal{G},\varphi\in \Phi,\\\alpha\in \mathcal{A},D\in \mathcal{D}}} \frac{1}{N} \sum_{j=1}^N \sum_{i=1}^m \int_{\mathbb{R}^d} D(F^{g,\varphi}\circ g_{i}(y))d\gamma_d^n(u)- D(F^{g,\varphi}\circ g_{i}(Y_j))\mathds{1}_{\{\omega_j=i\}}\geq \alpha/12 \Big)\nonumber\\
    \leq & |(\mathcal{D}_{\phi})_{1/n}|  \exp\left(-\frac{n(\alpha/12 -\frac{2}{n})^2}{2(1+\frac{2}{n})^2}\right) +|(\mathcal{D}\circ\mathcal{F}^{\mathcal{G},\Phi}\circ \mathcal{G})_{1/N}||\mathcal{A}_{1/N}|\exp\left(-\frac{N(\alpha/12 -\frac{2}{N})^2}{2(1+\frac{2}{N})^2}\right),
\end{align}
for $$\mathcal{D}_{\phi}=\big\{D\circ \phi_i|i\in \{1,....,m\},D\in \mathcal{D}\big\}.$$

Then, solving \eqref{align:petitsupprob}$=1/n$ with respect to $\alpha$ and recalling that $N\geq n$, we obtain that with probability at least $1-1/n$
\begin{align*}
&\sup_{(g,\varphi,\alpha)\in \mathcal{G}\times  \Phi\times \mathcal{A},D\in \mathcal{D}} L(g,\varphi,\alpha,D)-L_{N,n}(g,\varphi,\alpha,D)\\
& \leq C\log(n)^{C_2}\left(\sqrt{\frac{(1+1/n)^2\log(|((\mathcal{D}\circ \phi)_{1/n}|)}{n} }+\sqrt{\frac{(1+1/N)^2\log(|(\mathcal{D}\circ\mathcal{F}^{\mathcal{G},\Phi}\circ \mathcal{G})_{1/N}||\mathcal{A}_{1/N}|N)}{N} }\right).
\end{align*}

From Proposition 4.3 in \cite{stephanovitch2023wasserstein} we have 
$$\log(| \hat{\mathcal{F}}^{k,\eta}_\delta|_\epsilon)\leq C\delta^{-k}\log( \delta^{-1})\log(\epsilon^{-1}),$$
so using Lemma \ref{lemma:coveringd} we deduce that
\begin{equation}\label{eq:coveringprecise}
\log(|\mathcal{G}_{1/N}||(\Phi \circ \mathcal{G})_{1/N}||\mathcal{A}_{1/N}|)\leq C \log(N)^2 N^{\frac{d}{2\beta+d}}.
\end{equation}

On the other hand, from Lemma \ref{lemma:coveringd} we have 
\begin{equation}\label{eq:coveringpreciseD}
\log(|(\mathcal{D}\circ \phi)_{1/n}|)\leq  C \log(n)^2 n^{\frac{d}{2\beta+d}},
\end{equation}
so we deduce that
\begin{align}\label{align:boundcoverdcircf}
    \log(|(\mathcal{D}\circ\mathcal{F}^{\mathcal{G},\Phi}\circ \mathcal{G})_{1/N}||\mathcal{A}_{1/N}|) \leq&  C \log(N)^2 N^{\frac{d}{2\beta+d}}+\log(|(\mathcal{F}^{\mathcal{G},\Phi}\circ \mathcal{G}||\mathcal{A}_{1/N}|)\nonumber\\
    \leq&  C \log(N)^2 N^{\frac{d}{2\beta+d}}+\log(|\mathcal{G}_{1/N}||(\Phi \circ \mathcal{G})_{1/N}||\mathcal{A}_{1/N}|)\nonumber\\
    \leq& C \log(N)^2 N^{\frac{d}{2\beta+d}}.
\end{align}
We deduce that with probability at least $1-1/n$
\begin{align*}
&\sup_{(g,\varphi,\alpha)\in \mathcal{G}\times  \Phi\times \mathcal{A},D\in \mathcal{D}} L(g,\varphi,\alpha,D)-L_{N,n}(g,\varphi,\alpha,D)\\
& \leq C\log(n)^{C_2}\left(\sqrt{\frac{\log(n)}{n}}\log(n)^2 n^{\frac{d}{2\beta+d}}+\sqrt{\frac{\log(N)}{N}}\log(N)^2 N^{\frac{d}{2\beta+d}}\right)\\
& \leq C\log(n)^{C_2} n^{-\frac{\beta\wedge d/2}{2\beta+d}}. 
\end{align*}

\end{proof}

\subsubsection{Proof of Theorem \ref{theo:boundonddsur2}}\label{sec:theo:boundonddsur2}

\begin{proof} As justified in \eqref{align:trivresult}, we can suppose that $N$ is large enough so that we can apply Lemma \ref{lemma:genenonempty}. We are going to give estimates on the different terms appearing in the bound given by Theorem \ref{theo:biaisvariancedec}.
Suppose first that $d/2\leq \beta+1$, then taking $N=n$ we have from \eqref{eq:coveringprecise} and \eqref{align:boundcoverdcircf} that
$$\log(|(\mathcal{D}\circ\mathcal{F}^{\mathcal{G},\Phi}\circ \mathcal{G})_{1/N}||\mathcal{A}_{1/N}|)\leq C \log(n)^2 n^{\frac{d}{2\beta+d}}$$
and from \eqref{eq:approxGprecise}
$$\mathbb{E}[\Delta_{\mathcal{G}\times \Phi \times \mathcal{A}}]\leq Cn^{-\frac{(\beta + d/2)\wedge (2\beta+1)}{2\beta+d}}=Cn^{-\frac{1}{2}}.$$

On the other hand, from Lemma \ref{lemma:coveringd} we have 
$$\log(|(\mathcal{D}\circ \phi)_{1/n}|)\leq  C \log(n)^2 n^{\frac{d}{2\beta+d}}.$$

Furthermore, from Proposition \ref{prop:DeltaD}, taking $D_{\hat{\mu}}\in \mathcal{D}$ as the low frequency wavelet approximation of the function
$$h_{\hat{\mu}} \in \argmax_{h\in \mathcal{H}^{d/2}_1} \int h(x)d\hat{\mu}(x)-\int h(x)d\mu^\star(x),$$
we have
\begin{align*}
    \mathbb{E}[\Delta_{\mathcal{D}_{\hat{\mu}}}]\leq & C\mathbb{E}[\|\nabla h_{\hat{\mu}}-\nabla D_{\hat{\mu}}\|_\infty\log\left(1+d_{\mathcal{H}^{d/2}_1}(\hat{\mu},\mu^\star)^{-1}\right)^{C_2} d_{\mathcal{H}^{d/2}_1}(\hat{\mu},\mu^\star)^{\frac{\beta+1}{\beta+d/2}}]+Cn^{-1}\\
    \leq & C\mathbb{E}[\|h_{\hat{\mu}}-D_{\hat{\mu}}\|_{\mathcal{B}^{1,2}_{\infty,\infty}}\log\left(n\right)^{C_2} d_{\mathcal{H}^{d/2}_1}(\hat{\mu},\mu^\star)^{\frac{\beta+1}{\beta+d/2}}]+Cn^{-1}
    \\
    \leq & C\mathbb{E}[\log\left(n\right)^{C_2}  n^{-\frac{d/2-1}{2\beta+d}}d_{\mathcal{H}^{d/2}_1}(\hat{\mu},\mu^\star)^{\frac{\beta+1}{\beta+d/2}}]+Cn^{-1}.
\end{align*}

Then, taking $\delta=n^{-\frac{d/2}{2\beta+d}}$, we get from Theorem \ref{theo:biaisvariancedec} that
\begin{align*}
    \mathbb{E}[d_{\mathcal{H}^{d/2}_1}(\hat{\mu},\mu)]\leq C\log(n)^{C_2}\left(n^{-\frac{1}{2}}+   n^{-\frac{d/2-1}{2\beta+d}}\mathbb{E}[d_{\mathcal{H}^{d/2}_1}(\hat{\mu},\mu^\star)^{\frac{\beta+1}{\beta+d/2}}]\right),
\end{align*}
so if $ n^{-\frac{d/2-1}{2\beta+d}}\mathbb{E}[d_{\mathcal{H}^{d/2}_1}(\hat{\mu},\mu^\star)^{\frac{\beta+1}{\beta+d/2}}]\leq n^{-\frac{1}{2}}$ we have the result. Otherwise, using Jensen's inequality we get
\begin{align*}
    \mathbb{E}[d_{\mathcal{H}^{d/2}_1}(\hat{\mu},\mu)]\leq& C\log(n)^{C_2} n^{-\frac{d/2-1}{2\beta+d}}\mathbb{E}[d_{\mathcal{H}^{d/2}_1}(\hat{\mu},\mu^\star)^{\frac{\beta+1}{\beta+d/2}}]\\
    \leq& C\log(n)^{C_2} n^{-\frac{d/2-1}{2\beta+d}}\mathbb{E}[d_{\mathcal{H}^{d/2}_1}(\hat{\mu},\mu^\star)]^{\frac{\beta+1}{\beta+d/2}},
\end{align*}
so
\begin{align*}
    \mathbb{E}[d_{\mathcal{H}^{d/2}_1}(\hat{\mu},\mu^\star)]^{\frac{d/2-1}{\beta+d/2}}\leq& C\log(n)^{C_2} n^{-\frac{d/2-1}{2\beta+d}},
\end{align*}
which finally gives 
\begin{align*}
    \mathbb{E}[d_{\mathcal{H}^{d/2}_1}(\hat{\mu},\mu^\star)]\leq& C\log(n)^{C_2} n^{-\frac{1}{2}}.
\end{align*}

Now in the case $d/2 \geq \beta+1$, taking $N=n^{\frac{2\beta+d}{4\beta+2}}$ and $\delta=n^{-\frac{\beta+1}{4\beta+2}}$ we get the same result.

\end{proof}

\subsubsection{Proof of Theorem \ref{theo:boundongamma}}\label{sec:theo:boundongamma}

\begin{proof}
From Lemma \ref{lemma:firstbadbound} we have that with probability at least $1-1/n$, there exists $(g,\varphi)\in \mathcal{G}\times  \Phi$ satisfying the constraint \eqref{eq:estimator} and 
    $$d_{\mathcal{H}^{d/2}_1}(\hat{\mu},\mu)\leq C\log(n)^{C_2} n^{-\frac{(2\beta+1)\wedge d/2}{2\beta+d}}.$$ 
    Then, from Proposition \ref{prop:densitychecked}, the measure $\hat{\mu}$ satisfies Assumption  \ref{assump:model}. Now, if $\gamma\geq d/2$, applying Theorem \ref{theo:boundonddsur2} we have the result. Let us then focus on $\gamma \in [1,d/2)$.
Applying Theorem \ref{theo:theineqGAW} we get
\begin{align}\label{align:trickfornonprob}
d_{\mathcal{H}^\gamma_1}(\hat{\mu},\mu^\star)\leq C\log\left(1+d_{\mathcal{H}^{d/2}_1}(\hat{\mu},\mu^\star)^{-1}\right)^{C_2} d_{\mathcal{H}^{d/2}_1}(\hat{\mu},\mu^\star)^{\frac{\beta+\gamma}{\beta+d/2}}.
        \end{align}

In the case, $d_{\mathcal{H}^{d/2}_1}(\hat{\mu},\mu^\star)^{\frac{\gamma}{\beta+d/2}}\leq 1/n$ we have 
\begin{align*}
    d_{\mathcal{H}^\gamma_1}(\hat{\mu},\mu^\star)&  \leq  Cn^{-1}\log\left(1+d_{\mathcal{H}^{d/2}_1}(\hat{\mu},\mu^\star)^{-1}\right)^{C_2} d_{\mathcal{H}^{d/2}_1}(\hat{\mu},\mu^\star)^{\frac{\beta}{\beta+d/2}}\\
    & \leq Cn^{-1}, 
\end{align*}
so we have the result.

In the case, $d_{\mathcal{H}^{d/2}_1}(\hat{\mu},\mu^\star)^{\frac{\beta+\gamma}{\beta+d/2}}> 1/n$ we have

\begin{align*}
    d_{\mathcal{H}^\gamma_1}(\hat{\mu},\mu^\star) 
    & \leq C\log\left(n\right)^{C_2} d_{\mathcal{H}^{d/2}_1}(\hat{\mu},\mu^\star)^{\frac{\beta+\gamma}{\beta+d/2}}.
\end{align*}

Let us denote $A$  the event that there exists $(g,\varphi)\in \mathcal{G}\times  \Phi$ satisfying the constraint \eqref{eq:estimator}. Then, using Jensen's inequality we get
\begin{align*}
    \mathbb{E}[d_{\mathcal{H}^{\gamma}_1}(\hat{\mu},\mu^\star)]=&\mathbb{E}[d_{\mathcal{H}^{\gamma}_1}(\hat{\mu},\mu^\star)(\mathds{1}_{A}+\mathds{1}_{A^c})]\\
    \leq & \frac{1}{n} + \mathbb{E}[C\log\left(n\right)^{C_2} d_{\mathcal{H}^{d/2}_1}(\hat{\mu},\mu^\star)^{\frac{\beta+\gamma}{\beta+d/2}}]
    \\
    \leq & \frac{1}{n} + C\log\left(n\right)^{C_2}\mathbb{E}[ d_{\mathcal{H}^{d/2}_1}(\hat{\mu},\mu^\star)]^{\frac{\beta+\gamma}{\beta+d/2}},
\end{align*}
so using Theorem \ref{theo:boundonddsur2} we finally get 
$$\mathbb{E}[d_{\mathcal{H}^{\gamma}_1}(\hat{\mu},\mu)]\leq C \log(n)^{C_2} n^{-\frac{\beta+\gamma}{2\beta+d}}.$$
\end{proof}

\end{document}